\documentclass[11pt]{amsart}

\usepackage{latexsym}
\usepackage{amssymb}
\usepackage{amsmath}

\newtheorem{theorem}{Theorem}[section]
\newtheorem{lemma}[theorem]{Lemma}
\newtheorem{proposition}[theorem]{Proposition}
\newtheorem{corollary}[theorem]{Corollary}
\newtheorem{definition}[theorem]{Definition}
\newtheorem{example}[theorem]{Example}
\newtheorem{remark}[theorem]{Remark}

\newcommand\Gr{\mathop{\rm Gr}}
\newcommand\supp{\mathop{\rm supp}}

\newcommand\esssup{\mathop{\rm esssup}}

\newcommand\id{\mathop{\rm id}}
\newcommand\tr{\mathop{\rm tr}}
\newcommand\nph{\varphi}

\newcommand\nul{\mathop{\rm null}}

\newcommand\loc{\mathop{\rm loc}}
\newcommand\vn{\mathop{\rm VN}}
\newcommand\clos{\mathop{\rm Clos}}
\newcommand\cb{\mathop{\rm cb}}

\newcommand\cg{\mathop{C_r^*(G)}}

\newcommand{\cl}[1]{\mathcal{#1}}
\newcommand{\bb}[1]{\mathbb{#1}}

\def\inv{^{-1}}

\begin{document}

\title{Sets of multiplicity and closable multipliers on group algebras}

\author{V.S. Shulman}
\address{Department of Mathematics, Vologda State Technical University, Vologda, Russia}
\email{shulman.victor80@gmail.com}

\author{I.G. Todorov}
\address{Pure Mathematics Research Centre, Queen's University Belfast, Belfast BT7 1NN, United Kingdom}
\email {i.todorov@qub.ac.uk}

\author{L. Turowska}
\address{Department of Mathematical Sciences,
Chalmers University of Technology and  the University of Gothenburg,
Gothenburg SE-412 96, Sweden}
\email{turowska@chalmers.se}


\thanks{ The third author was supported by the Swedish Research Council.}

\dedicatory{To the memory of William Arveson, with gratitude and admiration}

%\author{V. S Shulman, I. G. Todorov and L. Turowska}

\date{19 September 2013}

\begin{abstract}
We undertake a detailed study of the sets of multiplicity in a second countable locally
compact group $G$ and their operator versions.
We establish a symbolic calculus for
normal completely bounded maps from the space $\cl B(L^2(G))$ of bounded
linear operators on $L^2(G)$ into the von Neumann algebra $\vn(G)$ of $G$ and
use it to show that a closed subset $E\subseteq G$
is a set of multiplicity if and only if the set $E^* = \{(s,t)\in G\times G : ts^{-1}\in E\}$ is a
set of operator multiplicity. Analogous results are established for $M_1$-sets and $M_0$-sets.
We show that the property of being a set of multiplicity is preserved under various operations,
including taking direct products, and establish an Inverse Image Theorem for
such sets. We characterise the sets of finite width that are also sets of operator multiplicity,
and show that every compact operator supported on a set of finite width
can be approximated by sums of rank one operators supported on the same set.
We show that, if $G$ satisfies a mild approximation condition,
pointwise multiplication by a given
measurable function $\psi : G\to \bb{C}$ defines a closable multiplier
on the reduced C*-algebra $C_r^*(G)$ of $G$ if and only if
Schur multiplication by the function
$N(\psi) : G\times G\to \bb{C}$, given by
$N(\psi)(s,t) = \psi(ts^{-1})$, is a closable operator when viewed as a densely defined linear map
on the space of compact operators on $L^2(G)$.
Similar results are obtained for multipliers on $\vn(G)$.
\end{abstract}

\maketitle

\tableofcontents

\section{Introduction}\label{s_intro}

The connections between Harmonic Analysis and the Theory of Operator Algebras have
a long and illustrious history. With his pivotal paper \cite{arveson}, W. B. Arveson opened up
a new avenue in that direction by introducing the notion of operator synthesis.
The relation between operator synthesis and spectral synthesis for
locally compact groups was explored in detail in \cite{f}, \cite{lt}, \cite{st}, \cite{et} and \cite{et2}, among others.
In this connection,  J. Froelich \cite{f} studied the
question of when the operator algebra associated with a commutative subspace lattice contains a non-zero compact operator. For any compact abelian group $G$ and a closed subset $E\subseteq G$,
he constructed  a commutative subspace lattice $\cl L_E$, such that the corresponding operator algebra
contains a non-zero compact operator if and only if $E$ is a set of multiplicity in the
sense of (commutative) Harmonic Analysis.

Recently, we observed \cite{stt} a connection between sets of multiplicity and
the closability of linear transformations that are a natural unbounded analogue of Schur multipliers.
This served as a motivation for our present study of sets of multiplicity in the general setting of
locally compact groups and their connection with
closable multipliers on group algebras.

Sets of multiplicity for the group of the circle arose in connection with the
problem of uniqueness of trigonometric series and have been
extensively studied (see \cite{gmcgehee}).
In a general locally compact group $G$,
sets of uniqueness (or, equivalently, of non-multiplicity)
were introduced by
M. Bo\.{z}ejko in \cite{bozejko_pams} as those closed subsets $E\subseteq G$ which
do not support non-zero elements
of the reduced C*-algebra $C^*_r(G)$ of $G$.

An operator counterpart of sets of multiplicity was introduced in \cite{stt}.
On the operator level, as well as on the level of locally compact groups,
two classes of sets of multiplicity have been mostly examined:
(operator) $M$-sets and (operator) $M_1$-sets.
Here we introduce the class of operator $M_0$-sets and
show, in Section \ref{s_sm}, that a closed subset $E$ of a second countable locally compact group
$G$ is an $M$-set (resp. $M_1$-set, $M_0$-set) if and only if the set
$E^* = \{(s,t) : ts^{-1}\in E\} \subseteq G\times G$ is an operator $M$-set
(resp. operator $M_1$-set, operator $M_0$-set). These results should be compared to
the result established in \cite{f}, \cite{lt} and \cite{st} stating that $E$ is a set of local spectral synthesis
if and only if $E^*$ is a set of operator synthesis.
The main technical tool we use here is a symbolic calculus
for weak* continuous completely bounded maps from
the algebra $\cl B(L^2(G))$ of bounded operators on $L^2(G)$ into
the von Neumann algebra $\vn(G)$ of $G$ (see Theorem \ref{p_maps}).
A significant role in our approach is played by
a locally compact version of the uniform Roe algebra which was introduced for
discrete groups in \cite{roe} and has been studied in various contexts.

In Section 4, we show that the property of being a set of
(operator) multiplicity is preserved under some natural operations.
These include direct products and a certain type of generalised union.
As a corollary, we recover M. Bo{\.z}ejko's result \cite{bozejko_pams} that
every countable closed set in a non-discrete locally compact group is
a set of uniqueness.
We also establish an Inverse Image Theorem for sets of operator multiplicity
(see Theorem \ref{th_invim}). En route, we give an affirmative answer of a question of
J. Froelich \cite{f} concerning the validity of a tensor product formula for masa-bimodules
(see Theorem \ref{l_tmin}).

In Section \ref{s_fw}, we examine sets of finite width. This class of sets has played a fundamental role
in the field since their introduction in \cite{arveson}
(see \cite{et}, \cite{et2}, \cite{st1} and the references therein).
We characterise the sets of finite width that are also
sets of operator multiplicity, and show that, in general, every
compact operator supported on a set of finite width is the norm limit of sums of rank one
operators supported on this set.

It is well-known that a measurable function
$\psi : G\to{\mathbb C}$ is a Herz-Schur multiplier
precisely when the function $N(\psi)$ given by $N(\psi)(s,t):=\psi(ts^{-1})$, is a
Schur multiplier on $G\times G$ \cite{bf} (see also \cite{j} and \cite{spronk}).
In Section \ref{s_cmga}, we establish a \lq\lq closable'' version of this result,
showing that for groups $G$ satisfying a certain approximation property,
$\psi$ is a closable multiplier on $C_r^*(G)$ if and only if $N(\psi)$
is a closable multiplier in the sense of \cite{stt}.
We present various examples of closable and non-closable multipliers.

In Section \ref{s_vna}, we discuss multiplier maps on the group von Neumann algebra $\vn(G)$.
We introduce the notion of a weak* closable operator, which is suitable
for the setting of dual Banach spaces, such as $\vn(G)$.
We show that a continuous function $\psi$ is a weak* closable multiplier
if and only if $N(\psi)$ is a local Schur multiplier \cite{stt},
which occurs precisely when $\psi$ belongs locally to the Fourier algebra $A(G)$.
Weak** closable multipliers on $C_r^*(G)$ \cite{stt} (see Section \ref{sub_cl})
are shown to form a proper subset of the class of
weak* closable multipliers, which in turn form a proper subset of the class of closable multipliers.

\section{Preliminaries}\label{s_prel}

In this section, we collect some definitions and results that will be needed in the sequel.

\subsection{Closable operators}\label{sub_cl}

Let $\cl X$ and $\cl Y$ be Banach spaces and $T :
D(T)\rightarrow \cl Y$ be a linear operator, where
the domain $D(T)$ of $T$ is a
dense linear subspace of $\cl X$. The operator $T$ is called
\emph{closable} if the closure $\overline{\Gr T}$ of its graph
$$\Gr T = \{(x,Tx) : x\in D(T)\} \subseteq \cl X \oplus \cl Y$$
is the graph of a linear operator. Equivalently, $T$ is
closable if $(x_k)_{k\in \bb{N}}\subseteq D(T)$, $y\in \cl Y$,
$\|x_k\|\rightarrow_{k\to \infty} 0$ and $\|T(x_k) - y\|\rightarrow_{k\to \infty} 0$ imply
that $y = 0$.
The operator $T$ is called \emph{weak**
closable} \cite{stt} if the weak* closure $\overline{\Gr T}^{w^*}$ of
$\Gr T$ in $\cl X^{**}\oplus \cl Y^{**}$ is the
graph of a linear operator. Equivalently, $T$ is weak** closable if
whenever $(x_j)_{j\in J}\subseteq D(T)$ is a net,
$y \in \cl Y^{**}$, $x_j\stackrel{w^*}{\rightarrow}_{j\in J} 0$ and
$T(x_j) \stackrel{w^*}{\rightarrow}_{j\in J} y$, we have that $y = 0$.
We note that in \cite{stt} weak** closable operators were called
weak* closable. We have chosen to alter our terminology
since we feel that the term \lq\lq weak* closable'' is better suited for
the notion introduced and studied in Section \ref{s_vna} of the present paper.

The domain of the adjoint operator of $T$ is the subspace
$$D(T^*) = \{g\in \cl Y^* : \exists f\in \cl X^* \mbox{ such that }
g(Tx) = f(x) \mbox{ for all } x\in D(T)\}$$
and the adjoint of $T$ is the operator $T^* : D(T^*) \to \cl X^*$
defined by letting $T^*(g) = f$, where $f$ is the
functional associated with $g$ in the definition of $D(T^*)$.

In the following proposition, which was stated in \cite{stt}, the equivalence
(iii)$\Leftrightarrow$(iv) is well-known (see, for example,
\cite[Chapter III, Section 5]{kato}), while the other implications can be
proved easily.

\begin{proposition}\label{wstarc}
Let $\cl X$ and $\cl Y$ be Banach spaces, $D(T)\subseteq \cl X$,
$T : D(T)\rightarrow \cl Y$ be a densely defined linear operator and
set $\cl D = D(T^*)$. Consider the
following conditions:

(i) \ \ $T$ is weak** closable;

(ii) \ $\overline{\cl D}^{\|\cdot\|} = \cl Y^*$;

(iii) $\overline{\cl D}^{w^*} = \cl Y^*$;

(iv) \ $T$ is closable.

\noindent Then (i)$\Longleftrightarrow$(ii)$\Longrightarrow$(iii)$\Longleftrightarrow$(iv).
\end{proposition}

\subsection{Locally compact groups}

If $H$, $H_1$ and $H_2$ are Hilbert spaces, we denote by
$\cl B(H_1,H_2)$ the space of all bounded linear operators from $H_1$ to $H_2$,
and set $\cl B(H) = \cl B(H,H)$.
Let $G$ be a locally compact group.
Left Haar measure on $G$ will be denoted by $m_G$ or $m$ and
integration with respect to $m_G$
along the variable $s$ will be denoted by $ds$.
We denote by $L^p(G)$, $p = 1,2,\infty$, the corresponding
Lebesgue spaces associated with $m_G$.
For a function $\xi : G\rightarrow \bb{C}$, we set as
customary $\check{\xi}(s) = \xi(s^{-1})$, $s\in G$.
Let $\lambda : G\to \cl B(L^2(G))$ be the left regular representation of $G$,
that is, $\lambda_s f(t) = f(s^{-1}t)$, $f\in L^2(G)$, $s,t\in G$, and
$M(G)$ be the \emph{measure algebra} of $G$, consisting by definition
of all bounded complex Borel measures on $G$. 
We denote the variation of $\theta\in M(G)$ by $|\theta|$ and let $\|\theta\| = |\theta|(G)$.
The support of a measure $\theta\in M(G)$ is the (closed) subset
$$\supp\theta = \left(\cup\{U\subseteq G : U \mbox{ open, } |\theta|(U) = 0\}\right)^c;$$
it is the smallest closed subset $E$ of $G$ with the property that if $U\subseteq E^c$
is a Borel set then $\theta(U) = 0$.
For a closed set $E\subseteq G$, let $M(E)$ be the set of all measures
$\theta$ in $M(G)$ with $\supp \theta \subseteq E$.
If $\theta \in M(G)$  then
the operator $\lambda(\theta)$ of convolution by $\theta$ is given by
$\lambda(\theta)(f)(t) = \int_G f(s^{-1}t) d\theta(s)$;
the map $\lambda : M(G)\to \cl B(L^2(G))$ is a
representation of $M(G)$ of $L^2(G)$. Since $L^1(G)$
is a Banach subalgebra of $M(G)$, the restriction of $\lambda$
to $L^1(G)$ is a representation of $L^1(G)$; we have
$$\lambda(f) g(t) = f\ast g(t) = \int f(s) g(s^{-1}t)ds, \ \ f\in L^1(G),  g\in L^2(G), t\in G.$$
The \emph{Fourier algebra}
$A(G)$ of $G$ \cite{eymard} is the algebra of coefficients of $\lambda$,
that is, the algebra of
functions of the form $s\rightarrow (\lambda_s\xi,\eta)$, for
$\xi,\eta\in L^2(G)$.
The \emph{Fourier-Stieltjes algebra} $B(G)$ of $G$
\cite{eymard} is, on the other hand, the algebra of coefficients of all continuous
unitary representations of $G$ acting on some Hilbert space, that is,
the algebra of all functions of the form $s\rightarrow (\pi(s)\xi,\eta)$, where
$\pi : G\rightarrow \cl B(H)$ is a continuous unitary representation, and
$\xi, \eta\in H$.
We denote by $C_r^*(G)$ the \emph{reduced C*-algebra} of $G$, that is,
the closure of $\lambda(L^1(G))$ in the operator norm. We let
$\vn(G) = \overline{C_r^*(G)}^{w^*}$ be the \emph{von
Neumann algebra} of $G$, and $C^*(G)$ be the \emph{full C*-algebra} of
$G$. It is known \cite{eymard} that $A(G)$ is a semisimple, regular, commutative
Banach algebra with spectrum $G$,
which can be identified with the predual $\vn(G)_*$
of $\vn(G)$ via the pairing $\langle u,T\rangle=(T\xi,\eta)$, where $u\in A(G)$ is given by
$u(s) = (\lambda_s\xi,\eta)$.
If $T\in \vn(G)$ and $u\in A(G)$, the operator $u\cdot T \in \vn(G)$ is
given by the relations
$\langle u\cdot T,v\rangle = \langle T,uv\rangle$, $v\in A(G)$.
The map $(u,T)\mapsto u\cdot T$ turns $\vn(G)$ into a Banach $A(G)$-module.

Let
$$MA(G) = \{v : G\to\bb{C} \ : \ vu\in A(G), \mbox{ for all } u\in A(G)\}$$
be the multiplier algebra of $A(G)$.
For each $v\in MA(G)$, the map $u\mapsto vu$ on $A(G)$ is bounded; its norm will be denoted by
$\|v\|_{MA(G)}$.
As usual, let $M^{\cb}A(G)$ be the subalgebra of $MA(G)$ consisting
of those $v$ for which the map $u \mapsto vu$ on $A(G)$ is completely bounded
\cite{de-canniere-haagerup}. We refer the reader to \cite{paulsen} and \cite{pisier}
for the basic of Operator Space Theory and completely bounded maps.

We denote by $C_0(G)$ the space of all continuous functions on $G$ vanishing at
infinity. The dual of $C_0(G)$ can be canonically identified with $M(G)$; the
duality between the two spaces will be denoted by $\langle \cdot,\cdot\rangle$.
Note that $A(G)\subseteq C_0(G)$ and that the adjoint
of this inclusion gives rise to the inclusion $\lambda(M(G))\subseteq \vn(G)$.
We refer the reader to \cite{eymard} for more details
about the notions discussed above.

If $J\subseteq A(G)$ is an ideal, let
$$\nul J = \{s\in G : u(s) = 0 \mbox{ for all } u\in J\}.$$
On the other hand, for a closed set  $E\subseteq G$, let
$$I(E)=\{f\in A(G) : f(s)=0, s\in E\},$$
$$J_0(E)=\{f\in A(G) : f \text{ vanishes on a nbhd of }
E\}$$
and
$J(E)=\overline{J_0(E)}.$
We have that
$\nul J(E) = \nul I(E) = E$ and that if
$J\subseteq A(G)$ is a closed ideal with $\nul J = E$, then
$J(E)\subseteq J\subseteq I(E)$.
The support $\supp(T)$ of an operator $T\in \vn(G)$ is given by
$$\supp(T)=\{ t\in G: u\cdot T\ne 0\text{  whenever } u\in A(G) \mbox{ and } u(t)\ne 0\}.$$
It is known (see \cite{eymard}) that the annihilator $J(E)^\perp$
of $J(E)$ in $\vn(G)$ coincides with the space of all
operators $T\in \vn(G)$ with $\supp(T)\subseteq E$.

\subsection{Masa-bimodules}

We fix, throughout the paper, standard measure spaces $(X,\mu)$ and $(Y,\nu)$;
this means that $\mu$ and $\nu$ are regular Borel measures
with respect to some
complete metrisable topologies (henceforth called admissible topologies)
on $X$ and $Y$, respectively.
A subset of
$X\times Y$ will be called a {\it rectangle}
if it is of the form $\alpha\times\beta$, where
$\alpha\subseteq X$ and $\beta\subseteq Y$ are measurable.
We equip $X\times Y$ with the $\sigma$-algebra generated by all
rectangles and denote by $\mu\times\nu$ the product measure.
A subset $E\subseteq X\times Y$ is called {\it marginally null} if
$E\subseteq (X_0\times Y)\cup(X\times Y_0)$, where $\mu(X_0) =
\nu(Y_0) = 0$. We call two subsets $E,F\subseteq X\times Y$ {\it
marginally equivalent} (and write $E\simeq F$) if their symmetric
difference is marginally null.

A subset $E$ of $X\times Y$ is called {\it $\omega$-open} if it is
marginally equivalent to the union of a countable set of rectangles.
The complements of $\omega$-open sets are called {\it
$\omega$-closed}. It is clear that the class of all $\omega$-open
(resp. $\omega$-closed) sets is closed under countable unions (resp.
intersections) and finite intersections (resp. unions).
Let ${\mathfrak B}(X\times Y)$ be the space of all  measurable
complex valued functions defined on the measure space $(X\times
Y,\mu\times \nu)$.
We say that two functions $\varphi$, $\psi\in {\mathfrak B}(X\times Y)$ are
\emph{equivalent}, and write $\varphi\sim \psi$, if the set
$D=\{(x,y)\in X\times Y:\psi(x,y)\ne\varphi(x,y)\}$ is null with respect to
$\mu\times\nu$. If $D$ is marginally null then we say that
$\varphi$ and $\psi$ coincide marginally almost
everywhere or that they are \emph{marginally equivalent}, and write $\varphi\simeq\psi$.

The following lemma was proved in \cite{eks}.

\begin{lemma}\label{e-compactness}
Suppose that compact admissible topologies can be chosen on $X$ and $Y$ and that
$\mu$ and $\nu$ are finite.
Let $E \subseteq \cup_{n=1}^{\infty} \gamma_n$ where $E$ is $\omega$-closed and $\gamma_n$ is
$\omega$-open, $n\in \bb{N}$. Then for each $\varepsilon>0$ there are subsets $X_{\varepsilon}\subseteq X$, $Y_{\varepsilon}\subseteq Y$ such that $\mu(X\setminus X_\varepsilon)<\varepsilon$,  $\nu(Y\setminus Y_\varepsilon)<\varepsilon$ and $E\cap (X_\varepsilon\times Y_\varepsilon)$ is contained in the union of finitely many of the subsets $\gamma_n$, $n\in \bb{N}$.
\end{lemma}
%{\bf lemma is proved for compact $X$, $Y$ but works for locally-compact. check. I. I do not see it.}

For Hilbert spaces  $H_1$ and $H_2$, we denote
by $\cl K(H_1,H_2)$ (resp. $\cl C_1(H_1,H_2)$,
$\cl C_2(H_1,H_2)$) the space of compact (resp. nuclear,
Hilbert-Schmidt) operators in $\cl B(H_1,H_2)$.
We often write $\cl K = \cl K(H_1,H_2)$.
Throughout the paper, we let
$H_1=L^2(X,\mu)$ and $H_2=L^2(Y,\nu)$.
The operator norm of $T\in \cl B(H_1,H_2)$ is denoted by $\|T\|$.
The space $\cl C_1(H_2,H_1)$ (resp.
$\cl B(H_1,H_2)$) can be naturally identified with the Banach space
dual of $\cl K(H_1,H_2)$ (resp. $\cl C_1(H_2,H_1)$), the duality
being given by the map $(T,S)\mapsto \langle T,S\rangle
\stackrel{def}{=}\tr(TS)$. Here $\tr A$ denotes the trace of a
nuclear operator $A$.

The space $L^2(Y\times X)$ will be identified with
$\cl C_2(H_1,H_2)$ via the map sending an element $k\in L^2(Y\times X)$
to the integral operator $T_k$ given by
$T_k \xi (y) = \int_{X} k(y,x)\xi(x) d\mu(x)$,
$\xi\in H_1$, $y\in Y$.
In a similar fashion, $\cl C_1(H_2,H_1)$ will be identified with
the space $\Gamma(X,Y)$ of all (marginal equivalence classes of)
functions $h : X\times Y\to{\mathbb C}$
which admit a representation
$$h(x,y)=\sum_{i=1}^{\infty} f_i(x)g_i(y),$$
where $f_i\in H_1$, $g_i\in H_2$, $i\in \bb{N}$,
$\sum_{i=1}^{\infty}\|f_i\|^2_{2} < \infty$ and
$\sum_{i=1}^{\infty}\|g_i\|^2_{2}<\infty$. Equivalently,
$\Gamma(X,Y)$ can be defined as the projective tensor product
$H_1\hat\otimes H_2$; we write $\|h\|_{\Gamma}$ for
the projective norm of $h\in \Gamma(X,Y)$.
The duality between $\cl B(H_1,H_2)$ and $\Gamma(X,Y)$ is given by
$$\langle T,f\otimes g\rangle = (Tf,\bar g),$$
for $T\in \cl B(H_1,H_2)$, $f\in L^2(X,\mu)$ and $g\in L^2(Y,\nu)$.

If $f\in L^{\infty}(X,\mu)$, let $M_f\in \cl B(H_1)$ be the
operator on $H_1$ of multiplication by $f$. The collection
$\{M_f : f\in L^{\infty}(X,\mu)\}$ is a maximal abelian selfadjoint
algebra (for short, masa) on $H_1$.
If $\alpha\subseteq X$ is measurable, we write $P(\alpha) =
M_{\chi_{\alpha}}$ for the multiplication by the characteristic
function of the set $\alpha$. The same notation will be used for
$H_2$. A subspace $\cl W\subseteq \cl B(H_1,H_2)$ will
be called a \emph{masa-bimodule} if $M_{\psi}TM_{\nph} \in \cl  W$
for all $T\in \cl W$, $\nph\in L^{\infty}(X,\mu)$ and $\psi\in
L^{\infty}(Y,\nu)$.

We say that an $\omega$-closed subset $\kappa\subseteq X\times Y$
{\it supports} an operator $T\in \cl B(H_1,H_2)$
(or that $T$ {\it is supported on} $\kappa$)
if $P(\beta)TP(\alpha) = 0$ whenever
$(\alpha\times\beta)\cap \kappa \simeq \emptyset$.
For any subset ${\mathcal M}\subseteq \cl B(H_1,H_2)$, there exists a smallest
(up to marginal equivalence) $\omega$-closed set
$\supp\cl M$ which supports every operator $T\in{\mathcal M}$ \cite{eks}.
By \cite{arveson} and \cite{st1}, for any $\omega$-closed set $\kappa$ there exists a
smallest (resp. largest) weak* closed masa-bimodule ${\mathfrak
M}_{\min}(\kappa)$ (resp. ${\mathfrak M}_{\max}(\kappa)$) with
support $\kappa$, in the sense that if $\mathfrak M\subseteq \cl
B(H_1,H_2)$ is a weak* closed masa-bimodule with $\text{supp }{\mathfrak M}
= \kappa$ then ${\mathfrak M}_{\min}(\kappa)\subseteq {\mathfrak
M}\subseteq {\mathfrak M}_{\max}(\kappa)$.

Let
$$\Phi(\kappa)=\{h\in\Gamma(X,Y): h\chi_{\kappa} \simeq 0\}$$
and
$$\Psi(\kappa)=\overline{\{h\in \Gamma(X,Y):
h \text{ vanishes on an $\omega$-open nbhd of }\kappa\}}^{\|\cdot\|_{\Gamma}}.$$
By \cite[Theorem~4.3, 4.4]{st1}, ${\mathfrak
M}_{\min}(\kappa) = \Phi(\kappa)^{\perp}$ and ${\mathfrak M}_{\max}(\kappa)=\Psi(\kappa)^\perp$.

Let $\sigma$ be a complex measure of finite total variation,
defined on the product $\sigma$-algebra
$\cl F$ of $X\times Y$.
We let $|\sigma|$ denote the variation of $\sigma$; thus, for a subset $E\in \cl F$,
the quantity $|\sigma|(E)$ equals the total variation of $\sigma$ on the set $E$.
We let $|\sigma|_X$ be the $X$-marginal measure of $|\sigma|$, that is, the measure on $X$ given by $|\sigma|_X(\alpha) = |\sigma|(\alpha\times Y)$.
We define $|\sigma|_Y$ similarly by setting $|\sigma|_Y(\beta) = |\sigma|(X\times\beta)$.
A complex measure $\sigma$ on $\cl F$ will be called an \emph{Arveson measure} if
$\sigma$ has finite total variation and
there exists a constant $c > 0$ such that
\begin{equation}\label{eq_carv}
|\sigma|_X \leq c \mu \ \ \mbox{ and } \ \ |\sigma|_Y \leq c \nu.
\end{equation}
We denote by
$\bb{A}(X,Y)$ the set of all Arveson measures on $X\times Y$ and let
$\|\sigma\|_{\bb{A}}$ be the smallest constant $c$ which satisfies the inequalities (\ref{eq_carv}).
We note that if $\sigma\in \bb{A}(X,Y)$ then $|\sigma|\in \bb{A}(X,Y)$ as well.

\subsection{Schur multipliers}
If $\nph$ is a function defined on a measure space $(Z,\theta)$, and
$\cl E$ is a space of measurable functions on $Z$, we
write $\nph\in^{\theta}\cl E$ when there exists a function $\psi\in \cl E$
such that $\nph$ and $\psi$ differ on a $\theta$-null set.
Let
$$J_{\nph}^{\cl E} = \{h\in \cl E : \nph h\in^{\theta} \cl E\}.$$
For   $\varphi\in {\mathfrak B}(X\times Y)$, the function
$\hat{\nph} : Y\times X\to \bb{C}$ is given by
$\hat{\nph}(y,x) = \nph(x,y)$, $x\in X$, $y\in Y$.
We set
$D(S_{\nph})=J_{\hat{\nph}}^{L^2(Y\times X)}$. Identifying $L^2(Y\times X)$
with $\cl C_2(H_1,H_2)\subseteq\cl
K(H_1,H_2)$, define  $S_{\nph} : D(S_{\nph}) \rightarrow \cl
K(H_1,H_2)$ to be the mapping given by $S_\nph(T_k)= T_{\hat{\nph} k}$.
We say that $\nph\in {\mathfrak B}(X\times Y)$ is a
{\it closable multiplier} (resp. {\it weak** closable multiplier})
\cite{stt} if the map $S_\nph$ is closable (resp. weak** closable)
when viewed as a densely defined linear operator on $\cl
K(H_1,H_2)$. If $S_{\nph}$ is moreover bounded in the operator norm,
$\nph$ is called a \emph{Schur multiplier}. If $\varphi$ is a Schur multiplier then the mapping
$S_\varphi$ extends by continuity to a (bounded) mapping on $\cl
K(H_1,H_2)$.
After taking its second dual, one obtains a bounded weak* continuous linear transformation on
$\cl B(H_1,H_2)$ which will also be denoted by $S_\varphi$.
The map $S_{\nph}$ is automatically
completely bounded and its completely bounded norm is still
equal to $\|\nph\|_{\frak{S}}$
(the reader is referred to \cite{paulsen} and
\cite{pisier} for the basics of Operator Space Theory, which will be
used throughout the paper).
By a result of V. V. Peller \cite{peller} (see also \cite{kp} and \cite{spronk}),
a function $\nph \in {\mathfrak B}(X\times Y)$ is a Schur multipliers
if and only if there exist sequences $(a_k)_{k\in \bb{N}}\subseteq L^{\infty}(X,\mu)$ and
$(b_k)_{k\in \bb{N}}\subseteq L^{\infty}(Y,\nu)$ with
$\esssup_{x\in X} \sum_{k=1}^{\infty} |a_k(x)|^2 < \infty$ and
$\esssup_{y\in Y} \sum_{k=1}^{\infty} |b_k(y)|^2 < \infty$
such that
$$\nph(x,y) = \sum_{k=1}^{\infty} a_k(x) b_k(y), \ \ \ \mbox{a.e. } (x,y)\in X\times Y.$$
In this case, $S_{\nph}(T) = \sum_{k=1}^{\infty} M_{b_k} T M_{a_k}$, $T\in \cl B(H_1,H_2)$.

Let $\frak{S}(X,Y)$ be the set of all Schur multipliers (we will
also write $\frak{S}(X\times Y)$ in the place of $\frak{S}(X,Y)$ if there is no risk of confusion).
If $\nph\in \frak{S}(X,Y)$, we set $\|\nph\|_{\frak{S}} = \|S_{\nph}\|$.
By \cite{peller},
$$\frak{S}(X,Y) = \{\nph \in
L^{\infty}(X\times Y) : \nph h \in^{\mu\times\nu} \Gamma(X,Y), \ \forall \ h\in \Gamma(X,Y)\}.$$
If $\nph\in \frak{S}(X,Y)$, let $m_{\nph} : \Gamma(X,Y) \to \Gamma(X,Y)$ be the
mapping given by $m_{\nph}(h) = \nph h$, $h\in \Gamma(X,Y)$;
then the adjoint of $m_{\nph}$ coincides with $S_{\nph}$.

Let $G$ be a locally compact group. The
map $P : \Gamma(G,G)\rightarrow A(G)$ given by
\begin{equation}\label{eq_p}
P(f\otimes g)(t) = \langle \lambda_t, f\otimes g\rangle =
(\lambda_t f,\overline{g}) = \int_G f(t^{-1}s) g(s) ds = g\ast \check{f}(t)
\end{equation}
is a contractive surjection.
The next lemma will be used repeatedly.

\begin{lemma}\label{l_for}
If $h\in \Gamma(G,G)$ then
\begin{equation}\label{eq_P}
P(h)(t) = \int_G h(t^{-1}s,s)ds, \ \ \  t\in G.
\end{equation}
\end{lemma}
\begin{proof}
Identity (\ref{eq_P}) is a direct consequence of (\ref{eq_p}) if $h$ is a finite sum of
elementary tensors. Let $h = \sum_{i=1}^{\infty} f_i\otimes g_i \in \Gamma(G,G)$,
where $\sum_{i=1}^{\infty} \|f_i\|_2^2 < \infty$ and $\sum_{i=1}^{\infty} \|g_i\|_2^2 < \infty$, and
let $h_n$ be the $n$th partial sum of this series.
By the continuity of $P$, $\|P(h_n) - P(h)\|\to 0$ in $A(G)$; since
$\|\cdot\|_{\infty}$ is dominated by the norm of $A(G)$, we
conclude that $P(h_n)(t) \to P(h)(t)$ for every $t\in G$.

By \cite[Lemma 2.1]{st1}, there exists a subsequence $(h_{n_k})_{k\in \bb{N}}$
of $(h_n)_{n\in \bb{N}}$ such that $h_{n_k}\rightarrow h$ marginally almost everywhere.
It follows that, for every $t\in G$, one has $h_{n_k}(t^{-1}s,s)\rightarrow h(t^{-1}s,s)$ for
almost all $s\in G$. By \cite[(4.3)]{lt}, the function $s\rightarrow \sum_{i=1}^{\infty} |f_i(t^{-1}s)| |g_i(s)|$
is integrable, and hence an application of the Lebesgue Dominated Convergence Theorem shows that
$\int_G h_{n_k}(t^{-1}s,s)ds \rightarrow _{k\rightarrow \infty} \int_G h(t^{-1}s,s)ds$, for every $t\in G$.
The proof is complete.
\end{proof}

For a function $f : G\rightarrow \bb{C}$, let $N(f) : G\times G\rightarrow \bb{C}$
be the function given by
\begin{equation}\label{eq_mapn}
N(f)(s,t) = f(ts^{-1}), \ \ \ \ s,t\in G.
\end{equation}
Note that in \cite{lt} and \cite{st}, the map $N'$ given by $N'(f)(s,t) = f(st^{-1})$
was used instead of $N$, but the results established in these papers remain valid
with the current definition as well.
It follows from
\cite{bf} (see also \cite{j} and \cite{spronk}) that $N$ maps $M^{\cb}A(G)$ isometrically into
$\frak{S}(G,G)$. Note that, if $G$ is compact, then $\Gamma(G,G)$ contains the
constant functions and hence $\frak{S}(G,G) \subseteq \Gamma(G,G)$; thus,
in this case $N$ maps $A(G)$ into $\Gamma(G,G)$.

\section{Sets of multiplicity and their operator versions}\label{s_sm}

\subsection{ Sets of multiplicity in arbitrary locally compact groups}

In this section, we study sets of multiplicity and their operator versions,
and examine the relations between them.
We start by recalling
the classical notion, where
$G={\mathbb T}$ is the group of the circle;
in this case,
$A({\mathbb T})=\{\sum_{n\in{\mathbb Z}}c_ne^{int}: \sum_{n\in\mathbb Z}|c_n|<\infty\}\simeq \ell^1(\mathbb Z)$. The space of \emph{pseudo-measures}
$PM({\mathbb T})=A({\mathbb T})^*$ can be identified with $\ell^\infty(\mathbb Z)$
via Fourier transform $F\mapsto (\hat F(n))_{n\in\mathbb Z}$,
and the space of \emph{pseudo-functions}
$PF(\mathbb T)=\{F\in PM(\mathbb T):\hat F(n)\to 0, \text{ as }n\to\infty\}$
is $*$-isomorphic to $C^*({\mathbb T})=C_r^*({\mathbb T})$.
Note that there is a canonical embedding $M(\bb{T})\subseteq PM(\bb{T})$
arising from the inclusion $A(\bb{T})\subseteq C(\bb{T})$.

If $E$ is a closed subset of $\mathbb T$,
let $PM(E)$ denote the space of all pseudo-measures supported on $E$,
$M(E)$ the space of measures $\mu\in M(G)$ with $\supp\mu \subseteq E$,
and $N(E)$ the weak* closure of $M(E)$.
For an ideal $J\subseteq A(G)$, let $J^\perp$ denote the annihilator of $J$ in $PM(\bb{T})$;
then $PM(E)=J(E)^\perp$ and $N(E)=I(E)^\perp$ (see, {\it e.g.}, \cite{gmcgehee}).

A closed set $E\subseteq \mathbb T$ is called an $M$-set if $PM(E)\cap PF(\mathbb T)\ne\{0\}$,
an $M_1$-set if $N(E)\cap PF(\mathbb T)\ne\{0\}$, and an
$M_0$-set if $M(E)\cap PF(\mathbb T)\ne\{0\}$.
The closed sets that are not $M$-sets are called sets of uniqueness.

A definition of sets of multiplicity for locally compact abelian groups
was proposed by I. Piatetski-Shapiro (see \cite[p.190]{herz}).
In \cite{bozejko_pams}, M. Bo$\dot{\rm z}$ejko
introduced sets of uniqueness in general locally compact groups.
Here we extend his definition to include versions of $M_1$-sets and of $M_0$-sets.

\begin{definition}\label{d_msets}
A closed subset $E\subseteq G$ will be called

(i) \ \  an \emph{$M$-set} if $J(E)^{\perp}\cap C_r^*(G) \neq \{0\}$;

(ii) \ an \emph{$M_1$-set} if $I(E)^{\perp}\cap C_r^*(G) \neq \{0\}$;

(iii) an \emph{$M_0$-set} if $\lambda(M(E))\cap C_r^*(G)\neq\{0\}$.

The set $E$ will be called a \emph{$U$-set} (resp.
a \emph{$U_1$-set}, a \emph{$U_0$-set})
if it is not an $M$-set (resp. an $M_1$-set, an $M_0$-set).
\end{definition}

\begin{remark}\label{remark1}
\rm
\noindent  {\bf (i)}
Since $\lambda(M(E))\subseteq I(E)^\perp\subseteq J(E)^\perp$,
every $M_0$-set is an $M_1$-set, and every $M_1$-set is an $M$-set.
It is known that these three classes of sets are distinct, see \cite{gmcgehee}.

{\bf (ii)}
If $G$ is amenable then $C_r^*(G)$ is $*$-isomorphic to $C^*(G)$
and it is a direct consequence of the definition that
a closed set $E\subseteq G$ is an $M$-set (resp. an $M_1$-set)
if and only if $J(E)$ (resp. $I(E)$) is not weak* dense
in $B(G)$.

{\bf (iii)} Measures $\mu\in M(G)$ satisfying the condition $\lambda(\mu)\in C_r^*(G)$ were studied in \cite{Blumlinger}
where the author characterised them in terms of their values on
certain Borel subsets of $G$. If $G$ is  compact or
abelian then this class of measures coincides with the
\emph{Rajchman measures} on $G$,
that is, the measures whose Fourier-Stieltjes coefficients
vanish at infinity (see \cite{Blumlinger}).
\end{remark}

We point out an easy source of examples of sets of multiplicity:

\begin{remark}\label{r_posme}
Every closed subset of positive Haar measure in a locally compact second countable group  is an
$M_0$-set.
\end{remark}
\begin{proof}
Let $E\subseteq G$ be a measurable subset of positive Haar measure
and $E_0\subseteq E$ be a compact set of positive Haar measure;
then $m(E_0) < \infty$. Let
$\theta$ be the measure given by $d\theta(x) = \chi_{E_0}(x) dm(x)$.
Clearly, $\supp\theta\subseteq E$ and $0\ne \lambda(\theta)=\lambda(\chi_{E_0})\in \cg$.
\end{proof}

\subsection{ Sets of operator multiplicity}

We next recall the notion of a pseudo-integral operator, which will
be essential for some of the subsequent results.
Note that \cite{stt}, given a family $\cl E$ of
$\omega$-open sets, there exists a minimal (with respect to
inclusion up to a marginally null set) $\omega$-open set $E$
which marginally contains every element from $\cl E$.
The set $E$ is called the \emph{$\omega$-union} of $E$
and denoted by $\cup_{\omega}\cl E$.
Recall that $(X,\mu)$ and $(Y,\nu)$ are standard measure spaces
and let $\sigma$ be an Arveson measure on $Y\times X$.
Denote by $\supp\sigma$ the $\omega$-closed subset of $Y\times X$
defined by
$$(\supp\sigma)^c =
\cup\mbox{}_{\omega}\{R\subseteq Y\times X : R \mbox{ is a rectangle such that }$$
$$\sigma(R') = 0 \mbox{ for each rectangle } R'\subseteq R\}.$$

\begin{proposition}\label{p_sm}
Let $\sigma\in \bb{A}(Y,X)$.

(i) \ The set $\supp\sigma$ is the smallest (up to marginal equivalence) $\omega$-closed
subset $E$ of $Y\times X$ such that
$\sigma(R) = 0$ for every rectangle $R\subseteq E^c$.

(ii) If $E\subseteq Y\times X$ is an $\omega$-closed set then
$\supp \sigma\subseteq E$ if and only if $|\sigma|(E^c) = 0$.
\end{proposition}
\begin{proof}
(i) Let $\cl R$ be the set of all rectangles $R\subseteq Y\times X$ such that
$\sigma(R') = 0$ for every rectangle $R'$ contained in $R$.
By \cite[Lemma 2.1]{stt}, $(\supp\sigma)^c \simeq \cup_{i=1}^{\infty} R_i$ for some
family $\{R_i\}_{i\in \bb{N}}\subseteq \cl R$.
Let $R\subseteq (\supp\sigma)^c$ be a rectangle. We will show that $\sigma(R) = 0$;
without loss of generality, we may assume that the measures $\mu$ and $\nu$
are finite. By Lemma \ref{e-compactness}, for every $n\in \bb{N}$ there exist
measurable subsets $X_n\subseteq X$ and $Y_n\subseteq Y$ such that
$\mu(X\setminus X_n) < 1/n$, $\nu(Y\setminus Y_n) < 1/n$ and
$R\cap (Y_n\times X_n)$ is contained in the union of a finite
subfamily of $\{R_i\}_{i\in \bb{N}}$. It follows that $\sigma(R\cap (Y_n\times X_n)) = 0$
for every $n$ and, since $\cup_{n=1}^{\infty} X_n$ and $\cup_{n=1}^{\infty} Y_n$ have
full measure, $\sigma(R) = 0$.

Suppose that $E$ is an $\omega$-closed set with the property that
$\sigma(R) = 0$ for every rectangle $R\subseteq E^c$.
By the definition of $\supp\sigma$, the set $E^c$ is marginally contained in $(\supp\sigma)^c$,
and hence $\supp\sigma\subseteq E$ up to marginal equivalence.

(ii) Suppose that $E^c \simeq \Omega = \cup_{i=1}^{\infty} R_i$,
where $R_i\subseteq Y\times X$ is a rectangle, $i\in \bb{N}$.
Assume, without loss of generality, that $R_i\cap R_j = \emptyset$ if $i\neq j$.
Fix $i\in \bb{N}$. By (i), if $R\subseteq R_i$ is a rectangle, then
$\sigma(R) = 0$. Since the product $\sigma$-algebra on $R_i$ is
generated by the rectangles contained in $R_i$, it follows that
$\sigma(F) = 0$ for every measurable (with respect to the product $\sigma$-algebra)
subset $F\subseteq R_i$.
Thus, if $F\subseteq \Omega$ is an arbitrary measurable subset then
$\sigma(F\cap R_i) = 0$ for each $i$; therefore, $\sigma(F) = 0$.

Now suppose that $F\subseteq E^c$ is a measurable subset. Then
$F\subseteq F'\cup F''$ as a disjoint union, where $F'\subseteq \Omega$
and $F''$ is marginally null. By the previous paragraph, $\sigma(F') = 0$,
while, since $\sigma$ is an Arveson measure, $\sigma(F'') = 0$. It follows that
$\sigma(F) = 0$. Thus, $|\sigma|(E^c) = 0$.

Conversely, if $|\sigma|(E^c) = 0$ then $\sigma(R) = 0$ for every measurable rectangle
contained in $E^c$. By (i), $\supp\sigma\subseteq E$.
\end{proof}

The first part of the following result was established
in \cite[Theorem 1.5.1]{arveson}; we include its full proof for completeness.

\begin{theorem}\label{th_ps}
Let $\sigma\in {\mathbb A}(Y,X)$. There exists a unique operator $T_{\sigma} : H_1\rightarrow H_2$ such that
$$(T_{\sigma}f,g) = \int_{Y\times X} f(x)\overline{g(y)} d\sigma(y,x), \ \ \  f\in H_1, g\in H_2.$$
Moreover, $\|T_{\sigma}\| \leq \|\sigma\|_{\bb{A}}$ and, for a given $\omega$-closed
subset $\kappa\subseteq X\times Y$,
the operator $T_{\sigma}$ is supported on $\kappa$ if and only if
$\supp\sigma\subseteq \hat{\kappa} \stackrel{def}{=} \{(y,x) : (x,y)\in \kappa\}$.
If $h\in \Gamma(X,Y)$ and $\sigma\in \bb{A}(Y,X)$ then
$\langle T_{\sigma}, h \rangle = \int_{Y\times X} \hat{h} d\sigma$.
\end{theorem}
\begin{proof}
Fix $\sigma\in \bb{A}(Y,X)$ and consider the sesqui-linear form $\phi : H_1\times H_2\rightarrow \bb{C}$
given by
$$\phi(f,g) = \int_{Y\times X} f(x)\overline{g(y)} d\sigma(y,x).$$
Note that $\phi$ is well-defined:
\begin{eqnarray*}
\left|\int_{Y\times X} f(x)\overline{g(y)} d\sigma(y,x)\right|^2 & \leq & \left(\int_{Y\times X} |f(x)||g(y)| d|\sigma|(y,x)\right)^2\\
&\leq & \int_{Y\times X} |f(x)|^2 d|\sigma|(y,x) \int_{Y\times X} |g(y)|^2 d|\sigma|(y,x)\\
& = & \int_{X} |f(x)|^2 d|\sigma|_X (x) \int_{Y} |g(y)|^2 d|\sigma|_Y(y)\\
& \leq & \|\sigma\|_{\bb{A}}^2 \|f\|_2^2 \|g\|_2^2.
\end{eqnarray*}
By the Riesz Representation Theorem, there exists a unique operator $T_{\sigma} : H_1\rightarrow H_2$ such that
$(T_{\sigma}f,g) = \phi(f,g)$; moreover, $\|T_{\sigma}\|\leq \|\sigma\|_{\bb{A}}$.

Let $\kappa\subseteq X\times Y$ and suppose that $\supp \sigma \subseteq \hat{\kappa}$.
Let $\alpha\subseteq X$ and $\beta\subseteq Y$ be measurable
subsets with $(\alpha\times\beta)\cap \kappa \simeq \emptyset$. By deleting null sets from
$\alpha$ and $\beta$ we may assume that, in fact, $(\alpha\times\beta)\cap \kappa = \emptyset$.
If $f\in H_1$ (resp. $g\in H_2$) is supported on $\alpha$ (resp. $\beta$) then, by Proposition \ref{p_sm},
$$(T_{\sigma}f,g) = \int_{(\beta\times\alpha)\cap\hat{\kappa}} f(x)\overline{g(y)} d\sigma(y,x) = 0;$$
thus, $T_{\sigma}$ is supported on $\kappa$.

Conversely, suppose that $T_{\sigma}$
is supported on $\kappa$
and let $\beta\times\alpha \subseteq Y\times X$ be a
rectangle of finite measure, marginally disjoint from $\hat{\kappa}$.
Then
$$\sigma(\beta\times\alpha) = (T_{\sigma}\chi_{\alpha},\chi_{\beta}) = 0,$$
and Proposition \ref{p_sm} implies that $\supp\sigma\subseteq \hat{\kappa}$, up to
a marginally null set.

Finally, suppose that $h\in \Gamma(X,Y)$ and $\sigma \in \bb{A}(Y,X)$.
Write $h = \sum_{i=1}^{\infty} f_i\otimes g_i$, where
$(f_i)_{i\in \bb{N}}\subseteq H_1$ and $(g_i)_{i\in \bb{N}}\subseteq H_2$ are
sequences of functions with
$\sum_{i=1}^{\infty} \|f_i\|_2^2 < \infty$ and
$\sum_{i=1}^{\infty} \|g_i\|_2^2 < \infty$.
The estimate in the first paragraph of the proof shows that
$\int_{Y\times X} \sum_{i=1}^{\infty} |f_i(x)||g_i(y)| d|\sigma|(y,x) < \infty$.

Let $h_n = \sum_{i=1}^n f_i\otimes g_i$; by the
Lebesgue Dominated Convergence Theorem,
$\int_{Y\times X} \hat{h}_n d\sigma \to_{n\to \infty} \int_{Y\times X} \hat{h} d\sigma$.
Thus,
\begin{eqnarray*}
\langle T_{\sigma}, h \rangle & = &
\lim_{n\to\infty} \langle T_{\sigma}, h_n \rangle
= \lim_{n\to\infty} \sum_{i=1}^n  \langle T_{\sigma}, f_i\otimes g_i\rangle\\
& = &
\lim_{n\to\infty}  \int_{Y\times X} \sum_{i=1}^n f_i(x) g_i(y) d\sigma(y,x)
= \int_{Y\times X} h(x,y) d\sigma(y,x).
\end{eqnarray*}
\end{proof}

For an $\omega$-closed set $F\subseteq Y\times X$,
 we denote by ${\mathbb A}(F)$ the set of all measures $\sigma$ in
${\mathbb A}(Y,X)$ such that $\supp\sigma\subseteq F$.

Operator versions of $M$-sets and $M_1$-sets were introduced
by the authors in \cite{stt} in connection with the study of closable multipliers.
We recall the relevant definition now,
introducing the additional notion of an $M_0$-set.

\begin{definition}\label{d_opver}
Let $(X,\mu)$ and $(Y,\nu)$ be standard measure spaces.
An $\omega$-closed set $\kappa\subseteq X\times Y$ is called

(i) \ \ an \emph{operator $M$-set} if $\cl K(H_1, H_2)\cap
\frak{M}_{\max}(\kappa) \neq \{0\}$;

(ii) \  an \emph{operator $M_1$-set} if $\cl K(H_1, H_2)\cap
\frak{M}_{\min}(\kappa) \neq \{0\}$;

(iii) an \emph{operator $M_0$-set} if there exists a non-zero measure $\sigma\in {\mathbb A}(\hat\kappa)$ such that $T_\sigma\in \cl K(H_1, H_2).$

We call $\kappa$ an \emph{operator $U$-set} (resp.
an \emph{operator $U_1$-set}, an \emph{operator $U_0$-set})
if it is not an operator $M$-set (resp. an operator $M_1$-set,
an operator $M_0$-set).
\end{definition}

(Operator) $M$-sets will be referred to as
\emph{sets of (operator) multiplicity}, while (operator) $U$-sets
-- as \emph{sets of (operator) uniqueness}.
It will follow from Theorem \ref{th_pse} that if $\sigma\in \bb{A}(\hat{\kappa})$ then
$T_{\sigma}\in \frak{M}_{\min}(\kappa)$.
Therefore, every operator $M_0$-set is an operator $M_1$-set,
while  every operator $M_1$-set is trivially an operator $M$-set.

\medskip

\noindent {\bf Remark }
Recall that $\mu\in M(G)$ is called a Rajchman measure if
$\lambda(\mu)\in C_r^*(G)$.
The compact operators of the form $T_{\sigma}$, where $\sigma\in \bb{A}(Y,X)$,
can be thought of as an operator version of these measures.

\medskip

\subsection{ A symbolic calculus}

Aiming at applications to multiplicity sets we establish here a kind of symbolic calculus for completely bounded maps
from $\cl B(L^2(G))$ to $\vn(G)$ (Theorem \ref{p_maps}).
We first recall the Stone-von Neumann Theorem in a suitable for our needs form.
Let $\cl D = \{M_a : a\in L^{\infty}(G)\}$ and $\cl D_0 = \{M_a : a\in C_0(G)\}$.
For each $s\in G$, let $\alpha_s : C_0(G)\rightarrow C_0(G)$ be given by
$\alpha_s f(t) = f(s^{-1}t)$. The map $s\mapsto \alpha_s$
is a homomorphism from $G$ into the automorphism group of $C_0(G)$,
and thus gives rise to the (C*-algebraic) crossed product $C_0(G)
\rtimes_{\alpha} G$. Denoting for a moment by $\pi :
C_0(G)\rightarrow \cl B(L^2(G))$ the representation
given by $\pi(g) = M_g$,
we have
that the pair $(\pi,\lambda)$ (where $\lambda$ is the left regular
representation of $G$ on $L^2(G)$), is a covariant representation of
the dynamical system $(C_0(G),G,\alpha)$. Thus, $(\pi,\lambda)$
gives rise to a representation $\pi\times\lambda$
of $C_0(G) \rtimes_{\alpha} G$ on $L^2(G)$.
By the Stone-von Neumann theorem (see \cite[Theorem 4.23]{williams}), this representation is
faithful and its image coincides with the algebra $\cl K$
of all compact operators on $L^2(G)$.
In particular, we claim that
\begin{equation}\label{eq_k}
\cl K = \overline{[AT : A\in \cl D_0, T\in C^*_r(G)]}^{\|\cdot\|}
= \overline{[ATB : A, B\in \cl D_0, T\in C^*_r(G)]}^{\|\cdot\|}
\end{equation}
(here, and in the sequel, $[\cl E]$ denotes the linear span of $\cl E$).
To see that (\ref{eq_k}) holds,
note that if $f\in L^1(G)$,
$T=\lambda(f)$ and $A,B\in \cl D_0$, then
$$AT=\int_G f(t)A \lambda_t dt\in (\pi\times \lambda)(C_0(G)\rtimes_\alpha G)=\cl K,$$
and thus $ATB\in \cl K$ as well.
Conversely, it is easy to observe (see, {\it e.g.}, \cite{ped}) that the operators of the form
$\sum_{i=1}^k \int_{E_i} A_i \lambda_s ds$,
where $E_i\subseteq G$ are measurable sets of finite measure and $A_i\in \cl D_0$, $i = 1,\dots,k$,
form a dense subset of
$(\pi\times \lambda)(C_0(G)\rtimes_\alpha G)$;
however, $\int_{E_i} A_i \lambda_s ds = A_i \lambda(\chi_{E_i})$, and
the first equality in
(\ref{eq_k}) is established.
To complete the proof of the second equality, let $(B_i)_{i=1}^{\infty}\subseteq \cl D_0$ be a
sequence strongly converging to the identity operator on $L^2(G)$, and
note that if $A\in \cl D_0$ and $T\in C^*_r(G)$, then
$AT = \lim_i ATB_i$ in norm, by the compactness of $AT$.

In the sequel, we will use the norm closed $\cl D$-bimodule generated by
$C_r^*(G)$
\begin{equation}\label{eq_a}
\cl A = \overline{[ATB : A, B\in \cl D, T\in C_r^*(G)]}^{\|\cdot\|}
\end{equation}
and the smallest norm closed subspace of $\cl B(L^2(G))$ containing
$C_r^*(G)$ and invariant under Schur multipliers
\begin{equation}\label{eq_aa}
\cl R = \overline{[S_{\nph}(T) : T \in C_r^*(G), \nph\in
\frak{S}(G,G)]}^{\|\cdot\|}.
\end{equation}
By (\ref{eq_k}),
\begin{equation}\label{eq_kcon}
\cl K \subseteq \cl A\subseteq \cl R.
\end{equation}

\begin{remark}\label{r_aatilde}
{\bf (i)}
{\rm Let $G$ be discrete. Then $\cl A = \cl R$. Indeed,
in this case $\cg$ is generated as a closed linear space by the
unitaries $\lambda_s$, $s\in G$, which are normalisers of the multiplication masa $\cl D$.
However, if $\nph\in \frak{S}(G,G)$ then $S_{\nph}(\lambda_s) = M_f\lambda_s\in \cl A$ for some
$f\in L^{\infty}(G)$ (see, {\it e.g.}, \cite[Proposition 14]{kp}). It follows that $\cl A$ is invariant
under Schur multiplication, and hence $\cl R = \cl A$. Note that, in the case $G$ is
infinite, $\cl K$ is strictly contained in $\cl A$ since $\lambda_s$ is a unitary operator in $C^*_r(G)$
which is not compact.}

{\rm In \cite{roe},
given a discrete group $G$,
J. Roe introduced what is now known as \emph{the uniform Roe algebra}
$UC^*_r(G)$ which equals, by definition, to the uniform closure in $\cl B(\ell^2(G))$
of the space of all matrices indexed by $G\times G$ with uniformly bounded entries
supported on sets of the form $\{(s,t)\in G\times G : ts^{-1}\in E\}$, where $E$ is finite.
We note that
$UC^*_r(G)$ coincides in this case with $\cl R$.
Indeed, the unitary generators $\lambda_s$ are represented by
matrices (indexed by $G\times G$) whose $s$th diagonal has all entries equal to 1,
and all other diagonals are zero.
Multiplying by an operator of the form $M_a$, where $a\in \ell^{\infty}(G)$,
we see that all matrices which, on a given diagonal, have a sequence from
$\ell^{\infty}(G)$, are in $\cl A = \cl R$; thus, $UC^*_r(G)\subseteq \cl R$.
Conversely, since $C_r^*(G)$ is generated as a norm closed subspace by
the operators of the form $\lambda_s$, we have that $\cl A\subseteq UC^*_r(G)$,
and hence $UC^*_r(G) = \cl R$.

The previous paragraph shows that the space
$\cl R$ can be thought of as a locally compact version of
the uniform Roe algebra.}

\smallskip

\noindent {\bf (ii)} {\rm If $G$ is compact then
$\cl K = \cl A = \cl R$. Indeed, in this case $C^*_r(G)\subseteq \cl K$
and since the compact operators are invariant under Schur multipliers, we
have that $\cl R\subseteq \cl K$, and the equalities follow from (\ref{eq_kcon}).}
\end{remark}

In view of Remark \ref{r_aatilde}, it is natural to ask whether
$\cl A = \cl R$ for every locally compact group $G$;
we do not  know whether this equality always holds.

If $G$ is compact then $N(A(G))\subseteq \Gamma(G,G)$ and hence
the formula
$$\langle E(T),u\rangle = \langle T, N(u)\rangle, \ \ \ T\in \cl B(L^2(G)), u\in A(G),$$
defines a canonical expectation $E$ from $\cl B(L^2(G))$ onto $\vn(G)$.
This is the motivation behind
the next theorem, where we exhibit a symbolic
calculus for completely bounded maps from $\cl B(L^2(G))$ into
$\vn(G)$ (that are not necessarily projections). Let us denote by
$CB^{w^*}(\cl B(L^2(G)),\vn(G))$ the space of weak* continuous completely
bounded maps from $\cl B(L^2(G))$ into $\vn(G)$. It has a
natural structure of a right Banach module over $\frak{S}(G,G)$, the action being
given by $\Phi\cdot\nph = \Phi\circ S_{\nph}$. Note that
$\Gamma(G,G)$ is also a right Banach module over $\frak{S}(G,G)$
under the action $\psi\cdot \nph = \psi \nph$.

\begin{theorem}\label{p_maps}
For every $\nph\in \Gamma(G,G)$ and every
$T\in \cl B(L^2(G))$, there exists a unique operator $E_{\nph}(T)\in \vn(G)$ such that
$$\langle E_{\nph}(T),u\rangle = \langle T,\nph N(u)\rangle, \ \ \ u\in A(G).$$
The transformation $\nph\rightarrow E_{\nph}$ is a contractive
$\frak{S}(G,G)$-module map from $\Gamma(G,G)$ into
$CB^{w^*}(\cl B(L^2(G)),\vn(G))$.
Moreover, if $\nph\in \Gamma(G,G)$ then
$E_\varphi(\lambda_s)=P(\varphi)(s)\lambda_s$, $s\in G$,
and  $E_{\nph}(T)\in C_r^*(G)$, for all $T\in \cl R$.
\end{theorem}
\begin{proof}
Fix $\nph\in \Gamma(G,G)$ and consider the mapping $e_{\nph} :
A(G)\rightarrow \Gamma(G,G)$ given by
$e_{\nph}(u)=\nph N(u)$, $u\in A(G)$.
The mapping $N : A(G)\rightarrow \frak{S}(G,G)$ is completely isometric
(see, {\it e.g.}, \cite{spronk}). On the other hand,
the mapping $\psi\rightarrow \nph\psi$ from $\frak{S}(G,G)$ into
$\Gamma(G,G)$ is completely bounded with completely bounded
norm not exceeding $\|\nph\|_{\Gamma}$. Indeed,
let $\psi_{i,j}\in \frak{S}(G,G)$, $i,j = 1,\dots,n$;
then, denoting by $F_{\nph}$ the functional on
$\cl B(L^2(G))$ given by $F_{\nph}(T) = \langle \nph,T\rangle$, we have
\begin{eqnarray*}
\left\|(\nph \psi_{i,j})_{i,j}\right\|_{M_n(\Gamma(G,G))}
& = & \left\|(\nph \psi_{i,j})_{i,j}\right\|_{CB(\cl B(L^2(G)), M_n(\mathbb C))}\\
&=&
\sup_{\|(T_{p,q})_{p,q}\|\leq 1} \|(\langle \nph \psi_{i,j}, T_{p,q}\rangle)_{(i,p),(j,q)}\|\\
& = &
\sup_{\|(T_{p,q})_{p,q}\|\leq 1} \|(\langle \nph, S_{\psi_{i,j}}(T_{p,q})\rangle)_{(i,p),(j,q)}\|\\
& \leq &
\sup_{\|(T_{p,q})_{p,q}\|\leq 1} \|F_{\nph}\| \|(S_{\psi_{i,j}}(T_{p,q}))_{(i,p),(j,q)}\|\\
& \leq &
\|\nph\|_{\Gamma} \|(S_{\psi_{i,j}})_{i,j}\|_{\cb}
= \|\nph\|_{\Gamma} \|(\psi_{i,j})_{i,j}\|_{M_n(\frak{S}(G,G))}.
\end{eqnarray*}
Thus, $e_{\nph}$ is completely
bounded and $\|e_{\nph}\|_{\rm cb}\leq \|\nph\|_{\Gamma}$.
It follows that the map $E_{\nph}  = e_{\nph}^*$
is a normal
completely bounded map from $\cl B(L^2(G))$ into $\vn(G)$
and $\|E_{\nph}\|_{\rm cb} \leq \|\nph\|_{\Gamma}$. The
identity
$$\langle E_{\nph}(T),u\rangle = \langle T,\nph N(u)\rangle, \ \ \ u\in A(G), T\in \cl B(L^2(G)),$$
holds by the definition of $E_{\nph}$.

It is obvious that the map $E:\nph\rightarrow E_{\nph}$ is linear
and, by the previous paragraph, it is contractive. Moreover, if
$\nph\in \Gamma(G,G)$, $\psi\in \frak{S}(G,G)$ and $u\in A(G)$, then
$$\langle E_{\nph\psi}(T),u\rangle = \langle T,\psi\nph N(u)\rangle =
\langle S_{\psi}(T),\nph N(u)\rangle
= \langle (E_{\nph}\circ S_{\psi})(T),u\rangle,$$ which shows that
$E$ is a $\frak{S}(G,G)$-module map.

Using (\ref{eq_P}), for every $u\in A(G)$ we have
\begin{eqnarray*}
\langle E_\varphi(\lambda_s),u\rangle & =& \langle\lambda_s,\varphi N(u)\rangle=
P(N(u)\varphi)(s)\\ & = & u(s)P(\varphi)(s) = \langle P(\varphi)(s)\lambda_s,u\rangle,
\end{eqnarray*}
which shows that $E_\varphi(\lambda_s) = P(\varphi)(s)\lambda_s$.

Now suppose that $f \in L^1(G)$ and let $a,b\in L^2(G)$.
Then
\begin{equation}\label{eq_snu}
S_{N(u)}(\lambda(f))=\lambda(uf), \ \ u\in A(G).
\end{equation}
Indeed, write
$N(u)=\sum_{i=1}^{\infty} f_i\otimes g_i$, where
$\|\sum_{i=1}^{\infty} |f_i|^2\|_{\infty} \leq C<\infty$
and
$\|\sum_{i=1}^{\infty} |g_i|^2\|_{\infty} \leq C<\infty$.
Then
$$(S_{N(u)}(\lambda(f))a,b)
= \sum_{i=1}^{\infty}\iint g_i(t)f(s)f_i(s^{-1}t)a(s^{-1}t)\overline{b(t)}dsdt.$$
%Let
%$r_N = \sum_{i=1}^{N}\int\int g_i(t)f(s)f_i(s^{-1}t)a(s^{-1}t)\overline{b(t)}dsdt$,
%$N\in \bb{N}$.
For a fixed $s\in G$,
\begin{eqnarray*}
\left|\sum_{i=1}^{\infty} g_i(t)f_i(s^{-1}t)a(s^{-1}t)\overline{b(t)}\right|^2
& \leq & \sum_{i=1}^{\infty} \left|g_i(t)\overline{b(t)}\right|^2
\sum_{i=1}^{\infty} \left|f_i(s^{-1}t) a(s^{-1}t)\right|^2\\
& \leq & C^2 |a(s^{-1}t)|^2 |b(t)|^2,
\end{eqnarray*}
for almost every $t\in G$.
On the other hand,
$$\iint |a(s^{-1}t)| |b(t)| dt \leq \|a\|_2\|b\|_2$$
and hence the $L^1$-norm of the function $t\to |a(s^{-1}t)| |b(t)|$
is independent of $s$.
The Lebesgue Dominated Convergence Theorem implies that
\begin{eqnarray*}
& & \iint \sum_{i=1}^{\infty} f(s) g_i(t)f_i(s^{-1}t)a(s^{-1}t)\overline{b(t)}dtds\\
& = & \int f(s) \sum_{i=1}^{\infty}\int g_i(t)f_i(s^{-1}t)a(s^{-1}t)\overline{b(t)}dt ds.
\end{eqnarray*}
Since the function $s\mapsto \sum_{i=1}^{\infty}\int g_i(t)f_i(s^{-1}t)a(s^{-1}t)\overline{b(t)}dt$
is (essentially) bounded, while $f\in L^1(G)$,
another application of the Lebesgue Dominated Convergence Theorem
shows that
\begin{eqnarray*}
& & \iint \sum_{i=1}^{\infty} f(s) g_i(t)f_i(s^{-1}t)a(s^{-1}t)\overline{b(t)}dtds\\
& = & \sum_{i=1}^{\infty} \iint  f(s) g_i(t)f_i(s^{-1}t)a(s^{-1}t)\overline{b(t)}dt ds.
\end{eqnarray*}
It now follows that
\begin{eqnarray*}
(S_{N(u)}(\lambda(f))a,b)
& = & \iint f(s)N(u)(s^{-1}t,t)a(s^{-1}t)\overline{b(t)}dtds \\
&=& \iint u(s)f(s)a(s^{-1}t)\overline{b(t)}dtds
 =  (\lambda(uf)a,b).
\end{eqnarray*}
Thus, (\ref{eq_snu}) is established.

The mapping $u\mapsto N(u)$ from $A(G)$ into ${\mathfrak S}(G,G)$ is an isometry (see, {\it e.g.}, \cite{spronk});
hence
$\|S_{N(u)}(\lambda(f))\|\leq \|N(u)\|_{{\mathfrak S}}\|\lambda(f)\|=\|u\|_{A(G)}\|\lambda(f)\|$ and
therefore  the mapping $u\mapsto \lambda(uf), A(G)\to C_r^*(G)$, is continuous.
We also have
\begin{eqnarray*}
& & \langle E_{a\otimes b}(\lambda(f)),u\rangle = \langle
\lambda(f),(a\otimes b) N(u)\rangle =
(S_{N(u)}(\lambda(f))a,\overline{b})\\
& = & (\lambda(uf)a,\overline{b})
 =  \iint u(s)f(s)a(s^{-1}t) b(t) ds dt\\
& = & \int u(s)f(s) \left(\int a(s^{-1}t) b(t) dt\right) ds
 =  \int u(s)f(s) (b\ast \check{a})(s) ds.
\end{eqnarray*}
Using (\ref{eq_p}), we conclude that
$$\langle E_{a\otimes b}(\lambda(f)),u\rangle = \int u(s)f(s) P(a\otimes b)(s) ds.$$
Note that, since $P(a\otimes b)\in A(G)$, the function
$P(a\otimes b)f$ belongs to $L^1(G)$ and hence
\begin{equation}\label{eq_cl}
\langle E_{\nph}(\lambda(f)),u\rangle  = \langle \lambda(P(\nph)f),
u\rangle
\end{equation}
for $\nph = a\otimes b$. By linearity,
(\ref{eq_cl}) holds whenever $\nph$ is a finite sum of elementary
tensors. By the continuity of the transformations $\nph\rightarrow E_{\nph}$,
$\nph \rightarrow P(\nph)$ and
$g\rightarrow \lambda(gf)$ (the last one mapping $A(G)$ into $C^*_r(G)$),
we conclude that (\ref{eq_cl}) holds
for all $\nph\in \Gamma(G,G)$.

Relation (\ref{eq_cl}) implies that
$E_{\nph}(\lambda(f)) = \lambda(P(\nph)f)$ $\in$ $\cg$, for all
$f\in L^1(G)$ and all $\nph\in \Gamma(G,G)$. Since $E_{\nph}$ is
norm continuous and $\lambda(L^1(G))$ is dense in $\cg$, we have
that $E_{\nph}(\cg)\subseteq \cg$. If $\psi\in \frak{S}(G,G)$ and
$T\in \cg$ then
$$E_{\nph}(S_{\psi}(T)) = E_{\nph\psi}(T)\in \cg.$$
It follows that $E_{\nph}(\cl R)\subseteq \cg$, for every $\nph \in \Gamma(G,G)$.
\end{proof}

A version of the next lemma for $\frak{M}_{\min}(\kappa)$ was proved in \cite[Proposition 5.3]{st2}.

\begin{lemma}\label{l_mmax}
If $\kappa\subseteq X\times Y$ is an $\omega$-closed set then
$$\frak{M}_{\max}(\kappa) = \{T\in \cl B(H_1,H_2) :
S_{\nph}(T) = 0, \mbox{ for all } \nph\in \frak{S}(X,Y),$$
$$\mbox{ vanishing on an $\omega$-open neighbourhood of } \kappa\}.$$
\end{lemma}
\begin{proof}
Suppose that $T\in \cl B(H_1,H_2)$ belongs to the set on the right hand side of the
above equality. If $\kappa\cap(\alpha\times\beta) \simeq \emptyset$ then
$\chi_{\alpha\times\beta} \in \frak{S}(X,Y)$ vanishes on the $\omega$-open
neighbourhood $(\alpha\times\beta)^c$ of $\kappa$ and hence
$M_{\chi_{\beta}}TM_{\chi_{\alpha}} = S_{\chi_{\alpha\times\beta}}(T) = 0$;
thus, $T\in \frak{M}_{\max}(\kappa)$.

Conversely, suppose that $T\in \frak{M}_{\max}(\kappa)$ and let
$\nph\in \frak{S}(X,Y)$ vanish on an $\omega$-open neighbourhood of $\kappa$.
If $h\in \Gamma(X,Y)$ then $\nph h\in \Gamma(X,Y)$ and vanishes on an
$\omega$-open neighbourhood of $\kappa$. By \cite{st1},
$$\langle S_{\nph}(T),h\rangle = \langle T,\nph h\rangle = 0,$$ showing that
$S_{\nph}(T) = 0$.
\end{proof}

The following lemma
will be needed in the proof of Theorem \ref{th_iffop}.

\begin{lemma}\label{eab}
Suppose that $T\in
\cl B(L^2(G))$ is non-zero. Then there exist $a$, $b\in L^2(G)$ such that
$E_{a\otimes b}(T) \ne  0$.
\end{lemma}
\begin{proof}
Let $T\in \cl B(L^2(G))$ be a non-zero operator, and suppose,
by way of contradiction, that
$E_{a\otimes b}(T) = 0$ for all $a,b\in L^2(G)$.
We may assume that
$T = M_{\chi_K}TM_{\chi_K}$ for some compact set $K\subseteq G$.
By Theorem \ref{p_maps},
$E_\varphi(T) = 0$ for every $\varphi\in\Gamma(G,G)$.
Since
$$\langle E_\varphi(T),u\rangle=\langle T,\varphi N(u)\rangle=\langle S_{N(u)}(T),\varphi\rangle,
\ \ \ u\in A(G), \varphi\in \Gamma(G,G),$$
we have that $S_{N(u)}(T) = 0$ for every $u\in A(G)$.
Let
$$\cl W = {\rm span}\{N(u)\psi : \psi\in \Gamma(G,G), u\in A(G)\}.$$
Then $\cl W\subseteq\Gamma(G,G)$ is a subspace, invariant under $\frak{S}(G)$, 
and $T\in \cl W^\perp$.
Denoting by $\text{null}(\cl W)$ the complement of the $\omega$-union \cite{stt} of the family 
$\{h^{-1}(\mathbb C\setminus\{0\} : h\in \cl W\}$, we have
$\text{null}(\cl W)\simeq\emptyset$. In fact, since $G$ is second countable and
locally compact, there exists an increasing sequence of compact sets $\{K_n\}$ such that 
$\cup_{n=1}^\infty K_n=G$. For each $n\in\mathbb N$, choose 
a function $u_n\in A(G)$ that takes the value $1$ on $K_n$. Then, up to a marginally null set, 
\begin{eqnarray*}
\text{null}(\cl W) & \subseteq & \cap_{n,m=1}^\infty\text{null} (N(u_n)\chi_{K_m}\times\chi_{K_m})
\subseteq \cap_{n,m=1}^\infty(K_n^*\cap(K_m\times K_m))^c\\
& = & \cap_{n=1}^\infty ((K_n^c)^*\cup(\cup_{m=1}^\infty K_m\times K_m)^c)
=  \cap_{n=1}^\infty (K_n^c)^*\\ & = & ((\cup_{n=1}^\infty K_n)^c)^*=\emptyset.
\end{eqnarray*}
By \cite[Corollary 4.3]{st1}, $\cl W$ is dense in $\Gamma(G,G)$ and hence $T=0$, a contradiction.
\end{proof}

If $E\subseteq G$, we let $$E^* = \{(s,t)\in G\times G : ts^{-1}\in E\}.$$
\noindent If $E$ is closed then $E^*$ is closed and hence, if $G$ is second countable, it is $\omega$-closed.

\subsection{ Multiplicity versus  operator multiplicity}

In the case of compact abelian groups,
a connection between $M$-sets (resp. $M_1$-sets) and
operator $M$-sets (resp. operator $M_1$-sets) was established in \cite{f}
(resp. \cite{stt}). Our aim now is to extend these results to arbitrary locally compact groups;
a corresponding statement for $M_0$-sets will be proved in the next subsection.

\begin{theorem}\label{th_iffop}
Let $G$ be a locally compact second countable group and
let $E\subseteq G$ be a closed subset.

\noindent (a) The following are equivalent:

(i) \ \ $E$ is an $M$-set;

(ii) \ $E^*$ is an operator $M$-set;

(iii) $\cl A\cap \frak{M}_{\max}(E^*) \neq \{0\}$;

(iv) $\cl R\cap \frak{M}_{\max}(E^*) \neq \{0\}$.

\noindent (b) The following are equivalent:

(i') \ \ $E$ is an $M_1$-set;

(ii') \ $E^*$ is an operator $M_1$-set;

(iii') $\cl A\cap \frak{M}_{\min}(E^*) \neq \{0\}$;

(iv') $\cl R\cap \frak{M}_{\min}(E^*) \neq \{0\}$.
\end{theorem}
\begin{proof}
(a) (i)$\Rightarrow$(ii) Let $E$ be an $M$-set; then there
exists a non-zero operator $T\in J(E)^{\perp}\cap C^*_r(G)$.
Suppose that $AT = 0$ for all $A\in \cl D_0$. Since $\cl D_0$ is
weak* dense in $\cl D$, there exists a net $(A_j)_{j\in \bb{J}}\subseteq \cl D_0$
such that $\lim_{j\in \bb{J}} A_j = I$ in the weak* topology. After passing to a limit, we obtain
that $T = 0$, a contradiction. Thus, there exists $A\in \cl D_0$
such that $AT \neq 0$; in view of (\ref{eq_k}), $AT \in \cl K$. By
\cite[Lemma 4.1]{lt}, $T\in \frak{M}_{\max}(E^*)$ and hence $AT\in
\frak{M}_{\max}(E^*)$; thus, $E^*$ is an $M$-set.

(ii)$\Rightarrow$(iii)$\Rightarrow$(iv) follow from the inclusions (\ref{eq_kcon}).

(iv)$\Rightarrow$(i)
Suppose that $T\in \cl R\cap \frak{M}_{\max}(E^*)$ is non-zero.
By Lemma~\ref{eab}, there exist $a,b\in L^2(G)$
such that $E_{a\otimes b}(T) \neq 0$. By Theorem \ref{p_maps},
$E_{a\otimes b}(T)\in C_r^*(G)$; we claim that, moreover,
$E_{a\otimes b}(T) \in J(E)^{\perp}$. To see this, let $u\in A(G)$ vanish on an open
neighbourhood of $E$. Then $N(u) \in \frak{S}(G,G)$ vanishes on an
$\omega$-open neighbourhood of $E^*$, and hence the function
$(a\otimes b)N(u) \in \Gamma(G,G)$ vanishes on an $\omega$-open
neighbourhood of $E^*$. On the other hand, by \cite[Theorem~4.3]{st1}, we have
$$(S_{N(u)}(T)a,\overline{b}) = \langle T, (a\otimes b) N(u)\rangle=0,$$
giving $\langle E_{a\otimes
b}(T),u\rangle = 0$. Thus, $0\neq E_{a\otimes b}(T) \in
J(E)^{\perp}$ and hence $E$ is an $M$-set.

(b) (i')$\Rightarrow$(ii')
We claim that $\lambda_s\in \frak{M}_{\min}(E^*)$
for every $s\in E$. To see this, suppose that
$w\in \Gamma(G,G)$ vanishes on the set $E^*$, that is, $w\chi_{E^*}
= 0$ marginally almost everywhere.
For every $r\in G$ and $s\in E$, we have that $(s^{-1}r,r)\in E^*$
and hence $w(s^{-1}r,r) = 0$ for every $s\in E$ and almost every $r\in G$.
By (\ref{eq_P}), $P(w)(s) = 0$ for every $s\in E$ and hence, by
(\ref{eq_p}), $\langle \lambda_s,w\rangle = 0$ for every $s\in E$;
the claim is thus proved.

Suppose that $E$ is an $M_1$-set, and let
$0\neq T\in I(E)^{\perp} \cap C_r^*(G)$.
A direct verification shows that $I(E)^{\perp} = \overline{[\lambda_s : s\in
E]}^{w^*}$.
It follows from the previous paragraph that
$T\in \frak{M}_{\min}(E^*)$. As in the proof of the implication
(i)$\Rightarrow$(ii), we conclude that there exists $A\in \cl D_0$
such that $0\neq AT \in \cl K\cap \frak{M}_{\min}(E^*)$, that is,
$E^*$ is an $M_1$-set.

(ii')$\Rightarrow$(iii')$\Rightarrow$(iv') follow from the inclusions (\ref{eq_kcon}).

(iv')$\Rightarrow$(i') Suppose that
$0\neq T\in \cl
R\cap \frak{M}_{\min}(E^*)$. As in the proof of (a), we can show
that there exist $a,b\in L^2(G)$ such that $E_{a\otimes b}(T)$
is a non-zero element of $C_r^*(G)$ annihilating $I(E)$.
\end{proof}

\subsection{ The case of $M_0$-sets}

In order to establish a statement for $M_0$-sets,
analogous to the ones from Theorem \ref{th_iffop},
we need a couple of auxiliary lemmas.

\begin{lemma}\label{l_compph}
If $\sigma$ is an Arveson measure on $G\times G$ then
for every $\nph \in \Gamma(G,G)$
there exists a unique measure $\sigma_{\nph} \in M(G)$ such that
$E_{\nph}(T_{\sigma}) = \lambda(\sigma_{\nph})$.
Moreover, if $\supp \sigma \subseteq \widehat{E^*}$
then $\supp \sigma_{\nph}\subseteq E$.
\end{lemma}
\begin{proof}
Let $\nph = \sum_{i=1}^{\infty} f_i\otimes g_i \in \Gamma(G,G)$ (here $\sum_{i=1}^{\infty} \|f_i\|_2^2 < \infty$
and $\sum_{i=1}^{\infty} \|g_i\|_2^2$ $<$ $\infty$);  note that
$$\left\|\sum_{i=1}^{\infty} |f_i|\otimes |g_i|\right\|_{\Gamma} \leq \|\nph\|_{\Gamma}.$$
If $u\in C_0(G)$ then
\begin{eqnarray*}
& & \left|\int_{G\times G} \nph(s,t) u(ts^{-1}) d \sigma(t,s) \right| \leq
\int_{G\times G} |\nph(s,t)| |u(ts^{-1})| d |\sigma|(t,s)\\
& \leq & \|u\|_{\infty} \int_{G\times G} \sum_{i=1}^{\infty} |f_i(s)| |g_i(t)| d |\sigma|(t,s)
=  \|u\|_{\infty}  \sum_{i=1}^{\infty} (T_{|\sigma|} |f_i|, |g_i|)\\
& = & \|u\|_{\infty} \langle T_{|\sigma|}, \sum_{i=1}^{\infty} |f_i|\otimes |g_i|\rangle
\leq \|u\|_{\infty} \|T_{|\sigma|}\| \left\|\sum_{i=1}^{\infty} |f_i|\otimes |g_i|\right\|_{\Gamma}\\
& \leq & \|u\|_{\infty} \| \sigma \|_{\bb{A}} \|\nph\|_{\Gamma}.
\end{eqnarray*}
It follows that the functional $R : C_0(G) \rightarrow \bb{C}$ given by
$$R(u) = \int_{G\times G} \nph(s,t) u(ts^{-1}) d \sigma(t,s), \ \ \ u\in C_0(G),$$
is well-defined and bounded. Hence,
there exists $\sigma_{\nph} \in M(G)$ such that
\begin{equation}\label{eq_sph}
\int_{G\times G} \nph(s,t) u(ts^{-1}) d \sigma(t,s) = \int_G u(x) d\sigma_{\nph}(x), \ \ u\in C_0(G).
\end{equation}
On the other hand,
$$\int_G u(x) d\sigma_{\nph}(x) = \langle \lambda(\sigma_{\nph}), u\rangle, \ \ \ u\in A(G).$$
By (\ref{eq_sph}) and Theorem \ref{th_ps},
$$\langle \lambda(\sigma_{\nph}), u\rangle =\langle T_\sigma,\nph N(u)\rangle= \langle E_{\nph}(T_{\sigma}), u\rangle, \ \ \ u\in A(G);$$
thus, $E_{\nph}(T_{\sigma}) = \lambda(\sigma_{\nph})$.

Now suppose that $\supp \sigma \subseteq \widehat{E^*}$ and that
$U\subseteq G$ is an open set, disjoint from $E$.
For any function $u\in C_0(G)$ with $\supp u\subseteq U$, we have that
$\supp N(u)\subseteq \widehat{U^*}$.
On the other hand, $\widehat{U^*}$ is  disjoint from
$\widehat{E^*}$ and hence Proposition \ref{p_sm} implies that
$|\sigma|(\widehat{U^*}) = 0$. Now (\ref{eq_sph}) shows that
$\int_G u(x) d\sigma_{\nph}(x) = 0$. It follows that $\sigma_{\nph}(U) = 0$;
thus, $\supp \sigma_{\nph}\subseteq E$.
\end{proof}

We will need the following fact, which was discussed in \cite[p. 347]{sourour} in the case
of a finite measure (here we need a $\sigma$-finite version of this as the Haar measure on a
locally compact non-compact group is such).

\begin{lemma}\label{l_sa}
Let $(X,\mu)$ and $(Y,\nu)$ be $\sigma$-finite standard measure spaces and $(\sigma^x)_{x\in X}$ be a family of
complex Borel measures on $Y$ such that, for every measurable $F\subseteq Y$, the function $x\mapsto \sigma^x(F)$
is measurable.
Suppose that the function $x\mapsto \|\sigma^x\|$ is integrable and essentially bounded
(with respect to the measure $\mu$). Then there exists a
Borel measure $\sigma$ on $Y\times X$ such that $\sigma(E) = \int_X\int_Y \chi_E(y,x) d\sigma^x(y) d \mu(x)$,
for every measurable set $E\subseteq Y\times X$,
and a constant $c > 0$ such that $|\sigma|(Y\times \alpha)\leq c \mu(\alpha)$
for every measurable set $\alpha\subseteq X$.
\end{lemma}
\begin{proof}
First of all, notice that the quantity
$$\sigma(E) = \int_X\int_Y \chi_E(y,x) d\sigma^x(y) d \mu(x)$$
is finite.
Indeed,
\begin{eqnarray*}
\left|\int_X\int_Y \chi_E(y,x) d\sigma^x(y) d m(x)\right| & \leq &
\int_X\int_Y \chi_E(y,x) d|\sigma^x|(y) d \mu(x)\\
& \leq & \int_X |\sigma^x|(Y) d \mu(x) < \infty.
\end{eqnarray*}
A direct verification now shows that $\sigma$ is a measure.
Moreover, the above estimate yields
$$|\sigma|(E) \leq \int_X\int_Y \chi_E(y,x) d|\sigma^x|(y) d \mu(x),$$
for every measurable set $E\subseteq Y\times X$.
Letting $c = \esssup_{x\in X} \|\sigma^x\|$, for every measurable $\alpha\subseteq X$, we have
$$|\sigma|(Y\times \alpha) \leq \int_{\alpha} \|\sigma^x\| d\mu(x) \leq  c\mu(\alpha).$$
\end{proof}

In the next theorem, we let
$$\frak{P}(\kappa) = \{T_{\mu} : \mu\in \bb{A}(\hat{\kappa})\}.$$

\begin{theorem}\label{thm0set}
Let $E\subseteq G$ be a closed set. The following are equivalent:

(i) \ \ $E$ is an $M_0$-set;

(ii) \ $E^*$ is an operator $M_0$-set;

(iii) $\cl A\cap \frak{P}(E^*) \neq \{0\}$;

(iv) $\cl R\cap \frak{P}(E^*) \neq \{0\}$.
\end{theorem}
\begin{proof}
(i)$\Rightarrow$(ii) Let $\theta\in M(G)$ be such that $\supp\theta\subseteq E$ and
$\lambda(\theta) \in C_r^*(G)$. Then $M_g\lambda(\theta)M_f$ is a compact operator
for all $f,g\in C_0(G)$ (see (\ref{eq_k})).

For each $x\in G$, let $\theta^x\in M(G)$ be given by
$\theta^x(\alpha) = \theta(x\alpha^{-1})$ and
$\theta_x$ be given by $\theta_x(\alpha) =
\theta(\alpha x^{-1})$, for any measurable
$\alpha\subseteq G$ (here $\alpha^{-1} = \{s^{-1} : s\in \alpha\}$).
Let $\theta^*\in M(G)$ be defined by $d\theta^*(s) = \overline{d\theta(s^{-1})}$;
then $\lambda(\theta^*) = \lambda(\theta)^*$.
First observe that $\|\theta^x\|
= \|\theta\|$ for each $x\in G$. Indeed, if $\{\alpha_j\}_{j=1}^N$ is a
measurable partition of $G$ then $\{x\alpha_j^{-1}\}_{j=1}^N$
is also such, and hence
$$\sum_{j=1}^N |\theta^x(\alpha_j)| = \sum_{j=1}^N
|\theta(x\alpha_j^{-1})| \leq \|\theta\|.$$ On the other hand, for
every $\epsilon > 0$, letting $\{\beta_k\}_{k=1}^K$ be a measurable
partition of $G$ such that $\sum_{k=1}^K |\theta(\beta_k)| > \|\theta\| -
\epsilon$, we see that $\{\beta_k^{-1} x\}_{k=1}^K$ is a
measurable partition of $G$ with $\sum_{k=1}^K |\theta^x(\beta_k^{-1}
x)| > \|\theta\| - \epsilon$, and so $\|\theta^x\| \geq \|\theta\|$.
Similarly, $\|\theta^*_x\| = \|\theta^*\|$ for all $x\in G$.

If $f,g\in C_0(G)$ then
\begin{eqnarray}\label{eq_forgsi}
(M_g\lambda(\theta)M_f \xi,\eta)&=& \iint f(y^{-1}x)\xi(y^{-1}x) g(x) \overline{\eta(x)} d\theta(y) dx\\
&=& \iint f(z)\xi(z) g(x) \overline{\eta(x)} d\theta^x(z) dx \nonumber
\end{eqnarray}
and, also,
\begin{eqnarray*}
(M_g\lambda(\theta)M_f \xi,\eta)&=&(M_{f}\xi,\lambda(\theta^*)M_{\bar g}\eta)\\
&=& \iint f(z)\xi(z) g(x^{-1}z) \overline{\eta(x^{-1}z)} \overline{d\theta^*(x)} dz\\
&=& \iint f(z)\xi(z) g(x) \overline{\eta(x)} \overline{d(\theta^*)^z(x)} dz.
\end{eqnarray*}
If, moreover, $f,g\in C_0(G)\cap L^1(G)$ and
$x\in G$, the total variation of the measure
$g(x) f(\cdot) d\theta^x(\cdot)$
equals $\int_G |f(z)| d |g(x)
\theta^x|$ which does not exceed $\|f\|_{\infty} \|g(x) \theta^x\|$.
Hence,
$\|g(x) f(\cdot) d\theta^x(\cdot)\|\leq
\|g\|_{\infty}\|f\|_{\infty}\|\theta\|$ for all $x\in G$.
Furthermore, the function $x\mapsto \|f\|_{\infty} \|g(x) \theta^x\|$
is integrable since $x\mapsto \|\theta^x\|$ is a constant function.

Similarly, the total variation of the measure $f(z) g(\cdot)
d(\theta^*)^z(\cdot)$ does not exceed
$\|g\|_{\infty}\|f\|_{\infty}\|\theta^*\|$, and the function $z\rightarrow \|g\|_{\infty} \|f(z) d(\theta^*)^z\|$
is integrable. Lemma \ref{l_sa} now shows that, if $f,g\in C_0(G)\cap L^1(G)$, then
$M_g\lambda(\theta)M_f$ is the pseudo-integral operator of the Arveson
measure $\sigma_{f,g,\theta}$ given by
$d\sigma_{f,g,\theta}(x,z) = g(x)f(z)\overline{d(\theta^*)^z(x)}dz=g(x) f(z) d\theta^x(z)dx$.
On the other hand, since $\lambda(\theta)\in C_r^*(G)$, the operator
$M_g\lambda(\theta)M_f$ is compact whenever $f,g\in C_0(G)\cap L^1(G)$.
It is now clear that, since $\theta\neq 0$, we can find functions $f,g\in C_0(G)\cap L^1(G)$ such that
$M_g\lambda(\theta)M_f$ is non-zero.

Suppose that $\alpha\times\beta$ is a measurable rectangle with
$(\alpha\times\beta)\cap E^* = \emptyset$ and $\xi\in L^2(G)$ (resp.
$\eta\in L^2(G)$) vanishes everywhere on $\alpha^c$ (resp.
$\beta^c$). For each $x\in G$, the function $y\mapsto
\xi(y^{-1}x)\overline{\eta(x)}$ vanishes on $E$ and hence, by (\ref{eq_forgsi}),
$(M_g\lambda(\theta)M_f \xi,\eta) = 0$. Thus, $M_g\lambda(\theta)M_f$ is supported on $E^*$.
%by Theorem \ref{th_ps}, the measure $\sigma_{f,g,\theta}$ is supported on $\widehat{E^*}$.

\smallskip

(ii)$\Rightarrow$(iii)$\Rightarrow$(iv) are trivial.

\smallskip

(iv)$\Rightarrow$(i) Suppose that $\sigma$ is an Arveson measure
supported on $\hat{E^*}$ such that $0\neq T_{\sigma}\in \cl R$.
By Lemma \ref{eab},
there exists $\nph\in \Gamma(G,G)$ such that $E_{\nph}(T_{\sigma}) \neq 0$.
By Lemma \ref{l_compph},
$E_{\nph}(T_{\sigma}) = \lambda(\sigma_{\nph})$, where
$\sigma_{\nph}$ is supported on $E$ and, by Theorem \ref{p_maps},
$\lambda(\sigma_{\nph})$ belongs to $C_r^*(G)$.
\end{proof}

\subsection{ An application: unions of sets of uniqueness}

It was shown in \cite[Proposition 5.3]{stt} that the union of two operator $U$-sets
(resp. operator $U_1$-sets) is an operator $U$-set (resp.
an operator $U_1$-set).
A similar statement holds for operator $U_0$-sets.

\begin{proposition}\label{m0union}
Let $E_1$, $E_2\subset X\times Y$ be $\omega$-closed
operator $U_0$-sets.
Then $E_1\cup E_2$ is an operator $U_0$-set.
\end{proposition}
\begin{proof}
Let $T_\sigma$ be a pseudo-integral compact operator supported on $E_1\cup E_2$;
we may assume that the total variation of $\sigma$ is $1$.
Let $\theta_i\in \Phi(E_i)\cap{\mathfrak S}(X,Y)$, $i = 1,2$,
and write $\theta_1(x,y)=\sum_{i=1}^{\infty} f_i(x)g_i(y)$, where
$\|\sum_{i=1}^{\infty} |f_i|^2\|_\infty\leq C$ and
$\|\sum_{i=1}^{\infty} |g_i|^2\|_\infty$ $\leq$ $C$.
We have that
$\theta_1\theta_2\in\Phi(E_1\cup E_2)$ and hence
\begin{equation}\label{theta1}
0=\langle T_\sigma,\theta_1\theta_2\rangle=\langle S_{\theta_1}(T_\sigma),\theta_2\rangle.
\end{equation}

Let $\rho$ be the measure on $Y\times X$ given by
$$\rho(E) = \int_{E} \theta_1(x,y)d\sigma(y,x).$$
Denoting by $\dot{\cup}$ the union of a pairwise
disjoint family of measurable sets, we have
\begin{eqnarray*}
& & |\rho|_X(\alpha) =  |\rho|(Y\times \alpha) =
\sup\left\{\sum_{j=1}^r |\rho(E_j)| : Y\times \alpha = \dot{\cup}_{j=1}^r E_j\right\}\\
& = &
\sup\left\{\sum_{j=1}^r \left|\int_{E_j} \theta_1(x,y)d\sigma(y,x)\right| :
Y\times \alpha = \dot{\cup}_{j=1}^r E_j\right\}\\
& \leq &
\sup\left\{\sum_{j=1}^r \int_{E_j} |\theta_1(x,y)|d|\sigma|(y,x) :
Y\times \alpha = \dot{\cup}_{j=1}^r E_j\right\}\\
&\leq &
\int_{Y\times\alpha}\sum_{i=1}^{\infty} |f_i(x)||g_i(y)|d|\sigma|(y,x)\\
& \leq &
\sum_{i=1}^{\infty} \left(\int_{Y\times\alpha}|f_i(x)|^2d|\sigma|(y,x)\right)^{1/2}
\left(\int_{Y\times\alpha}|g_i(y)|^2d|\sigma|(y,x)\right)^{1/2}\\
\end{eqnarray*}
\begin{eqnarray*}
& \leq &
\left(\int_{Y\times\alpha}\sum_{i=1}^{\infty} |f_i(x)|^2d|\sigma|(y,x)\right)^{1/2}
\left(\int_{Y\times\alpha}\sum_{i=1}^{\infty} |g_i(y)|^2d|\sigma|(y,x)\right)^{1/2}
\\
&\leq & C^2|\sigma|_X(\alpha)
\end{eqnarray*}
Similarly, $|\rho|_Y(\beta)\leq C^2 |\sigma|_Y(\beta)$ showing that $\rho$ is an Arveson measure.
Now the identity
$$(S_{\theta_1}(T_\sigma)\xi,\eta)=\int_{Y\times X}\theta_1(x,y)\xi(x)\overline{\eta(y)}d\sigma(y,x), \ \
\xi\in H_1, \eta\in H_2,$$
shows that $S_{\theta_1}(T_\sigma) = T_{\rho}$.

Let $h\in \Phi(E_2)$ and write
$h = \sum_{i=1}^{\infty} f_i\otimes g_i$, where
$\sum_{i=1}^{\infty} \|f_i\|_2^2 < \infty$ and
$\sum_{i=1}^{\infty} \|g_i\|_2^2 < \infty$.
Let $X_N = \{x\in X : \sum_{i=1}^{\infty} |f_i(x)|^2 \leq N\}$ and
$Y_N = \{y\in Y : \sum_{i=1}^{\infty} |g_i(y)|^2 \leq N\}$.
Then
$\chi_{X_N\times Y_N} h\in \frak{S}(X,Y)$ and
$\|\chi_{X_N\times Y_N} h -  h\|_{\Gamma} \to_{N\to\infty} 0$.
Thus, $\Phi(E_2)\cap{\mathfrak S}(X,Y)$ is dense in $\Phi(E_2)$,
and, by (\ref{theta1}), $T_\rho\in\Phi(E_2)^\perp = \frak{M}_{\min}(E_2)$.
As $E_2$ is an operator $U_0$-set,
$T_\rho=0$ and therefore $\rho = 0$.
By Theorem \ref{th_ps},
$\langle T_\sigma, \theta_1\rangle = \rho(Y\times X) = 0$.
Since this holds for any $\theta_1\in  \Phi(E_1)\cap{\mathfrak S}(X,Y)$,
the operator $T_\sigma$ is supported on $E_1$. Since $E_1$ is an operator $U_0$-set,
$T_\sigma=0$.
\end{proof}

Proposition \ref{m0union}, \cite[Proposition 5.3]{stt}, Theorem \ref{th_iffop} and Theorem~\ref{thm0set} have the following immediate corollary.

\begin{corollary}\label{c_usm}
Let $G$ be a locally compact second countable group.
Suppose that $E_1,E_2\subseteq G$ are $U$-sets (resp., $U_1$-sets, $U_0$-sets). Then $E_1\cup E_2$ is a $U$-set (resp. a $U_1$-set, a $U_0$-set).
\end{corollary}

%In view of Theorem \ref{th_iffop}, it is natural to ask whether a similar result holds for
%unions of sets $\kappa$ for which $\frak{M}_{\max}(\kappa)$ (resp., $\frak{M}_{\min}(\kappa)$, $\frak{P}(\kappa)$)
%have zero intersection with the bimodules $\cl A$ and $\cl R$. An affirmative answer to this question
%will be obtained in Section \ref{s_unio}.

%%%%%%%%%%%%%%%%%%%%%%%%%%%%%%%%%%%%%%%%%%%%%%%%%%%%%%%%%%%%%%%%%%%%%
%%%%%%%%%%%%%%%%%%%%%%%%%%%%%%%%%%%%%%%%%%%%%%%%%%%%%%%%%%%%%%%%%%%%%

\section{Preservation properties}\label{s_pp}

The aim of this section is to show that the property of being
a set of multiplicity, or a set of uniqueness,
is preserved under some natural
operations. The section is divided into three subsections.

\subsection{Sets possessing an m-resolution}

Here we consider a certain type of a countable union of
operator $U$-sets. Theorem \ref{p_msep} should
be compared to the classical result of N. K. Barry that a countable union of
$U$-sets is a $U$-set \cite{kl}.

\begin{definition}\label{d_res}
Let $(X,\mu)$ and $(Y,\nu)$ be standard measure spaces.

(i) A pair ($\kappa_1,\kappa_2$) of $\omega$-closed subsets of the
direct product $X\times Y$ will be called \emph{m-separable} if
there exist a function $\nph_1\in \frak{S}(X,Y)$
and $\omega$-open neighbourhoods $E_1$ and $E_2$ of $\kappa_1$
and $\kappa_2$, respectively,
such that
$\nph|_{E_1} = 1$ and $\nph|_{E_2} = 0$.

(ii) Let $\kappa\subseteq X\times Y$ be an $\omega$-closed set and
$\alpha$ be a countable ordinal. We call  a family
$(\kappa_{\beta})_{\beta \leq \alpha}$ of $\omega$-closed sets an \emph{m-resolution}
of $\kappa$ if
\begin{itemize}
\item  $\kappa_1 = \kappa$.

\item $\kappa_{\beta+1} \subseteq \kappa_{\beta}$, the set $\kappa_{\beta}\setminus\kappa_{\beta+1}$ is $\omega$-closed and the pair
$\kappa_{\beta+1},\kappa_{\beta}\setminus\kappa_{\beta+1}$ is an m-separable, for every ordinal $\beta < \alpha$;

\item $\kappa_{\beta} = \cap_{\gamma < \beta} \kappa_{\gamma}$, for every limit ordinal $\beta\leq \alpha$.
\end{itemize}
\end{definition}

\begin{theorem}\label{p_msep}
Let $(X,\mu)$ and $(Y,\nu)$ be standard measure spaces
and $\kappa\subseteq X\times Y$ be an $\omega$-closed set
which possesses an m-resolution
$(\kappa_{\beta})_{\beta\leq \alpha}$
such that $\kappa_{\beta} \setminus \kappa_{\beta+1}$
is an operator $U$-set, for each $\beta < \alpha$, and $\kappa_{\alpha}$ is an operator $U$-set.
Then $\kappa$ is an operator $U$-set.
\end{theorem}
\begin{proof}
Let $(\kappa_{\beta})_{\beta \leq \alpha}$ be an m-resolution of $\kappa$
such that $\kappa_{\beta}\setminus\kappa_{\beta+1}$ is an operator $U$-set for each $\beta < \alpha$.

We first observe  that if $T\in \frak{M}_{\max}(\kappa_{\beta})\cap \cl K$
for some ordinal $\beta < \alpha$, then $T\in \frak{M}_{\max}(\kappa_{\beta + 1})\cap \cl K$. In fact, let
 $\kappa_{\beta}' = \kappa_{\beta-1}\setminus\kappa_{\beta}$.
By our assumptions, $\kappa_{\beta}'$ is  an $\omega$-closed set
and there exists a function
$\nph \in \frak{S}(X,Y)$ such that $\nph = 1$
on an $\omega$-open neighbourhood of $\kappa_{\beta+1}$ and
$\nph = 0$ on an $\omega$-open neighbourhood of $\kappa_{\beta}'$.
Clearly, $1 - \nph \in \frak{S}(X,Y)$, $1 - \nph = 0$
on an $\omega$-open neighbourhood of $\kappa_{\beta+1}$ and
$1 - \nph = 1$ on an $\omega$-open neighbourhood of $\kappa_{\beta}'$.
Moreover, $T = S_{\nph}(T) + S_{1-\nph}(T)$.

For each $\psi\in \frak{S}(X,Y)$ vanishing on an $\omega$-open neighbourhood of $\kappa_{\beta+1}$,
the function $\psi\nph\in\frak{S}(X,Y)$ vanishes on an $\omega$-open neighbourhood of $\kappa_{\beta}$
and, since $T\in \frak{M}_{\max}(\kappa_{\beta})$, Lemma \ref{l_mmax} implies that
$S_{\psi}(S_{\nph}(T)) = S_{\psi\nph}(T) = 0$. By Lemma \ref{l_mmax} again,
$S_{\nph}(T)\in \frak{M}_{\max}(\kappa_{\beta+1})$. Similarly, $S_{1-\nph}(T)\in \frak{M}_{\max}(\kappa_{\beta}')$.
Since $\cl K$ is invariant under Schur multipliers, we conclude that
$S_{1-\nph}(T)\in \frak{M}_{\max}(\kappa_{\beta}') \cap \cl K$.
However, $\kappa_{\beta}'$ is an operator $U$-set by assumption.
It follows that $S_{1-\nph}(T) = 0$ and hence $T = S_{\nph}(T) \in \frak{M}_{\max}(\kappa_{\beta+1}) \cap \cl K$.

Let now $T\in \frak{M}_{\max}(\kappa)\cap \cl K$.
It follows by transfinite induction that $T\in \frak{M}_{\max}(\kappa_\beta)\cap \cl K$ for all $\beta\leq \alpha$. In fact, assuming that the statement holds for all $\gamma<\beta$ we get by the previous paragraph that $T\in \frak{M}_{\max}(\kappa_\beta)\cap \cl K$ if $\beta$ has a predecessor while,
if $\beta$ is a limit ordinal, the inclusion follows from the assumption that
$\kappa_\beta=\cap_{\gamma<\beta}\kappa_\gamma$.

Since $\kappa_{\alpha}$
is an operator $U$-set, we have now $T = 0$ and hence
$\kappa$ is an operator $U$-set.
\end{proof}

The following corollary should be compared to
M. Bo$\dot{\rm z}$ejko's result
\cite{bozejko_pams} that every compact countable set in a
non-discrete locally compact group is a $U$-set.

\begin{corollary}\label{c_count}
Let $G$ be a non-discrete locally compact second countable group and $E\subseteq G$ be a
closed countable set. Then $E$ is a $U$-set.
\end{corollary}
\begin{proof}
Recall  that the successive Cantor-Bendixson
derivatives of the set $E$ are defined as follows: let $E_0 = E$ and  for an ordinal
$\beta$, let $E_{\beta}$ be equal to the set of all
limit points of $E_{\beta-1}$ if $\beta$ has a predecessor, and to
$\cap_{\gamma < \beta} E_{\gamma}$ if $\beta$ is a limit ordinal.
Since $E$ is countable, there exists a countable ordinal $\alpha$
such that $E_{\alpha} = \emptyset$.
Moreover, $E_{\beta}\setminus E_{\beta+1}$ is a countable set consisting of
isolated points of $E$. By the regularity of $A(G)$, a pair of the form  $(\{s\}^*, F^*)$, where
$F$ is a closed set and $s\not\in F$, is m-separable. One hence easily obtains an m-resolution
for $E^*$.
On the other hand, if $G$ is not discrete then
${\mathfrak M}_{\max}(\{s\}^*) = \lambda_s\cl D$ does not contain non-zero compact operators.
It follows from Theorem \ref{p_msep} and Theorem
\ref{th_iffop} that $E$ is a $U$-set.
\end{proof}

%\noindent {\bf Remark. }
%Ternary masa-bimodules are those masa-bimodules,
%whose support have the form
%$\{(x,y)\in X\times Y : f(x) = g(y)\}$ for some measurable functions
%$f : X\to \bb{R}$ and $g : Y\to \bb{R}$.
%It is not difficult to see that the multiplicity free ternary masa-bimodules
%(that is, the ones whose non-degenerate part is not a bimodule over
%any von Neumann algebras bigger than the masas)
%have supports that are graphs of partially defined null set preserving measurable
%bijections (see \cite{t_spsyn}).
%In view of the proof of Corollary
%\ref{c_count}, it is natural to consider the support of a multiplicity free ternary
%masa-bimodule as the analogue of a singleton in the class of
%$\omega$-closed sets. The $\omega$-closed sets $\kappa$
%which possess an m-resolution $(\kappa_{\beta})_{\beta\leq
%\alpha}$ such that $\kappa_{\beta}\setminus \kappa_{\beta+1}$ is
%the support of a multiplicity free ternary masa-bimodule for every $\beta <
%\alpha$ can thus be considered as the operator analogue of a countable closed set of
%a locally compact group.

%{\bf Mozhet byt', ubrat' eto zamechanie? Vozmozhno, ono esche budet poleznym, pri podhodyaschej razrabotke, no v takom vide ono kazhetsya syrovatym dlya stat'i}.

\subsection{Inverse images}

In this subsection, we establish an Inverse Image Theorem for sets of multiplicity.
Our result, Theorem \ref{th_invim}, should be compared to
\cite[Theorem 4.7]{st1}, an inverse image result for
operator synthesis.

Let $(X,\mu)$, $(X_1,\mu_1)$, $(Y,\nu)$ and $(Y_1,\nu_1)$ be standard measure spaces.
We fix, for the remainder of this section, measurable mappings
$\nph : X\rightarrow X_1$ and $\psi : Y\rightarrow Y_1$  such that
$\nph(X)$ and $\psi(Y)$ are measurable,
the measure $\nph_*\mu$ on $X_1$ given by $\nph_*\mu(\alpha_1) = \mu(\nph^{-1}(\alpha_1))$
is absolutely continuous with respect to $\mu_1$, and the measure $\psi_*\nu$, defined similarly, is absolutely
continuous with respect to $\nu_1$.

Let $r : X_1\rightarrow \bb{R}^+$ be the Radon-Nikodym derivative of $\nph_*\mu$
with respect to $\mu_1$, that is, the $\mu_1$-measurable function such that
$\mu(\nph^{-1}(\alpha_1)) = \int_{\alpha_1} r(x_1)d\mu_1(x_1)$
for every measurable set $\alpha_1\subseteq X_1$.
Similarly, let $s : Y_1\rightarrow \bb{R}^+$
be the Radon-Nikodym derivative of $\psi_*\nu$ with respect to $\nu_1$.
Let $M_1 = \{x_1\in X_1 : r(x_1) = 0\}$ and $N_1  = \{y_1\in Y_1 : s(y_1) = 0\}$.
Note that
$\mu(\nph^{-1}(M_1)) = \int_{M_1} r(x_1)d\mu_1(x_1) = 0$.
Similarly, $\nu(\psi^{-1}(N_1)) = 0$.
Observe that, up to a $\mu_1$-null set, $M_1^c\subseteq \nph(X)$.
Indeed,
letting $M_2 = M_1^c \cap \nph(X)^c \subseteq X_1$, we see that $\nph^{-1}(M_2) = \emptyset$ and hence
$0 = \mu(\nph^{-1}(M_2)) = \int_{M_2} r_1(x) d\mu_1(x)$. Since $r_1(x) > 0$ for every $x\in M_2$,
we have that $\mu_1(M_2) = 0$. Similarly, $N_1^c\subseteq \nph(Y)$, up to a null set.

%{\bf Ya by skazal, chto sleduyuschij paragraf nemnozhko sbivaet chitatelya: kazhetsya, chto my ushli ot uzhe vvedennyh ob'ektov i obsuzhdaem chto-to obschee (naprimer $\phi$ -- eto "a" function, a potom v teoreme vse okazyvaetsya v nashih terminah). Luchshe chut' inache:}
We will say that
$\nph : X\to X_1$ is injective up to a null set and
has measurable inverse if there exists a subset $\Lambda\subseteq X$
with $\mu(\Lambda) = 0$, such that $\nph : \Lambda^c \to X_1$
is injective, $\varphi(\Lambda^c)$ is $\mu_1$-measurable and
$\nph^{-1}|_{\nph(\Lambda^c)}$ is a measurable function.
 If moreover  $\Lambda$ can be chosen so that $\mu_1(\nph(\Lambda^c)^c) = 0$, then we say that $\nph$ is bijective up to a null set. The following result must be known but we could not find a precise reference.

\begin{lemma}\label{l_vpi}
The operator $V_{\nph} : L^2(X_1,\mu_1)\rightarrow L^2(X,\mu)$ given by
$$V_{\nph}\xi (x) =
\begin{cases}
\frac{\xi(\nph(x))}{\sqrt{r(\nph(x))}} \;&\text{if }x\not\in \nph^{-1}(M_1),\\
0 &\text{if }x\in \nph^{-1}(M_1)
\end{cases}$$
is a partial isometry with initial space $L^2(M_1^c,\mu_1|_{M_1^c})$.
Moreover, if $\nph$ is injective up to a null set and has
measurable inverse then $V_{\nph}$ is surjective.
\end{lemma}
\begin{proof}
Note that, if $\xi \in L^2(X_1,\mu_1)$ then
\begin{eqnarray*}
\|V_{\nph}\xi\|^2 & = & \int_{\nph^{-1}(M_1)^c}\left|\frac{\xi(\nph(x))}{\sqrt{r(\nph(x))}}\right|^2 d\mu(x)
 =  \int_{M_1^c}\left|\frac{\xi(x_1)}{\sqrt{r(x_1)}}\right|^2 d\nph_*\mu(x_1)\\
& = & \int_{M_1^c}r(x_1)\left|\frac{\xi(x_1)}{\sqrt{r(x_1)}}\right|^2 d\mu_1(x_1)
= \int_{M_1^c} |\xi(x_1)|^2 d\mu_1(x_1).
\end{eqnarray*}
It follows that $V_{\nph}$ is a partial isometry with initial space $L^2(M_1^c,\mu_1|_{M_1^c})$.

Suppose that there exists a set $M\subseteq X$ such that $\mu(M) = 0$, $\nph(M^c)$ is measurable,
$\nph|_{M^c}$ is one-to-one, and $\nph^{-1} : \nph(M^c)\rightarrow M^c$ is measurable.
Let $\eta\in L^2(X,\mu)$ and define $\xi : X_1\rightarrow \bb{C}$ by setting
$\xi(x_1) = \sqrt{r(x_1)}\eta(\nph^{-1}(x_1))$ if $x_1\in \nph(M^c)$ and
$\xi(x_1) = 0$ if $x_1\not\in \nph(M^c)$. We claim that $\xi \in L^2(X_1,\mu_1)$. To
see this, note that
$$\mu(\nph^{-1}(\alpha_1)) = \int_{\alpha_1} r(x_1) d\mu_1(x_1),$$
for all $\mu_1$-measurable sets $\alpha_1\subseteq \nph(M^c)$.
Setting $\tilde{\mu}$ to be the measure on $M^c$ given by
$\tilde{\mu}(\alpha) = \mu_1(\nph(\alpha))$ for $\mu$-measurable subset $\alpha\subseteq M^c$
we have
$$\mu(\alpha) = \int_{\alpha} r(\nph(x)) d\tilde{\mu}(x).$$
It follows that
\begin{eqnarray*}
\|\xi\|_{L^2(X_1,\mu_1)} & = & \int_{\nph(M^c)} r(x_1)|\eta(\nph^{-1}(x_1))|^2d\mu_1(x_1)\\
& = & \int_{M^c} r(\nph(x))|\eta(x)|^2 d\tilde{\mu}(x)\\
& = & \int_{M^c} |\eta(x)|^2 d\mu(x) = \|\eta\|_{L^2(X,\mu)}
\end{eqnarray*}
since $M$ is $\mu$-null.
On the other hand,
$$\Lambda \stackrel{def}{=} \nph^{-1}(M_1^c \cap \nph(M^c)^c) \subseteq \nph^{-1}(\nph(M^c)^c)\subseteq M,$$
and hence $\Lambda$ is $\mu$-null. It follows as in the third paragraph of the present subsection that
$M_1^c \cap \nph(M^c)^c$ is $\mu_1$-null.
Thus, $V_{\nph}\xi = \eta$, and the proof is complete.
\end{proof}

We recall some facts from \cite{st1} that will be needed subsequently.
If $\kappa\subseteq X\times Y$ is $\omega$-closed, a \emph{$\kappa$-pair} is
an element
$$(P,Q)\in (\cl B(\ell^2)\bar\otimes L^\infty(X,\mu))\times  (\cl B(\ell^2)\bar\otimes L^\infty(Y,\nu))$$
such that, after the identification of $P$ and $Q$ with operator-valued
weakly measurable functions, defined on $X$ and $Y$, respectively,
$P$ and $Q$ take values that are projections and
$P(x)Q(y)=0$ marginally almost everywhere on $\kappa$.
A $\kappa$-pair is called simple if $P$ and $Q$ take finitely many values. The following was established in \cite{st1}.

\begin{theorem}\label{th_kappapair}
Let $\kappa\subseteq X\times Y$ be an $\omega$-closed set. Then
$$\frak{M}_{\min}(\kappa) = \{T\in \cl B(H_1,H_2) : Q(I\otimes T)P = 0, \ \forall \ \kappa\mbox{-pair }
(P,Q)\}$$
and
$$\frak{M}_{\max}(\kappa) = \{T\in \cl B(H_1,H_2) : Q(I\otimes T)P = 0, \ \forall \mbox{  simple }\kappa\mbox{-pair }
(P,Q)\}.$$
\end{theorem}

We now formulate and prove the main result of this subsection.

\begin{theorem}\label{th_invim}
Let $\nph : X \to X_1$ and $\psi : Y\to Y_1$
be measurable functions.
Let $\kappa_1\subseteq X_1\times Y_1$ and $\kappa = \{(x,y)\in X\times Y : (\nph(x),\psi(y))\in \kappa_1\}$.

(i) \ \
Suppose that $\nph$ and $\psi$ are injective up to a null set and have measurable inverses. If $\kappa_1$ is a an operator $U$-set (resp. an operator $U_1$-set)
then $\kappa$ is a an operator $U$-set (resp. an operator $U_1$-set).

(ii) \ Suppose that $\kappa_1\subseteq M_1^c\times N_1^c$.
If $\kappa_1$ is an operator $M$-set
(resp. an operator $M_1$-set) then $\kappa$ is an operator $M$-set (resp. an operator $M_1$-set).

(iii) Suppose that $\mu_1$ (resp. $\nu_1$) is equivalent to
$\nph_*\mu$ (resp. $\psi_*\nu$) and that $\nph$ and $\psi$ are
injective up to a null set and have measurable inverses. Then $\kappa_1$ is an
operator $M$-set (resp. an operator $M_1$-set) if and only if $\kappa$ is an
operator $M$-set (resp. an operator $M_1$-set).
\end{theorem}
\begin{proof}
(i) Let $\Theta$ be the linear map from the algebraic tensor product
$L^2(X_1,$ $\mu_1)\otimes L^2(Y_1,\nu_1)$
of $L^2(X_1,\mu_1)$ and $L^2(Y_1,\nu_1)$
sending $f\otimes g$ to $V_{\nph}f \otimes V_{\psi}g$. Since $V_{\nph}$ and $V_{\psi}$ are partial isometries,
$\Theta$ is contractive in the norm of $\Gamma(X_1,Y_1)$, and hence extends to a contractive
linear map  $\Theta : \Gamma(X_1,Y_1)\rightarrow \Gamma(X,Y)$.
By Lemma \ref{l_vpi}, $V_{\nph}$ and $V_{\psi}$ are surjective, and hence $\Theta$ has dense range.
Moreover, if $h\in \Gamma(X_1,Y_1)$ then
\begin{equation}\label{eq_eft}
\Theta(h)(x,y) = \frac{h(\nph(x),\psi(y))}{\sqrt{r(\nph(x)) s(\psi(y))}}, \
\mbox{ for m.a.e. } (x,y)\in \nph^{-1}(M_1)\times \psi^{-1}(N_1).
\end{equation}
To show (\ref{eq_eft}), write $h = \sum_{i=1}^{\infty} f_i\otimes g_i$,
where $\sum_{i=1}^{\infty} \|f_i\|_2^2 < \infty$ and
$\sum_{i=1}^{\infty} \|g_i\|_2^2$ $< \infty$, and
set $h_n = \sum_{i=1}^n f_i\otimes g_i$, $n\in \bb{N}$.
By the definition of $\Theta$,
identity (\ref{eq_eft}) holds for all $h_n$, $n\in \bb{N}$, if $(x,y)$
belongs to the set $\nph^{-1}(M_1)\times \psi^{-1}(N_1)$.
By \cite[Lemma 2.1]{st1}, there exists a subsequence $(h_{n_k})$ of $(h_n)$ which converges to
$h$ marginally almost everywhere. By passing to a further subsequence, we may assume that
$\Theta(h_{n_k})$ converges to $\Theta(h)$ marginally almost everywhere.
Identity (\ref{eq_eft}) now follows from the fact that if $E\subseteq X_1\times Y_1$
is marginally null then $\{(x,y)\in X\times Y : (\nph(x),\psi(y))\in E\}$ is marginally null.

We claim that $\Theta$ is weak* continuous. To show this, let
$(h_i)\subseteq \Gamma(X_1,Y_1)$ be a bounded net such that $h_i\rightarrow h$ in the weak* topology,
for some $h\in \Gamma(X_1,Y_1)$.
Suppose that $f\in L^2(X,\mu)$ and $g\in L^2(Y,\nu)$, and identify $f\otimes g$ with its corresponding rank one operator in $\cl B(L^2(X,\mu),L^2(Y,\nu))$.
Let $f_1 : X_1\to \bb{C}$ (resp. $g_1 : Y_1\to \bb{C}$) be the function given by
$f_1(x_1) = f(\nph^{-1}(x_1))$ (resp. $g_1(y_1) = g(\psi^{-1}(y_1))$) if $x_1\in \nph(\Lambda^c)$
(resp. $y_1\in \psi(M^c)$) and $f_1(x_1) = 0$ (resp. $g(y_1) = 0$) if
$x_1\not\in \nph(\Lambda^c)$ (resp. $y_1\not\in \psi(M^c)$), $\Lambda$, $M$ being null sets from the definition of injectivity up to null sets of $\nph$ and $\psi$, respectively.
Using (\ref{eq_eft}) and the facts that
$\nph_*\mu(M_1) = 0$ and $\psi_*\nu (N_1) = 0$,
we have
\begin{eqnarray*}
\langle \Theta(h_i), f\otimes g\rangle & = &
\int_{X\times Y} \frac{h_i(\nph(x),\psi(y))}{\sqrt{r(\nph(x))}\sqrt{s(\psi(y))}}f(x) g(y) d\mu\times\nu (x,y)\\
& = &
\int_{X_1\times Y_1} \frac{h_i}{\sqrt{r\otimes s}}(f_1\otimes g_1) d\nph_*\mu \times \psi_*\nu\\
& = &
\int_{X_1\times Y_1} \sqrt{r\otimes s} (f_1\otimes g_1) h_i d\mu_1\times \nu_1.
\end{eqnarray*}
The function
$\sqrt{r\otimes s} (f_1\otimes g_1)$ belongs to
$L^2(X_1,\mu_1)\otimes L^2(Y_1,\nu_1)$
(see the proof of Lemma \ref{l_vpi}).
Since $h_i\to^{w^*} h$,
the last integrals converge to
$$\int_{X_1\times Y_1} \sqrt{r\otimes s} (f_1\otimes g_1) h d\mu_1\times \nu_1.$$
Thus,
$$\langle \Theta(h_i), f\otimes g\rangle \rightarrow \langle \Theta(h), f\otimes g\rangle.$$
It follows by the boundedness of $(h_i)$ that
$$\langle \Theta(h_i), K\rangle \rightarrow \langle \Theta(h), K\rangle$$ for every compact operator $K$, that is,
$\Theta(h_i)\rightarrow \Theta(h)$ in the weak* topology.
Using the Krein-Shmulyan theorem, we obtain that $\Theta$ is weak* continuous.
Hence, if $\cl M_1\subseteq \Gamma(X_1,Y_1)$, then
\begin{equation}\label{eq_thetain}
\Theta(\overline{\cl M_1}^{w^*}) \subseteq \overline{\Theta(\cl M_1)}^{w^*}.
\end{equation}

Suppose that $h|_{\kappa_1} = 0$. If $(x,y)\in \kappa \setminus ((\nph^{-1}(M_1)\times Y)\cup
(X\times \psi^{-1}(N_1)))$ then, by (\ref{eq_eft}),
$$\Theta(h)(x,y) = \frac{h(\nph(x),\psi(y))}{\sqrt{r(\nph(x)) s(\psi(y))}}  = 0;$$
thus, $\Theta(\Phi(\kappa_1)) \subseteq \Phi(\kappa)$.
On the other hand, if $E_1$ is an $\omega$-open neighbourhood of $\kappa_1$ then
$(\nph\times\psi)^{-1}(E_1)$ is an $\omega$-open neighbourhood of $\kappa$.
Applying the same reasoning as above, and using the
continuity of $\Theta$ with respect to $\|\cdot\|_{\Gamma}$,
we conclude that $\Theta(\Psi(\kappa_1))\subseteq \Psi(\kappa)$.

Now suppose that $\kappa_1$ is an operator $U$-set, that is, $\overline{\Phi(\kappa_1)}^{w^*} = \Gamma(X_1,Y_1)$.
Using (\ref{eq_thetain}), we have
$$\Gamma(X,Y) = \overline{\Theta(\Gamma(X_1,Y_1))}^{\|\cdot\|} =
\overline{\Theta(\overline{\Phi(\kappa_1)}^{w^*})}^{\|\cdot\|}
\subseteq \overline{\Theta(\Phi(\kappa_1))}^{w^*}
\subseteq \overline{\Phi(\kappa)}^{w^*}.$$
Thus, $\overline{\Phi(\kappa)}^{w^*} = \Gamma(X,Y)$ and hence $\kappa$
is an operator $U$-set. It follows similarly that if $\kappa_1$
is an operator $U_1$-set then $\kappa$ is an operator $U_1$-set.

\smallskip

(ii)
Suppose that $\kappa_1$ is an operator $M_1$-set and
let $K_1$ be a non-zero compact operator in ${\mathfrak M}_{\min}(\kappa_1)$.
Let $K = V_{\psi}K_1 V_{\nph}^*$. As $\kappa_1\subseteq M_1^c \times N_1^c$,
$V_\nph^*V_\nph = P(M_1^c)$ and $V_\psi^*V_\psi = P(N_1^c)$,
we have that $K_1 = V_{\psi}^*KV_{\nph}$ and hence $K$ is a non-zero compact operator.

Let
$$(P,Q)\in (\cl B(\ell^2)\bar\otimes L^\infty(X,\mu))\times  (\cl B(\ell^2)\bar\otimes L^\infty(Y,\nu))$$ be a
$\kappa$-pair \cite{st1}; this means that,
after the identification of $P$ and $Q$ with operator-valued
weakly measurable functions, defined on $X$ and $Y$, respectively,
$P$ and $Q$ are projection-valued and
$P(x)Q(y)=0$ marginally almost everywhere on $\kappa$.
It follows from the proof of \cite[Theorem~4.7]{st1} that there exists a $\kappa_1$-pair
$$(\hat P,\hat Q)\in (\cl B(\ell^2)\bar\otimes L^\infty(X_1,\mu))\times  (\cl B(\ell^2)\bar\otimes L^\infty(Y_1,\nu)),$$
such that
$P(x)\leq \hat P(\nph(x))$ and $Q(y)\leq \hat Q(\psi(x))$ for almost all $x\in X$ and
almost all $y\in Y$.
By Theorem \ref{th_kappapair},
\begin{equation}\label{eq_qkp}
\hat Q(I\otimes K_1)\hat P=0.
\end{equation}
We claim that
\begin{equation}\label{eq_exch}
(I\otimes V_\nph^*)(R\circ\nph) =
R(I\otimes V_\nph^*) \mbox{ and } (S\circ\psi) (I\otimes V_\psi)=
(I\otimes V_\psi)S,
\end{equation}
whenever $R$ and $S$ are bounded operator-valued weakly measurable
functions on $X_1$ and $Y_1$, respectively.
It clearly suffices to show only the first of these identities.
Start by observing that
$P(\nph^{-1}(\alpha))V_\nph=V_\nph P(\alpha)$,
for all measurable subsets $\alpha\subseteq X_1$.
It follows that
(\ref{eq_exch}) holds when $R = \sum_{j=1}^k a_j\otimes \chi_{E_j}$,
where $(E_j)_{j=1}^k$ is a family of pairwise disjoint
measurable subsets of $X_1$ and $(a_i)_{i=1}^k$
is a family of bounded operators  on $\ell^2$.
If $R$ is arbitrary then, by Kaplansky's Density Theorem,
it is the strong limit of a sequence $(R_n)_{n\in \bb{N}}$, where $R_n$ is of the latter form and
$\|R_n\|\leq \|R\|$ for each $n$. By the proof of \cite[Theorem~4.6]{st1}, there exists
$S_1\subseteq X_1$ with $\mu_1(S_1)=0$ such that
$$R(x_1)=\mbox{s-lim}_{n\to\infty}R_{n_k}(x_1) \text{ if }x_1\notin S_1.$$
Let $S = \nph^{-1}(S_1)$; then
$R(\nph(x))=\mbox{s-lim}_{k\to\infty} R_{n_k}(\nph(x))$ if $x\notin S$.
As $\mu(S)=\nph_*\mu(S_1)$ and $\nph_*\mu$ is absolutely continuous with respect to $\mu_1$,
we have that $\mu(S)=0$ and hence $(R_{n_k}\circ \nph)_{k\in \bb{N}}$
converges almost everywhere to $R\circ \nph$.

Since $\|R_n\|=\esssup_{x_1\in X_1}\|R_n(x)\|_{\cl B(\ell^2)}$
and $(R_n)_{n\in \bb{N}}$ is bounded by $\|R\|$, there exists a $\mu_1$-null set $M\subseteq X_1$
such that
$\|R_n(x_1)\|_{\cl B(\ell^2)}\leq \|R\|$ for all $x_1\notin M$ and all $n\in\mathbb N$.
Therefore $\|R_n(\varphi(x))\|_{\cl B(\ell^2)}\leq \|R\|$ for all
$x\notin\nph^{-1}(M)$ and all $n\in \bb{N}$.
As $\mu(\nph^{-1}(M))=0$, we have
that $\|R_n\circ \varphi\|=\esssup_{x\in X}\|R_n(\varphi(x))\|_{\cl B(\ell^2)}\leq \|R\|$,
for all $n\in\bb{N}$.
By a straightforward application of the Lebesgue Dominated Convergence Theorem,
$(R_{n_k}\circ\nph)_{k\in \bb{N}}$ converges strongly to $R\circ \nph$.
As $(I\otimes V_\nph^*)(R_n\circ\nph) = R_n(I\otimes V_\nph^*)$ holds for every $n$ and
$(I\otimes V_\nph^*)(R_n\circ\nph)\to (I\otimes V_\nph^*)(R\circ\nph)$,
$R_n(I\otimes V_\nph^*) \to R(I\otimes V_\nph^*)$ in the strong operator topology,
(\ref{eq_exch}) is proved.
Using (\ref{eq_qkp}) and (\ref{eq_exch}), we now obtain
\begin{eqnarray*}
Q(I\otimes K)P & = & Q(\hat Q\circ\psi)(I\otimes V_\psi K_1V_\nph^*)(\hat P\circ\nph)P\\
& = & Q(I\otimes V_\psi)\hat Q(I\otimes  K_1)\hat P(I\otimes V_\psi^*)P = 0.
\end{eqnarray*}
By Theorem \ref{th_kappapair}, $K\in {\mathfrak M}_{\min}(E)$;
hence, $\kappa$ is an operator $M_1$-set.

Now suppose that $\kappa_1$ is an operator $M$-set and let
$K_1\in \frak{M}_{\max}(\kappa_1)$ be a non-zero compact operator.
Let $(P,Q)$ be a simple $\kappa$-pair \cite{st1}, that is, a $\kappa$-pair
$(P,Q)$ for which each of the projection valued functions
$P$ and $Q$ takes finitely many values.
We recall the construction of the pair $(\hat{P},\hat{Q})$ from \cite{st1}.
Let $(\xi_j)_{j\in \bb{N}}$ be a dense sequence in $\ell^2$.
It was shown on \cite[p. 311]{st1} that
there are null sets $M_0^1\subseteq X_1$ and $M_0\subseteq X$ and,
for each $j\in \bb{N}$,  a measurable function
$g_j : \nph(X)\setminus M_0^1 \rightarrow X$
with $\nph(g_j(x_1)) = x_1$ for all $x_1\in \nph(X)\setminus M_0^1$, and
$(P(g_j(\nph(x))\xi_j,\xi_j) > (P(x)\xi_j,\xi_j) - \frac{1}{j}$, $x\in X\setminus M_0$.
Let, similarly,
$(\eta_j)_{j\in \bb{N}}$ be a dense sequence in $\ell^2$ and
for each $j\in \bb{N}$, let $h_j : \nph(Y)\setminus N_0^1 \rightarrow Y$ be a measurable function
with $\psi(h_j(y_1)) = y_1$ for all $y_1\in \psi(Y)\setminus N_0^1$, and
$(Q(h_j(\psi(y))\eta_j,\eta_j) > (Q(y)\eta_j,\eta_j) - \frac{1}{j}$, $y\in Y\setminus N_0$,
where $N_0^1\subseteq Y_1$ and $N_0\subseteq Y$ are null sets.
Set
$$\hat{P}_n(x_1) = \vee_{j=1}^n P(g_j(x_1)), \ \ \hat{P}(x_1) = \vee_{j=1}^\infty P(g_j(x_1)), \ \ x_1 \in \nph(X)\setminus M_0^1,$$
$$\hat{Q}_n(y_1) = \vee_{j=1}^n Q(h_j(y_1)), \ \ \hat{Q}(y_1) = \vee_{j=1}^\infty Q(h_j(y_1)), \ \ y_1 \in \psi(Y)\setminus N_0^1.$$
We have that $\hat{P}_n\rightarrow_{n\rightarrow \infty} \hat{P}$ and
$\hat{Q}_n\rightarrow_{n\rightarrow \infty} \hat{Q}$ in the strong operator topology.
Furthermore, since $P$ (reps. $Q$) takes only finitely many values, the same is true for
$\hat{P}_n$ (resp. $\hat{Q}_n$), $n\in \bb{N}$.
If
$$(x_1,y_1)\in \kappa_1\cap (((\nph(X)\setminus M_0^1)\times ((\psi(Y)\setminus N_0^1)))$$
then $(g_j(x_1),h_j(y_1))\in \kappa$. However, $\kappa_1\subseteq M_1^c\times N_1^c$, while
$M_1^c\times N_1^c$ is marginally contained in $\nph(X)\times\psi(Y)$.
It follows that
$\hat{P}_n(x_1) \hat{Q}_n(y_1) = 0$ for marginally almost all $(x_1,y_1)\in \kappa_1$
and every $n\in \bb{N}$.
Thus, $(\hat{P}_n, \hat{Q}_n)$ is a simple $\kappa_1$-pair, $n\in \bb{N}$,
and hence, by Theorem \ref{th_kappapair}, $\hat{Q}_n (I\otimes K_1) \hat{P}_n = 0$ for every $n$.
Since $P\leq \hat{P}\circ \nph$ and $Q\leq \hat{Q}\circ \psi$, it follows from
(\ref{eq_exch}) and the first part of the proof that
$$P = P (\hat{P}\circ \nph) = \mbox{s-lim}_{n\rightarrow \infty} P (\hat{P}_n\circ \nph) \mbox{ and }
Q = Q (\hat{Q}\circ \psi) = \mbox{s-lim}_{n\rightarrow \infty} Q (\hat{Q}_n\circ \psi).$$
As in the previous paragraph, we conclude that
$$Q(I\otimes K)P  = \mbox{w-lim}_{n\rightarrow\infty} Q(I\otimes V_\psi)\hat{Q}_n(I\otimes  K_1)\hat{P}_n(I\otimes V_\psi^*)P = 0.$$
By Theorem \ref{th_kappapair}, $K\in \frak{M}_{\max}(\kappa)$ and since
$K$ is a non-zero compact operator, $\kappa$ is an operator $M$-set.

\smallskip

(iii) In this case $\mu_1(M_1) = 0$ and $\nu_1(N_1) = 0$;
thus, (iii) is immediate from (i) and (ii).
\end{proof}

\begin{remark}\label{elef} {\rm {\bf (i)}  The statement in Theorem \ref{th_invim} (ii) does not hold
without the assumption $\kappa_1\subseteq M_1^c\times N_1^c$; indeed,
assuming that $M_1$ and $N_1$ are non-null and
letting $\kappa_1 = M_1\times N_1$, we see that $\kappa_1$ is an operator $M$-set;
but $\kappa$ is marginally equivalent to the empty set and hence is an
operator $U$-set.

\smallskip

\noindent {\bf (ii)}  G. Eleftherakis has recently proved part (i) of Theorem \ref{th_invim}
without the injectivity assumption on the mappings $\nph$ and $\psi$, see \cite{george}.}
\end{remark}

\begin{corollary}\label{c_iig}
Let $G$ and $H$ be locally compact second countable groups with Haar measures
$m_G$ and $m_H$, respectively,  $\nph:G\to H$ be a
continuous homomorphism and $E$ be a closed subset of $H$.
Assume that $\nph_*m_G$ is absolutely continuous with respect to $m_H$.

(i) \ \
Suppose that $\nph$ is injective and has a continuous inverse on $\nph(G)$. If $E$ is a U-set (resp. a $U_1$-set) then
$\nph^{-1}(E)$ is a U-set (resp. a $U_1$-set).

(ii) \ Suppose that $\nph_*m_G$ is equivalent to $m_H$. If $E$ is an M-set (resp. an $M_1$-set) then $\nph^{-1}(E)$ is an M-set (resp. an $M_1$-set).

(iii) If $\nph$ is an isomorphism then $E$ is an M-set (resp. an $M_1$-set) if and only if $\nph^{-1}(E)$ is an M-set (resp. an $M_1$-set).
\end{corollary}
\begin{proof}
First observe that, since  $\nph$ is a homomorphism,
$\nph^{-1}(E)^*=(\nph\times\nph)^{-1}(E^*)$. If $\nph$ is an isomorphism then $\nph_*m_G$ is equivalent to $m_H$, see Remark \ref{r_suff}.
The corollary  now follows from Theorems~\ref{th_iffop} and \ref{th_invim}.
\end{proof}

\begin{remark}\label{r_suff}
\rm
We note that if $m_H(\nph(G))\ne 0$, $\nph$ is injective and proper
({\it i.e.}, the preimage of each compact set is compact) and
has continuous inverse on $\nph(G)$, when $\nph(G)$ is given
the relative topology, then the measure
$\nph_*m_G$ is absolutely continuous with respect to $m_H$.
If $\nph$ is an isomorphism, then the two measures are in fact equivalent.

In fact, as $G$ is second countable, it is $\sigma$-compact and hence there exist compact sets
$K_n\subseteq G$, $K_n\subseteq K_{n+1}$, $n\in \bb{N}$, such that $G=\cup_{n=1}^{\infty} K_n$.
Since $\nph$ is continuous, $\nph(K_n)$ is compact for each $n$ and hence
$\nph(G)=\cup_{n=1}^{\infty} \nph(K_n)$ is $\sigma$-compact.
Assume that $W$ is a Borel subset of $\nph(G)$ and $s\in\nph(G)$, say $s=\nph(t)$. Then
\begin{eqnarray*}
\nph_*m_G(sW) & = & m_G(\nph^{-1}(sW))=m_G(t\nph^{-1}(W))\\
& = & m_G(\nph^{-1}(W))=\nph_*m_G(W).
\end{eqnarray*}
Hence $\nph_*m_G$ is  left invariant  when restricted to the subgroup $\nph(G)$ of $H$. As $\nph$ is proper, $\nph_*m_G$ is finite on any compact subset of $\nph(G)$. Moreover, $\nph_*m_G$ is regular on $\nph(G)$:
for a measurable $W\subseteq\nph(G)$, we have
\begin{eqnarray*}\nph_*m_G(W)&=&m_G(\nph^{-1}(W))=\inf\{m_G(U):\nph^{-1}(W)\subseteq U, U\text{ is open}\}\\
&=&\inf\{m_G(U):\nph^{-1}(W)\subseteq\nph^{-1}(\nph(U)),  U\text{ is open}\}\\&=&
\inf\{\nph_*m_G(\nph(U)):W\subseteq\nph(U),  U\text{ is open}\}\\
&=&\inf\{\nph_*m_G(V) : W\subseteq V,  V\text{ is open in }\nph(G)\}
\end{eqnarray*}
(since $\nph$ has continuous inverse on $\nph(G)$,
the set $\nph(U)$ is open in the relative topology of $\nph(G)$ and any open set $V$ in $\nph(G)$ is the image of an open set in $G$),
{\it i.e.}, $\nph_*m_G$ is outer regular.
In a similar way one shows that it is inner regular.
Hence $\nph_*m_G$ satisfies the conditions of a Haar measure on $\nph(G)$.
Since the same holds true for the restriction of $m_H$ to $\nph(G)$, and $m_H(\nph(G))\ne 0$,
there exists $c > 0$ such that $\nph_*m_G = cm_H$.
Let $W\subseteq H$ be a any Borel subset of $H$. Then
$$\nph_*m_G(W)=\nph_*m_G(W\cap\nph(G))=cm_H(W\cap\nph(G))=c\int_W\chi_{\nph(G)}(x)dm_H(x),$$ giving the claim.
It follows that the measures $\nph_*m_G$ and $m_H$ are equivalent in the case
$\nph$ is an isomorphism.
\end{remark}

\subsection{Direct products}

In this subsection, we show that direct products
preserve the property of being an operator $M$-set (resp. an operator
$M_1$-set, an operator $M_0$-set).
En route, we establish two additional results which
we believe are interesting on their own right. Namely, we show that a tensor product formula holds
for the minimal masa-bimodules, answering in this way affirmatively a question posed by J. Froelich in \cite{f}.
Simultaneously, we show that the minimal masa-bimodule $\frak{M}_{\min}(\kappa)$ associated with an
$\omega$-closed set $\kappa$ is the closure of all pseudo-integral operators with symbols supported on
$\kappa$; this provides an alternative, \lq\lq synthetic'' description of $\frak{M}_{\min}(\kappa)$ in
measure-theoretic terms, similar to the topological one given originally by Arveson in \cite{arveson}.

Let $(X,\mu)$ and $(Y,\nu)$ are standard measure spaces.
Recall that $\cl F$ denotes the product $\sigma$-algebra on $Y\times X$.

%{\bf V dok-ve lemmy perehod ot vtoroj strochki k tret'ej mne neponyaten. Mozhet, luchshe dvigat'sya po $F$ -- snachala dlya pryamougol'nika, potom dlya ob'edineniya pryamougol'nikov, a predel'nyj perehod tozhe dolzhen poluchat'sya.}

\begin{lemma}\label{l_resam}
If $\sigma\in {\mathbb A}(Y,X)$ and $E\in \cl F$ then the measure $\sigma_E$ given by
$\sigma_E(F) = \sigma (E\cap F)$, $F\in \cl F$,
belongs to ${\mathbb A}(Y,X)$.
\end{lemma}
\begin{proof}
Let $\sigma\in {\mathbb A}(Y,X)$ and $E\in \cl F$.
If $\alpha\subseteq X$ is measurable then,
denoting by $\dot{\cup}$ the union of a family of pairwise disjoint
measurable sets, we have
\begin{eqnarray*}
|\sigma_E|_X (\alpha) & = & |\sigma_E|(Y\times\alpha) =
\sup \left\{\sum_{i=1}^k |\sigma_E(F_i)| : \dot{\cup}_{i=1}^k F_i = Y\times\alpha\right\}\\
& = & \sup \left\{\sum_{i=1}^k |\sigma(E\cap F_i)| : \dot{\cup}_{i=1}^k F_i = Y\times\alpha\right\}\\
& \leq & \sup \left\{\sum_{i=1}^k |\sigma|(E\cap F_i) : \dot{\cup}_{i=1}^k F_i = Y\times\alpha\right\}\\
& \leq & \sup \left\{\sum_{i=1}^k |\sigma|(F_i) : \dot{\cup}_{i=1}^k F_i = Y\times\alpha\right\} =
|\sigma|(Y\times\alpha)\\ & = & |\sigma|_X(\alpha).
\end{eqnarray*}
One shows similarly that $|\sigma_E|_Y \leq |\sigma|_Y$; it now follows that
$\sigma_E\in \bb{A}(Y,X)$.
\end{proof}

\begin{theorem}\label{th_pse}
Let $\kappa\subseteq X\times Y$ be an $\omega$-closed set. Then
$$\frak{M}_{\min}(\kappa) = \overline{\{T_{\sigma} : \sigma\in \bb{A}(Y,X), \supp\sigma\subseteq \hat{\kappa}\}}^{w^*}.$$
\end{theorem}
\begin{proof}
Let $\frak{M}_0(\kappa)$ denote the right hand side of the identity.
We first show that
$\frak{M}_0(\kappa)$ is a weak* closed masa-bimodule.
Since $T_{\sigma} + T_{\nu} = T_{\sigma + \nu}$, we have that $\frak{M}_0(\kappa)$ is a (weak*) closed subspace
of $\cl B(H_1,H_2)$.
It is moreover easy to check that if $\nph\in L^{\infty}(X,\mu)$ and $\psi\in L^{\infty}(Y,\nu)$ then
$M_{\psi}T_{\sigma}M_{\nph} = T_{\sigma'}$, where $\sigma'\in \bb{A}(Y,X)$ is given by
$$\sigma'(E) = \int_{Y\times X} \psi(y)\nph(x) d\sigma(y,x).$$
If $\sigma$ is supported on $\hat{\kappa}$ then clearly so is $\sigma'$;
hence, $\frak{M}_0(\kappa)$ is a masa-bimodule.

We next claim that $\supp\frak{M}_0(\kappa) = \kappa$.
Suppose that $\alpha\times\beta$ is a rectangle of finite measure
such that $P(\beta)T_{\sigma}P(\alpha) = 0$ for all $\sigma\in \bb{A}(Y,X)$
with $\supp\sigma \subseteq \hat{\kappa}$.
Let $\tau\in \bb{A}(X,Y)$ be arbitrary, and $\tau_{\hat{\kappa}}$ be the measure
defined as in Lemma \ref{l_resam}. Then  $\supp \tau_{\hat\kappa}\subseteq \hat{\kappa}$ and hence
$$\tau((\beta\times\alpha)\cap\hat{\kappa}) = \tau_{\hat\kappa}((\beta\times\alpha)\cap\hat{\kappa})
= (P(\beta)T_{\tau_{\hat\kappa}}P(\alpha)\chi_{\alpha},\chi_{\beta}) = 0.$$
By Arveson's Null Set Theorem \cite[Theorem 1.4.3]{arveson},
$(\beta\times\alpha)\cap\hat{\kappa}\simeq \emptyset$.
It follows that $\kappa$ is contained in the support of $\frak{M}_0(\kappa)$; on the other hand,
by Theorem \ref{th_ps}, $\supp\frak{M}_0(\kappa)\subseteq \kappa$, up to a marginally null set.
It follows that
$\kappa \simeq \supp\frak{M}_0(\kappa)$.

Thus 
$\frak{M}_{\min}(\kappa)\subseteq \frak{M}_0(\kappa)$. To show the converse inclusion, it suffices, by \cite[Theorem 4.4]{st1}, to show that if a function $h\in \Gamma(X,Y)$ vanishes on $\kappa$ then $\langle T_{\sigma},h\rangle = 0$  for each $\sigma\in \bb{A}(Y,X)$ supported by $\hat\kappa$. But this follows from the equality $\langle T_{\sigma},h\rangle = \int_{Y\times X} \hat h d\sigma$ (see Theorem~\ref{th_ps}).
\end{proof}

\begin{corollary}\label{c_mminal}
Let $\kappa\subseteq X\times X$ be an $\omega$-closed set such that $\frak{M}_{\max}(\kappa)$
is a unital algebra. Then $\frak{M}_{\min}(\kappa)$ is a (unital) algebra.
\end{corollary}
\begin{proof}
It was shown in \cite{arveson} that the set of all pseudo-integral operators
is an algebra. Since $\frak{M}_{\max}(\kappa)$ is an algebra,
the set $\frak{M}_0(\kappa)$ of all pseudo-integral operators in
$\frak{M}_{\max}(\kappa)$ is also an algebra.
Hence its weak* closure $\overline{\frak{M}_0(\kappa)}^{w^*}$ is also an algebra.
By Theorems \ref{th_ps} and \ref{th_pse}, $\frak{M}_{\min}(\kappa) = \overline{\frak{M}_0(\kappa)}^{w^*}$
and the proof is complete.
\end{proof}

The next theorem establishes a tensor product formula for the minimal masa-bimodules.
Let $(X_i,\mu_i)$ and $(Y_i,\nu_i)$ be
standard measure spaces, $i= 1,2$,
and consider the flip
$$\rho : (X_1\times Y_1) \times (X_2\times Y_2) \rightarrow (X_1\times X_2) \times (Y_1\times
Y_2)$$ given by
$$\rho((x_1,y_1),(x_2,y_2)) = ((x_1,x_2),(y_1,y_2)).$$
Below, for two weak* closed subspaces $\cl U$ and $\cl V$ of operators,
we denote by $\cl U\bar{\otimes}\cl V$ the
weak* closed subspace generated by the elementary tensors
$A\otimes B$ where $A\in \cl U$ and $B\in \cl V$.

\begin{theorem}\label{l_tmin}
Let $(X_i,\mu_i)$ and $(Y_i,\nu_i)$ be standard measure spaces and
$\kappa_i\subseteq X_i\times Y_i$ be $\omega$-closed sets, $i = 1,2$. Then
\begin{equation}\label{eq_eqq}
\frak{M}_{\min}(\kappa_1)\bar{\otimes} \frak{M}_{\min}(\kappa_2) =
\frak{M}_{\min}(\rho(\kappa_1\times\kappa_2)).
\end{equation}
\end{theorem}
\begin{proof}
We first note that, by \cite{m},
\begin{equation}\label{eq_tpfm}
\supp(\frak{M}_{\min}(\kappa_1)\bar{\otimes} \frak{M}_{\min}(\kappa_2))\simeq \rho(\kappa_1\times\kappa_2).
\end{equation}
By the minimality property of $\frak{M}_{\min}(\rho(\kappa_1\times\kappa_2))$ we have that
$$\frak{M}_{\min}(\rho(\kappa_1\times\kappa_2))\subseteq \frak{M}_{\min}(\kappa_1)\bar{\otimes} \frak{M}_{\min}(\kappa_2).$$

To see the reverse inclusion, it is enough prove that if $m\in \mathbb A(Y_1,X_1)$ and $n\in{\mathbb A}(Y_2,X_2)$ then $T_m\otimes T_n=T_\sigma$ for some measure $\sigma\in\mathbb A(Y_1\times Y_2,X_1\times X_2)$.
Indeed, by (\ref{eq_tpfm}),
$\supp T_\sigma\subseteq \rho(\kappa_1\times\kappa_2)$ and hence Theorem \ref{th_ps}
implies that $\supp\sigma\subseteq \widehat{\rho(\kappa_1\times\kappa_2)}$.
By Theorem \ref{th_pse}, $T_\sigma\in {\mathfrak M} _{\min}(\rho(\kappa_1\times\kappa_2))$.

Let $$\sigma(E)=\int_{Y_2\times X_2}\int_{Y_1\times X_1}\chi_{E}(y,x)dm(y_1,x_1)dn(y_2,x_2)$$ for every measurable $E\subseteq (Y_1\times Y_2)\times(X_1\times X_2)$.
If $\beta_i\subseteq Y_i$, $i = 1,2$, are measurable then
\begin{eqnarray*}
&&|\sigma|((\beta_1\times \beta_2)\times (X_1\times X_2))\\
&&\leq \int_{Y_2\times X_2}\int_{Y_1\times X_1}\chi_{(\beta_1\times \beta_2)\times (X_1\times X_2)}(y,x)d|m|(y_1,x_1)d|n|(y_2,x_2)\\
&&= |m|(\beta_1\times X_1)|n|(\beta_2\times X_2)\\
&& \leq \|m\|_{\bb{A}}\|n\|_{\bb{A}} \nu_1(\beta_1)\nu_2(\beta_2)
= \|m\|_{\bb{A}} \|n\|_{\bb{A}} (\nu_1\times\nu_2)(\beta_1\times\beta_2).
\end{eqnarray*}
It now easily follows that
$|\sigma|(F\times(X_1\times X_2))\leq  \|m\|_{\bb{A}} \|n\|_{\bb{A}} (\nu_1\times\nu_2)(F)$,
for any element $F$ in the product $\sigma$-algebra on $Y_1\times Y_2$.
Similar arguments show that
$|\sigma|((Y_1\times Y_2)\times E)\leq \|m\|_{\bb{A}} \|n\|_{\bb{A}} (\mu_1\times\mu_2)(E)$, for
every measurable $E\subseteq X_1\times X_2$. Hence $\sigma$ is an Arveson measure and $T_\sigma$ is a bounded operator from $L^2(X_1\times X_2)$ to $L^2(Y_1\times Y_2)$.

If $f_i\in L^2(X_i,\mu_i)$, $g_i\in L^2(Y_i,\nu_i)$, $i=1,2$, we have
\begin{eqnarray*}
&&((T_m\otimes T_n) f_1\otimes f_2,g_1\otimes g_2)\\
&&=\int_{Y_2\times X_2}\int_{Y_1\times X_1}f_1(x_1)f_2(x_2)\overline{g_1(y_1)g_2(y_2)}dm(y_1,x_1)dn(y_2,x_2)\\
&&=
\int_{(Y_1\times Y_2)\times(X_1\times X_2)} (f_1\otimes f_2)(x)\overline{(g_1\otimes g_2)(y)}d\sigma(y,x).
\end{eqnarray*}
and hence $T_m\otimes T_n=T_\sigma$, proving the statement.

\end{proof}

%{\bf Pochemu eto lemma, a ne teorema? Ona imeet samostoyatel'noe znachenie}

\begin{theorem}\label{produ}
Let $(X_i,\mu_i)$ and $(Y_i,\nu_i)$ be standard measure spaces and
$\kappa_i\subseteq X_i\times Y_i$ be $\omega$-closed sets, $i = 1,2$.
The set $\rho(\kappa_1\times\kappa_2)$ is an operator
$M$-set (resp. operator $M_1$-set) if and only if both
$\kappa_1$ and $\kappa_2$ are operator $M$-sets (resp. operator $M_1$-sets).
\end{theorem}
\begin{proof}
By \cite{m}, the support of
$\frak{M}_{\max}(\kappa_1)\bar{\otimes} \frak{M}_{\max}(\kappa_2)$ is
$\rho(\kappa_1\times\kappa_2)$. It follows that
\begin{equation}\label{eq_t1}
\frak{M}_{\max}(\kappa_1)\bar{\otimes} \frak{M}_{\max}(\kappa_2) \subseteq \frak{M}_{\max}(\rho(\kappa_1\times\kappa_2)).
\end{equation}
Assume first that $\kappa_1$ and $\kappa_2$ are operator $M_1$-sets (resp. operator $M$ sets).
Suppose that $T_i$ is a non-zero compact operator in
$\frak{M}_{\min}(\kappa_i)$ (resp. $\frak{M}_{\max}$ $(\kappa_i)$), $i
= 1,2$. By Theorem \ref{l_tmin} (resp. by (\ref{eq_t1})),
$T_1\otimes T_2$ is a non-zero compact operator in
$\frak{M}_{\min}(\rho(\kappa_1\times\kappa_2))$ (resp.
$\frak{M}_{\max}(\rho(\kappa_1\times\kappa_2))$). Hence $\rho(\kappa_1\times\kappa_2)$ is an
operator $M_1$-set (resp. an operator $M$-set).

We next show that if either $\kappa_1$ or $\kappa_2$ is an operator $U$-set then so is $\rho(\kappa_1\times\kappa_2)$.
Suppose that $T\in \cl K(H_1\otimes H_2,K_1\otimes K_2)$ is
supported on $\rho(\kappa_1\times\kappa_2)$.
Let $\omega\in (\cl K(H_2,K_2))^*=\cl C_1(K_2,H_2)$ and let $L_\omega$ be the slice map from  $\cl K(H_1\otimes H_2,K_1\otimes K_2)$ to
$\cl K(H_1,K_1)$ defined on elementary tensors by $L_\omega(A\otimes B) = \omega(B)A$.
Then $\supp L_\omega(T)\subseteq\kappa_1$.
In fact, if $\alpha\times\beta$ is a measurable rectangle marginally disjoint from
$\kappa_1$, then
$((\alpha\times X_2)\times(\beta\times Y_2))\cap  \rho(\kappa_1\times\kappa_2)\simeq\emptyset$ and
$$P(\beta)L_\omega(T)P(\alpha) = L_\omega((P(\beta)\otimes I)T(P(\alpha)\otimes I) = 0.$$
If $\kappa_1$ is an operator $U$-set, $L_\omega(T) = 0$ for all $\omega$ and hence $T = 0$.

If  $T\in \cl K(H_1\otimes H_2,K_1\otimes K_2)\cap
\frak{M}_{\min}(\rho(\kappa_1\times\kappa_2))$ and ($P$, $Q$) is a $\kappa_1$-pair, then
($P\otimes I$, $Q\otimes I$) is a $\rho(\kappa_1\times\kappa_2)$-pair and hence
$$Q(I_{\ell^2}\otimes L_\omega(T))P =
(\id\otimes L_\omega)((Q\otimes I)(I_{\ell^2}\otimes T)(P\otimes I)) =0;$$
by Theorem \ref{th_kappapair}, $L_\omega(T)\in \frak{M}_{\min}(\kappa_1)$.
If $\kappa_1$ is an operator $U_1$-set, arguments similar to the ones above show that
$T = 0$ and hence $\rho(\kappa_1\times\kappa_2)$ is an operator $U_1$-set.
\end{proof}

\begin{corollary}\label{c_produ}
Let $G_1$ and $G_2$ be locally compact second countable groups and $E_1\subseteq G_1$,
$E_2\subseteq G_2$ be closed sets. If $E_1$ and $E_2$ are $M$-sets
(resp. $M_1$-sets) then $E_1\times E_2$ is an $M$-set (resp. an
$M_1$-set).
\end{corollary}
\begin{proof}
Suppose that $E_1\subseteq G_1$ and $E_2\subseteq G_2$ are $M$-sets. By
Theorem \ref{th_iffop}, $E_1^*$ and $E_2^*$ are operator $M$-sets, and by Theorem \ref{produ},
$\rho(E_1^*\times E_2^*) = (E_1\times E_2)^*$ is an operator $M$-set. By Theorem \ref{th_iffop} again,
$E_1\times E_2$ is an $M$-set. A similar argument applies to $M_1$-sets.
\end{proof}

\section{Sets of finite width}\label{s_fw}

Let $(X,\mu)$ and $(Y,\nu)$ be standard measure spaces.
A subset $E\subseteq X\times Y$ is called
a \emph{set of finite width} if
there exist measurable functions
$f_i : X\to {\mathbb R}$, $g_i : Y\to {\mathbb R}$, $i=1,\ldots,n$,
such that
\begin{equation}\label{eq_fw}
E = \{(x,y)\in X\times Y : f_i(x)\leq g_i(y), i=1,\ldots,n\};
\end{equation}
the \emph{width} of $E$ is the smallest $n$
for which $E$ can be represented in the form (\ref{eq_fw}).
By \cite[Theorem~4.8]{st1} and \cite[Theorem 2.1]{t_spsyn}, any such set is operator synthetic.
In this section we identify those sets of finite width which are operator
$M_1$-sets, and hence operator $M$-sets.

We first assume that the measures $\mu$ and $\nu$ are finite
and the standard measure spaces $X$ and $Y$ arise from compact topologies.
A \emph{system} is a finite set $D$  of disjoint rectangles $\Pi = \alpha\times \beta$,
where $\alpha\subseteq X$ and $\beta\subseteq Y$ are measurable.
Set $r(\alpha\times \beta) = \min\{\mu(\alpha),\nu(\beta)\}$.
The \emph{volume} 
of a system $D = \{\Pi_j: 1\le j\le J\}$
is the number $r(D) \stackrel{def}{=} \max_{1\leq j\leq J} r(\Pi_j)$.
Let $U_D = \cup_{j=1}^J\Pi_j$ and call the systems $D_1$ and $D_2$
disjoint if $U_{D_1}\cap U_{D_2} = \emptyset$;
in this case, denote by $D_1\vee D_2$ their union.

With each system $D=\{\alpha_j\times \beta_j: 1\le j\le J\}$,
we associate the projection $\pi_D$ on ${\cl B}(H_1,H_2)$ by setting
$$\pi_D(T) = \sum_{j=1}^J P(\beta_j)TP(\alpha_j), \ \ \ T\in \cl B(H_1,H_2).$$
It is easy to see that $\pi_D$ depends only on $U_D$ and that
$\pi_{D_1\vee D_2} = \pi_{D_1}+\pi_{D_2}$; thus,
the mapping $U\to \pi_U$ is a projection-valued measure
on the algebra of sets generated by all rectangles.
Note that the range of $\pi_D$ coincides with $\frak{M}_{\max}(U_D)$.

A system $D=\{\alpha_j\times \beta_j: 1\le j\le J\}$ will be called \emph{diagonal} if
$\alpha_i\cap \alpha_j = \beta_i\cap \beta_j = \emptyset$ whenever $i\neq j$.
The system $D$ will be called \emph{$n$-diagonal},
if $D = D_1\vee D_2\vee\dots \vee D_n$ where $D_1,\dots,D_n$ are diagonal systems.
It is easy to see that $\|\pi_D\| = 1$ if $D$ is diagonal.
Hence, $\|\pi_D\|\le n$ if $D$ is $n$-diagonal.

\begin{lemma}\label{seq}
Let $(D^k)_{k\in \mathbb{N}}$ be a sequence of
$n$-diagonal systems such that $r(D^k)\to_{k\to\infty} 0$.
Then $\|\pi_{D^k}(T)\|\to_{k\to\infty} 0$ for each compact operator $T$.
\end{lemma}
\begin{proof}
It suffices to prove the statement for rank one operators
$T = u\otimes v$ where $u,v$ are bounded functions on $X$ and $Y$, because the set of all linear combinations of such operators is dense in ${\cl K}(H_1,H_2)$
and the sequence $(\pi_{D^k})_{k\in \mathbb{N}}$ is uniformly bounded.

If $D=\{\alpha_j\times \beta_j\}_{j=1}^J$ is a diagonal system, then for $T=u\otimes v$, we have
\begin{eqnarray*}
\|\pi_D(T)\|&\leq &\|\pi_D(T)\|_2=
\left\|\sum_{j=1}^J(\chi_{\alpha_j}\otimes\chi_{\beta_j})(u\otimes v)\right\|_{L^2(X\times Y,\mu\times\nu)}\\
&\leq & \|u\|_\infty\|v\|_\infty\left(\sum_{j=1}^J \mu(\alpha_j)\nu(\beta_j)\right)^{1/2}\\
&\leq & \|u\|_\infty\|v\|_\infty\left(\sum_{j=1}^J r(\Pi_j)(\mu(\alpha_j)+\nu(\beta_j)\right)^{1/2}\\
&\leq & \|u\|_\infty\|v\|_\infty r(D)^{1/2}(\mu(X)+\nu(Y))^{1/2}.
\end{eqnarray*}
It follows that if $D$ is an $n$-diagonal system then
$$\|\pi_D(T)\|\leq n\|u\|_\infty\|v\|_\infty r(D)^{1/2}(\mu(X)+\nu(Y))^{1/2}.$$
Hence  $\|\pi_{D^k}(T)\|\to_{k\to\infty} 0$ whenever $r(D^k)\to_{k\to\infty} 0$.
\end{proof}

Let us call a set $E$ \emph{$n$-quasi-diagonal}
if for each $\varepsilon>0$ there is an $n$-diagonal system $D$ with $E\subseteq U_D$ and $r(D)<\varepsilon$.

We say that a (measurable) function defined on
a measure space is non-atomic if it is not constant on any set of positive measure.

\begin{lemma}\label{1-quas}
Let $f : X\to {\mathbb R}$, $g : Y\to {\mathbb R}$ be Borel maps and assume that
$f$ is non-atomic. Then the set
$$E_{f,g} = \{(x,y)\in X\times Y : f(x) = g(y)\}$$
is 1-quasi-diagonal.
\end{lemma}
\begin{proof}
Let $\mu_f$ be the measure on the Borel $\sigma$-algebra of
$\bb{R}$ given by $\mu_f(C) = \mu(f^{-1}(C))$. By our assumption, $\mu_f$ is non-atomic and finite. Hence, for every
$\varepsilon > 0$,
there exists a partition ${\mathbb R} = \cup_{j=1}^NC_j$ with $\mu_f(C_j)<\varepsilon/\nu(Y)$ for all $j$.
In fact,  letting $g(x)=\mu_f(-\infty,x])$ we have that $g$ is a bounded increasing function such that $g({\mathbb R})\subseteq [0,C]$, where $C=\mu_f({\mathbb R})$. As $\mu_f$ is non-atomic, $g$ is continuous and
$(0,C)\subseteq g({\mathbb R})$.
Let $0 = a_0<a_1<\ldots<a_{N+1} = C$ be a partition of $[0,C]$ such that
$a_{i+1}-a_i < \varepsilon/\nu(Y)$, $0\leq i\leq N$,
and $g(x_i) = a_i$, $1\leq i\leq N$.
Set $C_0 = (0,x_1]$, $C_i = (x_i,x_{i+1}]$ if $0<i<N$,
and $C_N = (x_{N},\infty)$. Then $\mathbb R=\cup_{i=1}^NC_i$ and $\mu_f(C_i)<\varepsilon/\nu(Y)$,
$1\leq i \leq N$.

Setting $\alpha_j = f^{-1}(C_j)$, $\beta_j = g^{-1}(C_j)$ and $D = \{\alpha_j\times \beta_j: 1\le j\le N\}$, we
now see that $D$ is diagonal, $E\subseteq U_D$ and $r(D) < \varepsilon$.
\end{proof}

Fix $T\in {\cl B}(H_1,H_2)$, $F\in \cl C_1(H_2,H_1)$ and set
$$\varphi(\Pi) = \langle\pi_{\Pi}(T),F\rangle,$$
for each rectangle $\Pi\subseteq X\times Y$.
We say that $\Pi$ is \emph{$\varphi$-null}, if $\varphi(\Pi^{\prime}) = 0$
for all rectangles $\Pi^{\prime}\subseteq \Pi$.

\begin{lemma}\label{phi-zero}
If $\Pi = \cup_{j=1}^{\infty}\Pi_j$ and each $\Pi_j$ is $\varphi$-null then $\Pi$ is $\varphi$-null.
\end{lemma}
\begin{proof}
It suffices to show that $\varphi(\Pi)=0$. Without loss of generality we may assume that all $\Pi_j$ are mutually disjoint.

By Lemma \ref{e-compactness}, for each $\varepsilon$, there are $X_{\varepsilon}\subseteq X$ and $Y_{\varepsilon}\subseteq Y$ such that $\mu(X\setminus X_{\varepsilon})<\varepsilon$, $\nu(Y\setminus Y_{\varepsilon})<\varepsilon$ and the rectangle $\Pi^{\varepsilon} = \Pi\cap (X_{\varepsilon}\times  Y_{\varepsilon})$ is covered by a finite number of recangles $\Pi_j$, say, $\Pi^{\varepsilon}\subseteq \cup_{j=1}^m\Pi_j$.
Set $\Pi_j^{\epsilon} = \Pi_j \cap (X_{\epsilon} \times Y_{\epsilon})$; we have
$$\varphi(\Pi^{\varepsilon}) = \sum_{j=1}^m \varphi(\Pi_j^{\varepsilon}) = 0.$$
On the other hand,
if $\Pi = \alpha\times \beta$ then
$\varphi(\Pi^{\varepsilon}) = \langle P(Y_{\varepsilon})P(\beta)TP(\alpha)P(X_{\varepsilon}), F\rangle$ and,
since $P(X_{\varepsilon})\to I$, $P(Y_{\varepsilon})\to I$ in the strong operator topology,
we conclude that
$\lim_{\varepsilon\to 0}\varphi(\Pi^{\varepsilon}) = \varphi(\Pi)$.
Thus, $\varphi(\Pi)=0$ and the proof is complete.
\end{proof}

\begin{theorem}\label{comp-synth}
If $E$ is a set of finite width then
$\frak{M}_{\max}(E)\cap \cl K$ coincides with
the norm-closure $\frak{M}_0(E)$ of the subspace of
$\frak{M}_{\max}(E)$ generated by its rank one operators.
\end{theorem}
\begin{proof}
We may assume that the measures $\mu$ and $\nu$ are finite and
the standard spaces $X$ and $Y$ arise from compact topologies. Indeed,
if this is not the case, write $X = \cup_{n=1}^{\infty} X_n$ and $Y = \cup_{n=1}^{\infty} Y_n$
as increasing unions, where $X_n$ and $Y_n$ are compact, $\mu(X_n) < \infty$ and
$\nu(Y_n) < \infty$. Then $P(X_n)\to_{n\to \infty} I$ and $P(Y_n)\to_{n\to \infty} I$
in the strong operator topology. If $T\in \frak{M}_{\max}(E)\cap \cl K$ then
$P(Y_n)TP(X_n)\to_{n\to \infty} T$ in norm, and hence we may restrict our attention to
each of $E\cap (X_n\times Y_n)$, which is a set of finite width when considered
as a subset of $X_n\times Y_n$.

We use induction on the width $n$ of $E$.
With the convention that all measurable rectangles are sets
of width zero, the statement clearly holds for $n=0$.
Suppose that the assertion of the theorem is true for sets of width
smaller than $n$, and let
$$E = \{(x,y)\in X\times Y : f_j(x) \leq g_j(y), j = 1,\dots,n\},$$
where
$f_j : X\to {\mathbb R}$ and $g_j : Y\to {\mathbb R}$ are measurable functions, $j = 1,\ldots,n$.
Let $F\in \Gamma(X,Y)$ be in the annihilator of
$\frak{M}_0(E)$.  We need to show that $\langle T,F\rangle = 0$ for each compact
operator $T\in \frak{M}_{\max}(E)$.
Assume first that all $f_j$, $j = 1,\dots,n$,
are non-atomic.  By Lemma \ref{1-quas}, the sets
$$E_j = \{(x,y): f_j(x) = g_j(y)\}, \ \ \ j = 1,\dots,n,$$
are 1-quasi-diagonal and hence their union
$\cup_{j=1}^nE_j$ is $n$-quasi-diagonal. Let $E^{\prime} = E\cap (\cup_{j=1}^nE_j)$;
then $E^{\prime}$ is $n$-quasi-diagonal and
$$E^{\prime\prime} \stackrel{def}{=} E\setminus E^{\prime} = \{(x,y): f_j(x)<g_j(y), j = 1,\dots,n\}$$
is $\omega$-open.

Let $D$ be an $n$-diagonal system with $E^{\prime}\subseteq U_D$.
If $\Pi$ is a rectangle, disjoint from $U_D$, then
$\Pi\subseteq E^{\prime\prime}\cup E^c$; since both
$E^{\prime\prime}$ and $E^c$ are $\omega$-open,
$\Pi = \cup_{i=1}^{\infty}\Pi_i$ where each of $\Pi_i$ is either a subset of $E^{\prime\prime}$ or of $E^c$.

Let, as above, $\varphi(\alpha\times \beta) = \langle P(\beta)TP(\alpha),F\rangle$,
where $\alpha\subseteq X$ and $\beta\subseteq Y$ are measurable.
If $\Pi_i\subseteq E^c$ and $\Pi_i'\subseteq \Pi$ then $\varphi(\Pi_i') =0$
by the fact that $T$ is supported on $E$.

On the other hand,
if $\Pi_i = \alpha_i\times \beta_i\subseteq E^{\prime\prime}$ then,
clearly, $\Pi_i\subseteq E$ whence
$P(\beta_i)TP(\alpha_i) \in \frak{M}_{0}(\Pi_i)\subseteq \frak{M}_{0}(E)$.
It follows that $\varphi(\Pi_i) =0$. The same argument shows that
$\varphi(\Pi_i') =0$ whenever $\Pi_i'$ is a rectangle with
$\Pi_i'\subseteq \Pi_i$, and hence $\Pi_i$ is $\varphi$-null.
By Lemma \ref{phi-zero}, $\Pi$ is $\varphi$-null.
We thus showed that every rectangle disjoint from $U_D$ is $\varphi$-null.

Let $\widetilde{D} = \{\Pi^{\prime}_k: 1\le k\le m\}$ be a system
such that $(U_{D})^c = U_{\widetilde{D}}$. It follows from the previous paragraphs that
$$\langle\pi_{\widetilde{D}}(T),F\rangle = \sum_{k=1}^m \varphi(\Pi^{\prime}_k) = 0.$$
Hence
$$\langle T,F\rangle = \langle\pi_D(T),F\rangle+ \langle\pi_{\widetilde{D}}(T),F\rangle = \langle\pi_D(T),F\rangle$$
and $|\langle T,F\rangle| \le \|F\|\|\pi_D(T) \|$. Since $E'$ is $n$-quasi-diagonal, there exists a sequence $(D^k)_{k\in \bb{N}}$ of $n$-diagonal systems
such that $E'\subseteq U_{D^k}$ for each $k$ and $r(D^k)\to_{k\to\infty} 0$.
By Lemma \ref{seq}, $\|\pi_{D^k}(T)\|\to_{k\to\infty} 0$ and thus $\langle T,F\rangle = 0$.

Now let $f_j$ be arbitrary. Then we can write $X$ as a disjoint
union $\cup_{k=0}^{\omega}X_k$, $\omega\le \infty$, where  $X_0$ is a subset of $X$ such that all $f_j$ are non-atomic on $X_0$ and for each $k>0$  at least one of the functions $f_j$ is constant on $X_k$.

Set $P_k = P(X_k)$, $F_k(x,y) = \chi_{X_k}(x)F(x,y)$ and $T_k = TP_k$;
then $\langle T,F\rangle = \sum_{k=0}^{\omega} \langle T_k,F_k\rangle$ and it hence suffices
to show that $\langle T_k,F_k\rangle = 0$ for each $k$. It is clear that
$T_k$ is supported on $E_k \stackrel{def}{=} E\cap(X_k\times Y)$ and $F_k$ annihilates $\frak{M}_0(E_k)$.

The equality $\langle T_0,F_0\rangle  = 0$ follows from the
first part of the proof.
Let $k>0$, and suppose, for example, that the function $f_1$ is constant on $X_k$:
$f_1(x) = a$, for $x\in X_k$. Set $Y_k = \{y\in Y: g_1(y)\ge a\}$. Then
$$E_k = \{(x,y)\in X_k\times Y_k : f_j(x)\le g_j(y), j = 2,\dots,n\}.$$
Thus $E_k$ is a set of width at most $n-1$, when considered as a subset of $X_k\times Y_k$.
Since $T$ is supported on $E_k$, we have $T_k = P(Y_k)T_k$. Moreover,
$\chi_{X_k\times Y_k} F_k$ annihilates $\frak{M}_0(E_k)$ and hence
$$\langle T_k,F_k\rangle = \langle P(Y_k)T_k,\chi_{X_k\times Y_k} F_k\rangle = 0$$
by the inductive assumption.
\end{proof}

\begin{corollary}\label{fin-wid}
Let $E$ be a set of finite width.
The following conditions are equivalent:

(i) \ \ $E$ is an operator $U$-set;

(ii) \ $E$ does not support a non-zero Hilbert-Schmidt operator;

(iii) $\mu\times\nu(E) = 0$;

(iv) \ $E$ does not support a  non-zero nuclear operator;

(v) \ \ $E$ does not contain a rectangle of non-zero measure.
\end{corollary}
\begin{proof}
We may assume that $\mu$ and $\nu$ are finite, for if
$X=\cup_{k=1}^\infty X_k$, $Y=\cup_{k=1}^\infty Y_k$, where $(X_k)_{k=1}^\infty$ and
$(Y_k)_{k=1}^\infty$ are increasing sequences of subsets of finite measure and
$T\in{\cl B}(H_1,H_2)$ is a non-zero compact operator supported in $E$, then
so is $P(Y_k)TP(X_k)$ for some $k$.

(i)$\Rightarrow$(ii) is trivial.

(ii)$\Rightarrow$(iii)
If $\mu\times\nu(E)$ were non-zero, then $T_k$, where $k(x,y)=\chi_E(x,y)$, would be  a non-zero Hilbert-Schmidt operator supported in $E$.

(iii)$\Rightarrow$(iv)
If $E$ supports a non-zero nuclear operator then by
\cite[Theorem 6.7]{eks}, $E$ supports a non-zero rank one
operator $u\otimes v$, $u\in L^2(X,\mu)$, $v\in L^2(Y,\nu)$.
As $u\otimes v$ is supported on $\supp u\times\supp v$, we have
$\mu\times\nu(E)\ne 0$, a contradiction.

(iv)$\Rightarrow$(v) If $E$ contains a non-zero rectangle $\alpha\times \beta$ then
$\chi_{\alpha}\otimes\chi_{\beta}$ is a non-zero nuclear operator supported in $E$, a contradiction.

(v)$\Rightarrow$(i) If $E$ supports a non-zero compact operator then it follows from  Theorem~\ref{comp-synth} that there exists a non-zero rank one operator  $u\otimes v$ supported in $E$. But then $\supp u\times\supp v$ is a non-zero rectangle contained in $E$, a contradiction.
\end{proof}

\noindent {\bf Remark } We note that the conditions from Corollary \ref{fin-wid}
are also equivalent to the set $E$ being a $U_1$-set, as well as to $E$ being a $U_0$-set.

\medskip

We have the following immediate corollary.

\begin{corollary}
A non-zero bounded operator from $L^2(\mathbb{R}^n)$ to $L^2(\mathbb{R}^m)$
cannot be compact if it is supported on a manifold of the form
$y_j = \phi(x_1,\dots,x_n)$, for some measurable function
$\phi : \bb{R}^n\to \bb{R}$ and some $j = 1,\dots,m$,
or on a set that can be partitioned into finitely many such sets.

In particular, the support of a non-zero compact operator from
$L^2(\mathbb{R}^n)$ to $L^2(\mathbb{R}^1)$ is not contained in a smooth
manifold of dimension strictly less than $n+1$.
\end{corollary}
\begin{proof}
Assume, without loss of generality, that $j = 1$. Let
$\psi : \bb{R}^m\to \bb{R}$ be given by $\psi(y_1,\ldots, y_m) = y_1$ and
$E=\{(x_1,\ldots x_n,y_1,\ldots, y_m)\in\mathbb R^n\times\mathbb R^m: y_1=\phi(x_1,\ldots, x_n)\}=\{(x,y)\in\mathbb R^n\times\mathbb R^m: \psi(y)=\phi(x)\}$.
As $\psi$ is non-atomic, Lemma \ref{1-quas} implies that
$E$ is $1$-diagonal. By Lemma~\ref{seq}, $E$ does not support a non-zero compact operator.
By \cite[Proposition 5.3]{stt} there is no non-zero compact operator supported
on a set that can be partitioned into finitely many sets of the form
$\{(x_1,\ldots,x_n,y_1,\ldots, y_m) : y_j = \phi(x_1,\dots,x_n)\}$.
\end{proof}

\begin{corollary}
Let $H$ be a closed subgroup of a locally compact second countable group $G$.
Then $H$ is an $M$-set if and only if $m(H)>0$.
\end{corollary}
\begin{proof}
If $m(H)>0$ then is an $M$-set by Remark~\ref{r_posme}.
Assume now that $m(H)=0$. By Theorem~\ref{th_iffop}, it suffices to see that
$$H^*=\{((s,t): ts^{-1}\in H\}=\{(s,t): Ht=Hs\}$$
is an operator $U$-set.
Let $f:G\to G/H$, $f(s)=Hs$.
Then $H^*=f^{-1}\times f^{-1}(D)$, where $D=\{(z,z):z\in G/H\}$. By \cite[5.22, 8.14]{hewittI}
$G/H$ is a metrisable $\sigma$-compact space.
The measure $f_*m$ on $G/H$ is non-atomic as $f_*m(\{aH\}) = m(aH) = m(H) = 0$, for every $a\in G$.
Since every bounded operator on $L^2(G/H,f_*m)$
supported on $D$ is a multiplication operator, the only compact operator on
$L^2(G/H,f_*m)$ supported on $D$ is the zero operator.  Therefore $D$ is an operator $U$-set and, by Remark~\ref{elef} (2), so is $H^*$.
 \end{proof}

If $\omega: G \rightarrow \bb{R}^+$ is a continuous homomorphism,
let $E_{\omega,t} = \{t \in G : \omega(t) \leq t\}$; we call the subsets of $G$ of this form
\emph{level sets} (see \cite{et}).

\begin{corollary}\label{c_fwha}
Let $E_j\subseteq G$, $j = 1,\dots,n$, be level sets.
The set $E \stackrel{def}{=} \cap_{j=1}^n E_j$ is an $M$-set if and only if $m(E) > 0$.
\end{corollary}
\begin{proof}
It is straightforward to check that if $F\subseteq G$ is a level set then
$F^*$ is a set of width one.
It follows that $E^*$ is a set of finite width. By Theorem \ref{th_iffop} and
Corollary \ref{fin-wid}, $E$ is an $M$-set if and only if
$(m\times m)(E^*) > 0$.
This condition is equivalent to $m(E) > 0$ by
the identity
$$
(m\times m)(E^*) = \int_G m(Et) dt = \int_G \Delta(t) m(E) dt,
$$
where $\Delta$ is the modular function.
\end{proof}

\section{Closable multipliers on group C*-algebras}\label{s_cmga}

Let $G$ be a locally compact group equipped with left Haar measure $m$
and $\psi : G\to \bb{C}$ be a measurable function.
It is well-known \cite{bf}, \cite{j}
(see also \cite[Theorem 8.3]{pisier}) that
pointwise multiplication on $L^1(G)$ by the function
$\psi$ defines a completely bounded map on $C^*_r(G)$ if and only if
the function $N(\psi)$ is a Schur multiplier.
In this section, we prove a version of this result for closable maps (see Theorem \ref{th_duc}).

Let
$$D(\psi) = \{f\in L^1(G) : \psi f\in L^1(G)\};$$
it is easy to see that
the operator $f\rightarrow \psi f$, $f\in D(\psi)$, viewed as
a densely defined operator on $L^1(G)$, is closable.
Since $\lambda(L^1(G))$ is dense in $C_r^*(G)$ and $\|\lambda(f)\|\leq \|f\|_1$, $f\in L^1(G)$,
the space $\lambda(D(\psi))$ is dense in $C_r^*(G)$
in the operator norm.
Thus, the operator $S_{\psi} : \lambda(D(\psi)) \rightarrow C_r^*(G)$ given by
$S_{\psi}(\lambda(f)) = \lambda(\psi f)$ is a densely defined operator
on $C_r^*(G)$.

We wish to study the question of when $S_{\psi}$ is closable.
To this end, we recall \cite{eymard} that the Banach space dual of $C^*_r(G)$
can be canonically identified with
the weak* closure $B_{\lambda}(G)$ of $A(G)$ within the Fourier-Stieltjes algebra $B(G)$.
A direct verification shows that the domain of
$S_{\psi}^*$ is equal to
$$J_{\psi}^{\lambda} \stackrel{def}{=}
J_{\psi}^{B_\lambda(G)} = \{g\in B_\lambda(G) : \psi g \in^m B_\lambda(G)\}$$
and that $S_{\psi}^*(g)$ is equivalent to $\psi g$ for every $g\in J_{\psi}^\lambda$.
By Proposition \ref{wstarc}, $S_{\psi}$ is closable (resp. weak* closable) if and only if
$J_{\psi}^\lambda$ is weak* dense (resp. norm dense) in $B_{\lambda}(G)$.
We denote by $\clos(G)$ the set of all measurable functions $\psi$ for which $S_{\psi}$ is closable.

A function $f$ on $G$ is said to \emph{belong to} $A(G)$
(resp. \emph{almost belong to} $A(G)$) \emph{at the point} $t\in G$
if there exists a neighbourhood $U$ of $t$ and a function $u\in A(G)$ such that $f(s) = u(s)$ for all
(resp. $m$-almost all) points $s\in U$.
Set \cite{stt}
$$E_{f} = \{t\in G : f \mbox{ does not almost belong to } A(G) \mbox{ at } t\}.$$
We say that $f$ (almost) belongs locally to $A(G)$ if $f$
(almost) belongs to $A(G)$ at every point and let $A(G)^{\loc}$ be the set of functions which
belong to $A(G)$ at every point. 
If $f$ almost belongs to $A(G)$ at every point then $f$ is equivalent to a function from $A(G)^{\loc}$.
To see this, we first show that, given a compact set $K\subseteq G$ and a function
$f$ that almost belong to $A(G)$ at each point of $G$, there exists  $g\in A(G)$ such that $f$ is equivalent
to $g$ on $K$. In fact, for each $t\in G$ there exists a neighbourhood $V_t$ of $t$ and $g_t\in A(G)$ such that $f\sim g_t$on $V_t$. Then $K\subseteq \cup_{t\in F} V_t$ for some finite $F\subseteq K$.
By the regularity of $A(G)$, there exist $h_t\in A(G)$, $t\in F$,
such that $\sum_{t\in F} h_t(x)=1$ if $x\in K$ and $h_t(x)=0$ if $x\notin V_t$, $t\in F$.
Hence, for almost all $x\in K$, we have
$f(x)=\sum_{t\in F} f(x)h_t(x) = \sum_{t\in F} g_t(x)h_t(x)$,
while $\sum_{t\in F} g_th_t\in A(G)$.
As the group $G$ is $\sigma$-compact we can find compact subsets $K_n\subseteq G$,
$K_n\subseteq K_{n+1}$ such that $G = \cup_{n=1}^\infty K_n$, and a sequence of functions $g_n\in A(G)$ such that $f\sim g_n$ on $K_n$ for any $n$.
As $g_n$ are continuous, we obtain $g_{n+1}=g_n$ on $K_n$.
Define a function $g$ by letting $g(x)=g_n(x)$ if $x\in K_n$. Then $g$ is well-defined, continuous
and $f\sim g$. Clearly, $g$ belongs to $A(G)$ at every point of $G$.

The following fact was established in \cite{stt} in the case $G$ is abelian; its proof,
however, does not use the commutativity of $G$.

\begin{lemma}[\cite{stt}]\label{l_epsi}
For every measurable function $\psi : G\rightarrow \bb{C}$, let
$$J_{\psi} \stackrel{def}{=} J_{\psi}^{A(G)} = \{f\in A(G) : \psi f\in^m A(G)\}.$$
Then $E_{\psi} = \nul J_{\psi}$.
\end{lemma}

We say that a locally compact group $G$ has {\it property {\rm (A)}} if  there exists a net  $(u_i)\subseteq A(G)$ such that for each $g\in B_\lambda(G)$, $u_ig\to g$ in the weak$^*$-topology of $B_\lambda(G)$.
Note that if $(u_i)\subseteq A(G)$ is a net such that $u_i\to 1$  uniformly on compact sets and
$\text{sup}\|u_i\|_{MA(G)}<\infty$ (in particular, if $G$ is weakly amenable)
then $G$ has property (A).
In fact,
for $g\in B_\lambda(G)$ and $f\in C_c(G)$, we have
$$\langle \lambda(f),gu_i-g\rangle=\int_G f(t)g(t)(u_i(t)-1)dt\to 0.$$
It follows from \cite[Proposition~1.2]{de-canniere-haagerup} that, if $u\in A(G)$ then
$\|u\|_{MA(G)}$ coincides with the norm of $u$ as a
multiplier of $B_\lambda(G)$. Thus,
$\|gu_i-g\|_{B(G)}\leq \|u_i\|_{MA(G)}\|g\|_{B(G)}+\|g\|_{B(G)}$.
The statement now follows from the fact that the set of all $\lambda(f)$, $f\in C_c(G)$, is dense in $C_r^*(G)$.

Since $C_r^*(G)^* = B_\lambda(G)$ and $A(G)\subseteq B_\lambda(G)$, the elements of
$C_r^*(G)$ can be identified with functionals on $A(G)$ continuous with respect to
the restriction to $A(G)$ of the weak* topology of $B_\lambda (G)$; this identification
is made in the next proposition.

\begin{proposition}\label{p_chcl}
Let $G$ be a locally compact group with property {\rm (A)}
and $\psi : G\rightarrow \bb{C}$ be a measurable function.
The operator $S_\psi$ is closable if and only if there is no non-zero operator
$T\in C_r^*(G)$ which annihilates $J_\psi$.
In particular,

(i) \ if $E_\psi$ is an $U$-set then $S_\psi$ is closable;

(ii) if $E_\psi$ is an $M_1$-set then $S_\psi$ is not closable.
\end{proposition}
\begin{proof}
Since $A(G)$ is an ideal in $B(G)$, property (A) implies that
the weak* closures of $J_{\psi}$ and $J_\psi^{\lambda}$ in $B_\lambda(G)$ coincide.
The first statement now follows from Proposition \ref{wstarc}.

By Lemma \ref{l_epsi},
$$J(E_\psi)\subseteq \overline{J_\psi}\subseteq I(E_\psi).$$
Parts (i) and (ii) follow from these inclusions and the definitions
of a $U$-set and an $M_1$-set.
\end{proof}

\begin{corollary}\label{c_po}\
Let $G$ be a locally compact group with property {\rm (A)}
and $\psi : G\rightarrow \bb{C}$ be a measurable function.
If $m(E_{\psi}) > 0$ then $S_{\psi}$ is not closable.
\end{corollary}
\begin{proof}
By Remark \ref{r_posme},
$E_{\psi}$ is an $M_1$-set. Now the claim follows from Proposition \ref{p_chcl} (ii).
\end{proof}

Recall from Section \ref{s_prel} that, for a measurable function
$\nph : G\times G \rightarrow \bb{C}$, we let
$S_{\nph}$ be the operator, densely defined on $\cl  K(L^2(G))$,
with domain
$$D(S_{\nph}) = \{T_k\in \cl C_2(H_1,H_2) : \hat{\nph} k\in L^2(G\times G)\}
\subseteq \cl  K(L^2(G)).$$
It was shown in \cite{stt} that
the domain $D(S_{\nph}^*) \subseteq \Gamma(G,G)$ of its adjoint
is given by
$$D(S_{\nph}^*) = \{h\in \Gamma(G,G) : \nph h\in^{m\times m} \Gamma(G,G)\}.$$

\begin{theorem}\label{th_duc}
Let $G$ be a second countable locally compact group
with property (A), $\psi : G\rightarrow \bb{C}$ be a measurable function and $\nph = N(\psi)$.
The following are equivalent:

(i) \ \ the operator $S_{\psi}$ is closable;

(ii) \ the operator $S_{\nph}$ is closable;

(iii) $\cl A \cap D(S_{\nph}^*)^{\perp} = \{0\}$;

(iv) \  $\cl R \cap D(S_{\nph}^*)^{\perp} = \{0\}$.
\end{theorem}
\begin{proof}
(iv)$\Rightarrow$(iii)$\Rightarrow$(ii) follows from the fact that $\cl K\subseteq \cl A\subseteq \cl R$
and the fact that $S_{\nph}$ is closable if and only if $\cl K\cap D(S_{\nph}^*)^\perp = \{0\}$.

(ii)$\Rightarrow$(i) If $S_{\psi}$ is not closable then, by Proposition \ref{p_chcl},
there exists a non-zero $T\in C_r^*(G)$ which annihilates $J_{\psi}$.
Let $A\in \cl D_0$ be such that $AT\neq 0$.
In view of (\ref{eq_k}), it suffices to show that
$AT$ annihilates $D(S_{\nph}^*)$.
Since $D(S_{\nph}^*)$
is invariant under $\frak{S}(G,G)$,
it suffices to show that $T$ annihilates $D(S_{\nph}^*)$.

Let $h\in D(S_{\nph}^*)$.
Writing $h = \sum_{i=1}^{\infty} f_i\otimes g_i$,
where $\sum_{i=1}^{\infty} \|f_i\|_2^2 < \infty$ and $\sum_{i=1}^{\infty} \|g_i\|_2^2 < \infty$,
and using Lemma \ref{l_for}, for every $f\in L^1(G)$, we have
\begin{eqnarray*}
\langle\lambda(f),h\rangle & = &
\left\langle\lambda(f),\sum_{i=1}^{\infty} f_i\otimes g_i\right\rangle
= \sum_{i=1}^{\infty} (\lambda(f)(f_i),\bar g_i)\\
& = & \iint f(s) \sum_{i=1}^{\infty} g_i(t)f_i(s^{-1}t)dt ds
= \langle \lambda(f), P(h)\rangle.
\end{eqnarray*}
It follows that $\langle T,h\rangle = \langle T,P(h)\rangle$.
Since $\nph h \in^{m\times m} \Gamma(G,G)$,
identity (\ref{eq_P}) implies that
$\psi P(h) = P(\nph h)\in^m P(\Gamma(G,G)) = A(G)$,
and hence $P(h) \in J_{\psi}$.
Thus, $\langle T,P(h)\rangle = 0$ and hence
$\langle T, h\rangle = 0$.

(i)$\Rightarrow$(iv)
Let $S_{\psi}$ be closable and suppose that $0\neq T\in \cl R\cap D(S_{\nph}^*)^\perp$.
By Lemma~\ref{eab}, there exist $a,b\in L^2(G)$ such that
$E_{a\otimes b}(T) \neq 0$.
Suppose that $u\in J^{\lambda}_{\psi}$; then
$$\nph(a\otimes b)N(u) = (a\otimes b)N(\psi u)\in \Gamma(G,G)$$
and hence $(a\otimes b)N(u) \in D(S_{\nph}^*)$. Thus
$$\langle E_{a\otimes b}(T),u\rangle = \langle T,(a\otimes b)N(u)\rangle = 0.$$
By Theorem \ref{p_maps},
$E_{a\otimes b}(T)$ is a (non-zero) element of $C_r^*(G)$; in view of
Proposition \ref{p_chcl}, this contradicts the closability of $S_{\psi}$.
\end{proof}

\begin{corollary}\label{p_algec}
The set $\clos(G)$ is an algebra under pointwise addition and multiplication.
\end{corollary}
\begin{proof}
Let $\psi_i\in \clos(G)$, $i = 1,2$. Then $N\psi_1 + N\psi_2 = N(\psi_1 + \psi_2)$
and $N(\psi_1\psi_2) = (N\psi_1)(N\psi_2)$. By \cite[Theorem 5.2]{stt},
the closable multipliers on $\cl K(L^2(G))$ form an algebra under pointwise addition and multiplication.
The claim now follows from Theorem \ref{th_duc}.
\end{proof}

We now give some examples of closable and non-closable multipliers.

\begin{example}\rm {\bf  A non-closable multiplier on $C_r^*(\mathbb T)$.}
Using the arguments in \cite[7.8.3-7.8.6]{rudin} (see also \cite[Proposition~9.9]{st2}),
one can show that there exist $c = (c_n)_{n\in \bb{Z}}\in \ell^p(\mathbb Z)$, $p>2$,
and $d = (d_n)_{n\in \bb{Z}}\in \ell^1(\mathbb Z)$ with
$\overline{d}_n \stackrel{def}{=} \bar d_{-n}$, $n\in \bb{Z}$, such that
$c\ast d=0$ and $c\ast\overline{d}\ne 0$.
Let  $f\in  A(\mathbb T)\subseteq L^1(\mathbb T)$
be the function whose sequence of Fourier coefficients is $d$ and $F$ be the
pseudo-function (that is, the bounded linear functional on $A(\bb{T})$)
whose sequence of Fourier coefficients is $c$.
We have $f\cdot F=0$ while $\bar f \cdot F\ne 0$.
After identifying the dual of $A(\bb{T})$
with $\vn(\bb{T})$, we view $F$ as the operator
on $L^2(\bb{T})$ determined by the identities $\widehat{F\xi}=c\hat\xi$, $\xi\in L^2(\mathbb T)$
(where $\hat{\eta}$ denotes the Fourier transform of a function $\eta\in L^2(\bb{T})$).
%We have
%$$\|F\xi\|_2=\|c\hat\xi\|_2\leq \|c\|_\infty\|\xi\|_2, \ \ \ \xi\in L^2(\mathbb T).$$
Moreover (see the start of Section 3),
$F\in C^*_r(\mathbb T)$.
Let ${h_n}\in L^1(\mathbb T)$  be such
that $\lambda(h_n)\to_{n\to \infty} F$ in the operator norm.
Then
$$\|\lambda(h_n) - F\|=\sup\{|\langle \lambda(h_n) - F, u\rangle| :  u\in A(\bb{T}), \|u\| = 1\} \to_{n\to \infty}  0.$$
It follows that
$$\sup\{|\langle \lambda(fh_n) - f\cdot F, u\rangle| :  u\in A(\bb{T}), \|u\| = 1\} \to_{n\to \infty} 0$$
which in turn implies that
$\lambda(f h_n)\to f\cdot F$ in the operator norm.
Similarly, $\lambda(\bar f h_n)\to_{n\to \infty} \bar f \cdot F$.

Let $\psi : \bb{T}\to \bb{C}$ be given by
$\psi(t)=\bar f(t)/f(t)$ if $f(t)\ne 0$ and $\psi(t)=0$ otherwise. Then
$$S_\psi(\lambda(fh_n))=\lambda(\psi fh_n)=\lambda(\bar f h_n)\not\to_{n\to \infty} 0$$
while $\lambda(fh_n)\to 0$.
Hence $\psi$ is a non-closable multiplier.
\end{example}

\begin{example}{\bf A continuous non-closable multiplier on $C_r^*(\bb{T})$.}
\rm
The following example was given in \cite{stt}. We recall the construction for
completeness.
Let $X\subseteq \bb{T}$ be a closed set of positive Lebesgue measure and
$S\subseteq X$ be a dense subset of Lebesgue measure zero.
By \cite[Chapter II, Theorem 3.4]{katznelson}, there exists
$h\in C({\mathbb T})$ whose Fourier series diverges at every point of $S$.
By the Riemann Localisation Principle, any function which belongs locally to
$A({\mathbb T})$ at $t\in {\mathbb T}$ has a convergent Fourier series at $t$;
hence, $S\subseteq E_h$ and since $E_h$ is closed, $X\subseteq E_h$.
Therefore $m(E_h) > 0$ and $S_h$ is not closable by Corollary \ref{c_po}.
\end{example}

\begin{example}
{\bf A class of idempotent closable multipliers on $C^*_r(\bb{R})$.}
\rm
Let $F\subseteq \bb{R}$ be a closed set which is
the union of countably many intervals. We claim that $\chi_F\in \clos(\bb{R})$.
Let $\psi = \chi_F$; then $E_{\psi}$ is the set of boundary points of $F$.
Thus $E_{\psi}$ is contained in the set of endpoints of the intervals
whose unions is $F$, and hence $E_{\psi}$ is countable.
The claim now follows from Proposition \ref{p_chcl} and Corollary \ref{c_count}.

This example should be compared with the well-known fact that there are no
bounded non-trivial idempotent multipliers on $C^*_r(\bb{R})$.
\end{example}

We next discuss the weak** closability of the operator $S_{\psi}$ (in the sense of Section \ref{sub_cl}).
We have the following necessary condition.

\begin{proposition}\label{p_nece}
If $S_\psi$ is weak** closable then $\psi\in A(G)^{\loc}$.
\end{proposition}
\begin{proof}
Suppose that $S_{\psi}$ is weak** closable. By Proposition \ref{wstarc},
$J_{\psi}^{\lambda}$ is dense in $B_\lambda(G)$.
Thus, $A(G)J_{\psi}^{\lambda}$ is dense in $A(G)B_\lambda(G) = A(G)$. However,
$A(G)J_{\psi}^{\lambda}\subseteq J_{\psi}$ and hence
$J_{\psi}$ is dense in $A(G)$.
By Lemma \ref{l_epsi}, $\psi\in A(G)^{\loc}$.
\end{proof}

We point out that the converse of Proposition \ref{p_nece} does not hold
for non-compact groups. In fact, let
$G$ be a non-discrete locally compact abelian group with dual group $\Gamma$.
Then $B_\lambda(\Gamma)=B(\Gamma)\ne A(\Gamma)$.
By \cite[Corollary 8.2.6]{gmcgehee}, there exists $f\in B(\Gamma)$, $|f(x)|>1$,
$x\in\Gamma$, such that $\psi\stackrel{def}{=}\frac{1}{f}\notin B(\Gamma)$; on the other hand,
 $\psi \in A(\Gamma)^{\loc}$ (see the arguments in \cite[Remark~7.11]{stt}). We have that $J_{\psi}^{\lambda}\subseteq
(f)$, where $(f)$ is the ideal in $B({\Gamma})$ generated by $f$.
As $f$ is not invertible in $B(\Gamma)$, $(f)$  is contained in a maximal
ideal, and hence can not be dense in $B(\Gamma)$.

It follows that a version of Theorem \ref{th_duc}, with weak** closability
in the place of closability, does not hold. Indeed, by \cite[Theorem 7.8]{stt},
for abelian groups,  $N(\psi)$ is a weak** closable multiplier
if and only if
$\psi\in A(G)^{\loc}$. In view of these remarks, the following question arises.

\medskip

\noindent {\bf Question.} Is  $S_\psi$  weak** closable only when $S_\psi$ is bounded?

\medskip

Note that if $G$ is compact then $S_\psi$ is weak** closable if and only if $S_\psi$ is
bounded, that is, if and only if $\psi\in A(G)$; this follows from
Proposition \ref{p_nece} and the fact that in this case $A(G) = A(G)^{\loc}$.

\section{Closable multipliers on group von Neumann algebras}\label{s_vna}

In this section we turn our attention to multipliers acting on $\vn(G)$.
We will need an appropriate version of
closability suited for working with dual spaces, which we now introduce.
Let $\cl X$ and $\cl Y$ be dual Banach spaces, with specified preduals $\cl X_*$
and $\cl Y_*$, respectively, and $D(\Phi)\subseteq \cl X$ be
a weak* dense subspace. We say that an operator $\Phi : D(\Phi)\rightarrow \cl Y$
is \emph{weak* closable}
if the conditions $x_i\in \cl X$, $y\in \cl Y$,
$x_i\rightarrow_{w^*} 0$, $\Phi(x_i)\rightarrow_{w^*} y$ imply that
$y = 0$. Here, the weak* convergence is in the designated weak* topologies of
$\cl X$ and $\cl Y$.

%{\bf Mozhet byt', soderzhanie Remark napisat' odnoj frazoj prosto v tekste:
Note that, since the *-weak closure of the graph of $\Phi$ contains its norm-closure, each weak* closable operator is closable.

\medskip

%\noindent {\bf Remark.} Suppose that $\Phi$ is a weak* closable operator. Then it is
%closable. Indeed, letting
%$x_k\in \cl X$, $y\in \cl Y$ be such that $x_k\rightarrow 0$ and
%$\Phi(x_k)\rightarrow y$ in norm, we have that the convergence are
%also in the weak* topology, and hence $y = 0$.

We have the following characterisation of weak* closability.

\begin{proposition}\label{p_eqpw}
Let $D(\Phi)\subseteq \cl X$ be a weak* dense subspace and
$\Phi : D(\Phi) \rightarrow \cl Y$ be a linear operator.
The following are equivalent:

(i) \ the operator $\Phi$ is weak* closable;

(ii) the space $D_*(\Phi) = \{g\in \cl Y_* : x\rightarrow \langle
\Phi(x),g\rangle \mbox{ is w* -cont. on } D(\Phi)\}$ is dense in $\cl
Y_*$.
\end{proposition}
\begin{proof}
(ii)$\Rightarrow$(i) Suppose that $x_{i}\rightarrow 0$ and $\Phi(x_{i})\rightarrow y$
in the corresponding weak* topologies. If $g\in D_*(\Phi)$ then the map $x\rightarrow \langle \Phi(x),g\rangle$
is weak* continuous on $D(\Phi)$. Since $D(\Phi)$ is weak* dense in $\cl X$,
it extends to a weak* continuous functional on the whole of $\cl X$
and hence there exists $f\in \cl X_*$ such that $\langle \Phi(x),g\rangle = \langle x,f\rangle$, $x\in D(\Phi)$.
In particular, $\langle \Phi(x_{i}),g\rangle = \langle x_{i},f\rangle \rightarrow 0$.
On the other hand,  $\langle \Phi(x_{i}),g\rangle \rightarrow \langle y,g\rangle$. Thus,
$\langle y,g\rangle = 0$ for all $g\in D_*(\Phi)$. Since $D_*(\Phi)$ is (norm) dense in $\cl Y_*$,
we conclude that $y = 0$.

(i)$\Rightarrow$(ii)
For an operator $T$ with domain $D$, let $\Gr' T = \{(T\xi,\xi) : \xi \in D\}$.
Let $\Phi_* : D_*(\Phi)\to \cl X_*$ be defined by letting $\Phi_*(g) = f$, where,
for $g\in D_*(\Phi)$, the element $f\in \cl X_*$ is the (unique) weak* continuous
functional on $\cl X$ such that $\langle\Phi(x),g\rangle = \langle x,f\rangle$, $x\in D_*(\Phi)$.
We claim that
\begin{equation}\label{eq_dd}
(\Gr \Phi)_{\perp} \subseteq \Gr\mbox{}'(-\Phi_*).
\end{equation}
To see this, let $(f,g)\in (\Gr \Phi)_{\perp}$; then
$\langle f,x\rangle = - \langle g,\Phi(x)\rangle$, for all $x\in D(\Phi)$.
It follows that $g\in D(\Phi_*)$ and $\Phi_*(g) = -f$; thus, (\ref{eq_dd}) is proved.

Now suppose that $y\in \cl Y$ annihilates $D_*(\Phi)$.
Then $(0,y)$ annihilates $\Gr'(-\Phi_*)$ and (\ref{eq_dd}) implies that
$$(0,y) \in ((\Gr \Phi)_{\perp})^{\perp} = \overline{\Gr \Phi}^{w^*}.$$
Since $\Phi$ is weak* closable, $y = 0$ and so $D_*(\Phi)$ is norm dense in $\cl Y_*$.
\end{proof}

The von Neumann algebra $\vn(G)$ possesses two natural and, in the case
$G$ is non-discrete, genuinely different,
weak* dense selfadjoint subalgebras, one of them being $\lambda(L^1(G))$,
and the other being the (non-closed) linear span of the left translation operators
$$\vn\mbox{}_0(G) = [\lambda_s : s\in G].$$
Given a continuous function $\psi : G\rightarrow \bb{C}$, we can now consider, along with the
operator $S_{\psi}$ with domain $D(\psi)$, a
linear operator $S'_{\psi} : \vn_0(G)\rightarrow \vn_0(G)$
given by $S'_{\psi}(\lambda_s) = \psi(s)\lambda_s$, $s\in G$.
Our aim in the next theorem is to characterise the weak* closability of $S_{\psi}$ and $S'_{\psi}$.

\begin{theorem}\label{th_vn}
Let $\psi : G\rightarrow \bb{C}$ be a continuous function and $\nph = N(\psi)$.
The following are equivalent:

(i) \ \ the operator $S_{\psi}$ is weak** closable;

(ii) \ the operator $S'_{\psi}$ is weak* closable;

(iii) the function $\psi$ belongs locally to $A(G)$ at every point;

(iv) \ the function $\nph$ is a local Schur multiplier on $\cl K(L^2(G))$;

(v) \ \ the operator $S_{\nph}$ is weak** closable; 

(vi) \  $\overline{D(S_{\nph}^{*})}^{\|\cdot\|_{\Gamma}} = \Gamma(G,G)$,
$\overline{D(S_{\nph}^{**})}^{w^*} = \cl B(L^2(G))$, $\vn_0(G)\subseteq D(S_{\nph}^{**})$ and the operator
$S_{\nph}^{**} : D(S_{\nph}^{**})\rightarrow \cl B(L^2(G))$ is
weak* closable;

(vii) $\overline{D(S_{\nph}^{*})}^{\|\cdot\|_{\Gamma}} = \Gamma(G,G)$,
$\vn_0(G)\subseteq D(S_{\nph}^{**})$ and the operator $S_{\nph}^{**} :
D(S_{\nph}^{**})\rightarrow \cl B(L^2(G))$ is weak* closable.
\end{theorem}
\begin{proof}
We have that
\begin{eqnarray*}
D_*(S'_{\psi}) & = & \{f\in A(G) : T\rightarrow \langle S'_{\psi}(T),f\rangle \mbox{ is w*-continuous on } \vn\mbox{}_0(G)\}\\
& = & \{f\in A(G) : \exists \ u\in A(G): \langle S'_{\psi}(T),f\rangle = \langle T,u\rangle, T\in \vn\mbox{}_0(G)\}\\
& = &
\{f\in A(G) : \exists \ u\in A(G)\mbox{ with }  \langle S'_{\psi}(\lambda_s),f\rangle = \langle \lambda_s,u\rangle, \ s\in G\}\\
& = &
\{f\in A(G) :  \exists \ u\in A(G) \mbox{ with } \psi(s)f(s) = u(s), s\in G\}\\
& = &
\{f\in A(G) :  \psi f \in A(G)\} = J_{\psi},
\end{eqnarray*}
where the last equality follows from the fact that $\psi$ is continuous.
The equivalence (ii)$\Leftrightarrow$(iii) now follows from Lemma \ref{l_epsi} and Proposition \ref{p_eqpw}.

Similarly,
\begin{eqnarray*}
D_*(S_{\psi}) & = & \{f\in A(G) : g\rightarrow \langle S_{\psi}(\lambda(g)),f\rangle \mbox{ is w*-continuous on } D(\psi)\}\\
& = & \{f\in A(G) : \exists \ u\in A(G): \langle \lambda(\psi g), f\rangle = \langle \lambda(g), u\rangle, g\in D(\psi)\}\\
& = &
\{f\in A(G) : \exists \ u\in A(G): \int_G \psi f g = \int_G ug, \  g\in D(\psi)\}\\
& = &
\{f\in A(G) : \exists \ u\in A(G) \mbox{ such that } \psi f \sim u \} = J_{\psi}
\end{eqnarray*}
(recall that by $u \sim v$ we mean that $u=v$ almost everywhere on $G$).
The fourth equality in the latter chain can be seen as follows.
Let $K\subseteq G$ be a compact set; then $\psi|_K$ is bounded and hence
$L^1(K)\subseteq D(\psi)$. It follows that $\int_K \psi fg = \int_K ug$ for all
$g\in L^1(K)$. Since $\psi f|_K$ and $u|_K$ belong to $L^{\infty}(K)$, we conclude that
$\psi f|_K = u|_K$ almost everywhere. Since this holds for every compact $K\subseteq G$,
we have that $\psi f \sim u$.

The equivalence (i)$\Leftrightarrow$(iii) follows, as above, from Lemma \ref{l_epsi} and Proposition \ref{p_eqpw}.

(iii)$\Rightarrow$(iv)
We claim that $\psi u\in A(G)$ for every $u\in A(G)\cap C_c(G)$.
Indeed, since $\psi\in A(G)^{\loc}$, for every $t\in G$
there exists a neighborhood $V_t$ of $t$ and a function
$g_t\in A(G)$ such that $\psi = g_t$ on $V_t$.
Since $\text{supp}(u)$ is compact  there exists a finite set $F\subseteq G$ such that
$\text{supp}(u)\subseteq \cup_{t\in F}V_t$. It follows from the regularity of $A(G)$ that
there exist $h_t\in A(G)$, $t\in F$, such that $\sum_{t\in F}h_t(x)=1$ if $x\in\text{supp}(u)$ and
$h_s(x)=0$ if $x\notin  V_s$ for each $s\in F$ (see the proof of \cite[Theorem 39.21]{hewittI}).
Then for every $x\in G$ we have
$$\psi(x) u(x) =\sum_{t\in F}\psi(x)h_t(x) u(x) = \sum_{t\in F} g_t(x)h_t(x) u(x),$$
which gives $\psi u\in A(G)$.

Let $(K_n)_{n=1}^{\infty}$
be an increasing sequence
of compact sets such that, up to a null set,
$\cup_{n=1}^{\infty} K_n = G$.
Choose, for each $n\in \bb{N}$, a function  $\psi_n\in A(G)\cap C_c(G)$
that takes the value $1$ on $K_n K_n^{-1}$.
By the previous paragraph, $\psi \psi_n\in A(G)$ and therefore $N(\psi \psi_n)$ is a
Schur multiplier. Thus, for each $h\in \Gamma(G,G)$, we have
$$\nph \chi_{K_n\times K_n}h = N(\psi \psi_n)\chi_{K_n\times K_n}h\in \Gamma(G,G).$$
It follows that $\nph|_{K_n\times K_n}$ is a Schur multiplier and hence
$\nph$ is a local Schur multiplier.

(iv)$\Rightarrow$(v) follows from the fact that every local Schur multiplier is a weak* closable
multiplier \cite{stt}.

(v)$\Rightarrow$(vi) Suppose that $S_{\nph}$ is weak** closable.
By Proposition \ref{wstarc}, the space $D(S_{\nph}^*)$ is dense in
$\Gamma(G,G)$ in the norm topology.
We have that
$$D(S_{\nph}^{**}) = \{T\in \cl B(L^2(G)) : h\rightarrow \langle T, S_{\nph}^*(h)\rangle \mbox{ is continuous on }  D(S_{\nph}^*)\}.$$
The space $D(S_{\nph}^{**})$ is weak* dense in $\cl B(L^2(G))$
since it contains the norm dense subspace $D(S_{\nph})$.

Suppose that $h\in D(S_{\nph}^*)$; then $S_{\nph}^*(h) = \nph h\in^{m\times m} \Gamma(G,G)$
and hence, if $T\in D(S_{\nph}^{**})$ then
$$\langle T,\nph h\rangle  = \langle T,S_{\nph}^{*}(h)\rangle = \langle S_{\nph}^{**}(T),h\rangle.$$
The mapping
$$T\rightarrow  \langle S_{\nph}^{**}(T),h\rangle, \ \ \ T\in D(S_{\nph}^{**}),$$
is thus weak* continuous and hence $h\in D_*(S_{\nph}^{**})$. In other words,
$D(S_{\nph}^*) \subseteq D_*(S_{\nph}^{**})$; since
$D(S_{\nph}^*)$ is dense in norm in $\Gamma(G,G)$, the same holds true for
$D_*(S_{\nph}^{**})$. By Proposition \ref{p_eqpw}, $S_{\nph}^{**}$ is weak* closable.

Let $s\in G$. We show that $\lambda_s \in D(S_{\nph}^{**})$. Recall that
$P : \Gamma(G,G)\rightarrow A(G)$ is the canonical
contractive surjection satisfying (\ref{eq_P});
for every $h\in D(S_{\nph}^*)$, using Lemma \ref{l_for}, we see that
\begin{equation}\label{eq_ls}
\langle\lambda_s, S_{\nph}^*(h)\rangle =
\langle \lambda_s,\nph h\rangle = P(\nph h)(s) = \psi(s) P(h)(s) = \langle
\psi(s)\lambda_s,h\rangle.
\end{equation}
Thus, $\lambda_s\in D(S_{\nph}^{**})$,
$S_{\nph}^{**}(\lambda_s) = \psi(s)\lambda_s$, and (vi) is proved.

(vi)$\Rightarrow$(vii) is trivial.

(vii)$\Rightarrow$(ii) Suppose that $(T_i)_i\subseteq
\vn_0(G)$ and $T\in \vn(G)$ are such that $T_i\rightarrow^{w^*} 0$
and $S'_{\psi}(T_i)\rightarrow^{w^*} T$. Then $(T_i)$ (resp.
$(S'_{\psi}(T_i))$) converges to zero (resp. $T$) in the weak*
topology of $\cl B(L^2(G))$.
Identity (\ref{eq_ls}) shows that
$S_{\nph}^{**}(R) = S'_{\psi}(R)$ for every $R\in \vn_0(G)$.
Since $S_{\nph}^{**}$ is weak* closable, $T = 0$.
\end{proof}

\noindent {\bf Remark } If $\psi$ is not assumed to be continuous, then
all conditions in Theorem \ref{th_vn} apart from (ii) remain equivalent, provided that in (iii)
we require that $\psi$ almost belongs locally to $A(G)$ at every point.

\medskip

Proposition \ref{p_nece} and Theorems \ref{th_duc} and \ref{th_vn}
yield the following implications:
$$S_{\psi} \mbox{ is weak** closable } \Longrightarrow S_{\psi} \mbox{ is weak* closable } \Longrightarrow S_{\psi} \mbox{ is closable.}$$
Theorem \ref{th_vn} and the example after Proposition \ref{p_nece}
show that there exists a continuous function $\psi$ for which
$S_{\psi}$ is weak* closable but not weak** closable.
On the other hand, Proposition \ref{p_chcl} implies that if $E_{\psi}$ is a non-empty $U$-set
then $\psi$ is closable but $\psi\not\in A(G)^{\loc}$; thus, by Theorem \ref{th_vn}, $S_{\psi}$
is not weak* closable. For example, for $G={\mathbb R}$, $\psi=\chi_{[0,+\infty)}$,
we have $E_\psi=\{0\}$  which is a non-empty $U$-set by Corollary~\ref{c_count}. One can also find a continuous function $\psi$ for which $E_\psi$ is a one-point set of uniqueness.
In fact, consider a  function $\psi(t)$ on $[0,\pi]$ which is smooth on the open interval $(0,\pi)$ and  $\psi(0)=\psi(\pi)=0$. Assume also that $\psi'(\pi)=0$ and that
the integral $\int_0^1\psi(t)/tdt$ diverges. Extend $\psi$ to an odd (continuous) function on $[-\pi,\pi]$. By \cite[Chapter II.14]{kahane}, $\psi\not\in A({\mathbb T})^{\loc}=A({\mathbb T})$. As $\psi$ is smooth at any $t\ne 0$, $\psi$ belongs to $A(\mathbb T)$ at any such point $t$. Therefore $E_\psi=\{0\}$.


\begin{thebibliography}{99}


\bibitem{akt}
\textsc{ M. Anoussis, A. Katavolos and I.G. Todorov},
{\it Ideals of $A(G)$ and bimodules over maximal abelian selfadjoint algebras},
{\rm preprint}


\bibitem{arveson}
\textsc{ W.B. Arveson}, {\it Operator algebras and invariant subspaces},
{\rm Ann. Math. (2) 100 (1974), 433-532}

%\bibitem{BS}
%{\sc D.P. Blecher and R. Smith}, {\it The dual of the Haagrup
%tensor product}, {\rm J. London Math. Soc. (2) 4 (1992), 126-144}

\bibitem{Blumlinger}
\textsc{M. Bl\"umlinger},
\textit{Characterization of measures in the group $C^*$-algebra of a locally compact group},
\textrm{Math. Ann. 289 (1991), no. 3, 393-402}


\bibitem{bozejko_pams}
\textsc{M. Bozejko},
\textit{Sets of uniqueness on noncommutative locally compact groups},
\textrm{Proc. Amer. Math. Soc. 64 (1977), no. 1, 93-96}



\bibitem{bf} \textsc{M. Bo\.{z}ejko and G. Fendler},
{\it Herz-Schur multipliers and completely bounded multipliers of
the Fourier algebra of a locally compact group},
\textrm{Boll. Un. Mat. Ital. A (6) 2 (1984), no. 2, 297-302}

\bibitem{de-canniere-haagerup}
\textsc{J. De Canni\`{e}re and U. Haagerup}, \textit{Multipliers of
the Fourier algebras of some simple Lie groups and their discrete
subgroups}, \textrm{Amer. J. Math. 107 (1985), no. 2, 455--500}


\bibitem{di}
{\sc J. Dixmier}, {\it Les C*-alg\'{e}bres et leurs repr$\acute{e}$sentations},
{\rm \'{\rm E}ditions Jacques Gabay, Paris, 1996}

\bibitem{george} {\sc G.K. Eleftherakis}, {\it Morita equivalence of masa-bimodules}, {\rm preprint}.

\bibitem{et}
{\sc G.K. Eleftherakis and I.G. Todorov}, {\it Ranges of bimodule projections and reflexivity},
\textrm{J. Funct. Anal. 262 (2012), no. 11, 4891-4915}


\bibitem{et2}
{\sc G.K. Eleftherakis and I.G. Todorov}, {\it Operator synthesis and tensor products},
\textrm{preprint, arXiv:1301.3640}



\bibitem{eks}
\textsc{ J.A. Erdos, A. Katavolos and V.S. Shulman},
{\it Rank one subspaces of bimodules over maximal abelian selfadjoint algebras},
{\rm J. Funct. Anal. 157 (1998), no.2, 554-587}


\bibitem{eymard}
{\sc P. Eymard}, {\it L'alg$\grave{e}$bre de Fourier d'un groupe
localement compact}, \textrm{Bulletin de la S.M.F. 92 (1964),
181-236}


\bibitem{f}  {\sc J. Froelich}, {\it Compact operators, invariant
subspaces and spectral synthesis}, \textrm{J. Funct. Anal. 81 (1988), 1-37}


\bibitem{gmcgehee}
{\sc C.C. Graham and O.C. McGehee}, {\it Essays in commutative harmonic analysis},
{\rm Springer-Verlag, Berlin-New York, 1979}


\bibitem{herz}{\sc C.S. Herz}, {\it The spectral theory of bounded functions},
{\rm Trans. Amer. Math. Soc. 94 (1960), 181-232}


\bibitem{hewittI}  {\sc E. Hewitt and K.A. Ross},  {\it Abstract harmonic analysis. Vol. I. Structure of topological groups, integration theory, group representations},
{\rm Springer-Verlag, Berlin-New York, 1979}

%\bibitem{hewitt-rossII} {\sc E. Hewitt and K.A. Ross}, {\it Abstract harmonic analysis. Vol. II: Structure and analysis for compact groups. Analysis on locally compact abelian groups},
%{\rm Springer-Verlag, New York-Berlin, 1970}

\bibitem{j}
{\sc P. Jolissaint},
{\it A characterisation of completely bounded multipliers of Fourier algebras},
{\rm Colloquium Math. 63 (1992), 311-313}

\bibitem{kp}
\textsc{A. Katavolos and V. Paulsen}, \textit{On the ranges of
bimodule projections}, \textrm{Canad. Math. Bull. 48 (2005), no. 1, 97-111}

\bibitem{kahane} {\sc J.P. Kahane}, {\it S\'{e}ries de Fourier absolument convergentes},
{\rm Springer-Verlag, Berlin, 1970}

\bibitem{katznelson} {\sc Y. Katznelson, } {\it An introduction to harmonic analysis}, {\rm Cambridge University Press,
Cambridge, 2004}

\bibitem{kl}
{\sc A.S. Kechris and A. Louveau},
{\it Descriptive set theory and the structure of sets of uniqueness},
{\rm Cambridge University Press, Cambridge, 1987}

\bibitem{lt}
{\sc J. Ludwig and L. Turowska}, {\it On the connection between sets of
operator synthesis and sets of spectral synthesis for locally compact groups},
{\rm J. Funct. Anal. 233 (2006), 206-227}


\bibitem{kato} { \sc T. Kato},
{\it Perturbation theory for linear operators},
{\rm Springer-Verlag, Berlin, 1995}

\bibitem{m}
{\sc M. McGarvey, L. Oliveira and I.G. Todorov},
{\it Normalisers of operator algebras and tensor product formulas},
{\rm Rev. Math. Iber., to appear}


\bibitem{paulsen}
{\sc V. Paulsen}, {\it Completely bounded maps and operator
algebras}, {\rm Cambridge University Press, Cambridge, 2002}

\bibitem{ped}
{\sc G. Pedersen}, {\it C*-algebras and their automorphism groups},
{\rm Academic Press, London-New York-San Francisco, 1979}


\bibitem{peller} {V.V. Peller,}
{\it Hankel operators in the theory of perturbations of unitary and
selfadjoint operators}, {\rm Funct. Anal. Appl. 19 (1985), no.
2, 111-123}


\bibitem{pisier}
{\sc G. Pisier}, {\it Introduction to operator space theory},
{\rm Cambridge University Press, Cambridge, 2003}


\bibitem{roe}
{\sc J. Roe},
{\it Lectures on coarse geometry},
{\rm University Lecture Series 31, American Mathematical Society, Providence, 2003}


\bibitem{rudin}
{\sc W. Rudin},
{\it Fourier analysis on groups},
{\rm  Interscience Publishers, New York-London, 1962}

\bibitem{st1}
{\sc V.S. Shulman and L. Turowska},
{\it Operator synthesis I. Synthetic sets, bilattices and tensor algebras},
{\rm J. Funct. Anal. 209 (2004), 293-331}

\bibitem{st2}
{\sc V.S. Shulman and L. Turowska},
{\it Operator synthesis II. Individual synthesis and linear operator equations},
{\rm  J. Reine Angew. Math.  590  (2006), 143-187}

\bibitem{stt}
{\sc V.S. Shulman, I.G. Todorov and L. Turowska},
{\it Closable multipliers},
{\rm  Integral Eq. Operator Th.  69  (2011),  no. 1, 29-62}

%\bibitem{smith}
%\textsc{R. R. Smith},
%\textit{Completely bounded module maps and the Haagerup tensor product},
%\textrm{J. Funct. Anal. 102 (1991), 156-175}

\bibitem{sourour}
{\sc A. R. Sourour},
{\it Pseudo-integral operators},
{\rm Trans. Amer. Math. Soc. 253 (1979), 339-363}

\bibitem{spronk} {\sc N. Spronk}, {\it Measurable Schur multipliers and completely bounded multipliers
of the Fourier algebras}, {\rm Proc. London Math. Soc (3) 89 (2004), 161-192}

\bibitem{st} {\sc N. Spronk and L. Turowska}, {\it Spectral synthesis and operator synthesis
for compact groups},
{\rm J. London Math. Soc. (2) 66 (2002), 361-376}

\bibitem{t_spsyn}
{\sc I.G. Todorov}, {\it Spectral synthesis and masa-bimodules},
{\rm J. London Math. Soc. (2) 65 (2002), no. 3, 733-744}

\bibitem{williams}
{\sc D. Williams}, {\it Crossed products of C*-algebras}, {\rm American Mathematical Society, Providence, 2007}
\end{thebibliography}
\end{document}